\documentclass[12pt]{article}
\usepackage{amsmath}
\usepackage{graphicx,psfrag,epsf}
\usepackage{enumerate}
\usepackage{natbib}
\usepackage{url}

\usepackage[margin = 2.3cm]{geometry}

\usepackage{amssymb, amsthm, ascmac, authblk, bm, booktabs, caption, comment, enumerate, latexsym, lscape, longtable, mathrsfs, mathtools, multirow, natbib, newtxtext, psfrag, setspace, subcaption, thmtools}
\usepackage[stable]{footmisc}

\usepackage{enumitem}

\usepackage{mathtools}
\usepackage{mathrsfs}
\usepackage{color}
\usepackage{pgfplots}
\pgfplotsset{compat=1.16}
\usepackage{here}
\usepackage{multirow}
\usepackage{caption}
\usepackage[subrefformat=parens]{subcaption}
\usepackage{authblk}
\usepackage[colorlinks=true, citecolor=teal, linkcolor=teal, urlcolor=teal, backref=page]{hyperref}
\usepackage{threeparttable}
\usepackage{comment}
\usepackage{stackengine}
\usepackage{tikz}
\usetikzlibrary{arrows.meta}
\usetikzlibrary{bending}

\usepackage{lineno}

\usepackage[capitalize]{cleveref}

\newtheorem{theorem}{Theorem}
\newtheorem{lemma}{Lemma}
\newtheorem{proposition}{Proposition}
\newtheorem{corollary}{Corollary}
\newtheorem{assumption}{Assumption}

\theoremstyle{definition}
\newtheorem*{rmk*}{Remark}
\newtheorem{rmk}{Remark}

\DeclareMathOperator{\Var}{Var}
\DeclareMathOperator{\Cov}{Cov}
\DeclareMathOperator{\E}{\mathbb{E}}

\DeclareMathOperator{\trace}{tr}

\def\be#1{\begin{equation*}#1\end{equation*}}
\def\ben#1{\begin{equation}#1\end{equation}}

\def\besn#1{\begin{equation}\begin{split}#1\end{split}\end{equation}}

\def\bm#1{\begin{multline*}#1\end{multline*}}
\def\bmn#1{\begin{multline}#1\end{multline}}
\def\ba#1{\begin{align*}#1\end{align*}}
\def\ban#1{\begin{align}#1\end{align}}

\newcommand{\ol}[1]{\overline{#1}}
\newcommand{\ul}[1]{\underline{#1}}

\newcommand{\wt}[1]{\widetilde{#1}}

\newcommand{\mcl}[1]{\mathcal{#1}}
\newcommand{\norm}[1]{\left\|#1\right\|}
\newcommand{\bra}[1]{\left(#1\right)}
\newcommand{\cbra}[1]{\left\{#1\right\}}
\newcommand{\sbra}[1]{\left[#1\right]}
\newcommand{\abs}[1]{\left|#1\right|}

\newcommand{\fourier}[1]{\mathfrak{F}#1}

\newcommand{\eps}{\varepsilon}
\newcommand{\pr}{\mathbb{P}} 
\newcommand{\symg}{\mathbb S} 
\newcommand{\pdf}{\mathcal{P}} 

\allowdisplaybreaks

\begin{document}
\baselineskip=18pt

\renewcommand\thmcontinues[1]{Continued}
\onehalfspacing

\title{\bf Gaussian Approximation for High-Dimensional $U$-statistics with Size-Dependent Kernels}
\author[1]{Shunsuke Imai \thanks{The first author's research was supported by JSPS KAKENHI Grant Number 24KJ1472.}}
\author[2]{Yuta Koike \thanks{The second author's research was supported by JST CREST Grant Number JPMJCR2115 and JSPS KAKENHI Grant Numbers JP22H00834, JP22H00889, JP22H01139, JP24K14848.}}

\affil[1]{Graduate School of Economics, Kyoto University}
\affil[2]{Graduate School of Mathematical Sciences, University of Tokyo}

\maketitle

\begin{abstract}
    
    Motivated by small bandwidth asymptotics for kernel-based semiparametric estimators in econometrics, this paper establishes Gaussian approximation results for high-dimensional fixed-order $U$-statistics whose kernels depend on the sample size. 
    Our results allow for a situation where the dominant component of the Hoeffding decomposition is absent or unknown, including cases with known degrees of degeneracy as special forms.
    The obtained error bounds for Gaussian approximations are sharp enough to almost recover the weakest bandwidth condition of small bandwidth asymptotics in the fixed-dimensional setting when applied to a canonical semiparametric estimation problem. 
    We also present an application to an adaptive goodness-of-fit testing and the simultaneous inference on high-dimensional density weighted averaged derivatives, along with discussions about several potential applications. 
    \vspace{2mm}

\end{abstract}

\clearpage

\section{Introduction} \label{sec:introduction}
    $U$-statistics are a general class of estimators and test statistics, and its application covers wide-ranging statistical problems. 
    Among various types of $U$-statistics, we consider a situation where the order is fixed, the kernel possibly depends on the sample size $n$, and the dominant component of the Hoeffding decomposition is absent or unknown. 
    This setting includes cases with known degrees of degeneracy as special forms.
    In this paper, we establish Gaussian approximation results for such $U$-statistics in the high-dimensional setting where the dimensions of $U$-statistics grow as the sample size increases.

    $U$-statistics of our interest arise in many important statistical applications. 
    Examples of degenerate $U$-statistics with $n$-dependent kernels include test statistics for the specification of parametric regression models \citep[e.g.][]{hardle1993comparing,zheng1996consistent} and semiparametric regression model \citep[e.g.][]{fan1996consistent} and for the conditional mean independence \citep[e.g.][]{fan1996consistent}. 
    A notable example of $U$-statistics whose dominant Hoeffding components are absent or unknown appears in \textit{small bandwidth asymptotics} for two-step linear kernel-based semiparametric estimators \citep{cattaneo2014small}.
    This class of estimators can be applied to a variety of specific statistical problems beyond semiparametric inference, as introduced later (Sections \ref{sec:app_ade}, \ref{sec:DWAD} and \ref{sec:discussion}).  
    Among estimators in this class, one noteworthy example is a kernel-based estimator for density-weighted average derivatives (DWADs). 
    This estimator can be applied to statistical inference on a wide range of parameters, including finite-dimensional parameters in single-index models, as well as various marginal parameters motivated by economic theory. 
    Further details and more specific applications are discussed in Section \ref{sec:DWAD} and \ref{sec:discussion}.

     In the framework of small bandwidth asymptotics, the influential functions of estimators do not always have asymptotically linear forms. In terms of second-order $U$-statistics, the Haj\'ek projections of the estimators are not always dominant over the quadratic terms and the distributional approximations are performed based on both linear and quadratic terms. By contrast, the classical semiparametric inference procedures require some restrictions on the tuning parameters and data generating processes so that the influential functions have asymptotically linear forms and the quadratic terms are ignored. 
    Recently, \cite{cattaneo2024higher} have established Edgeworth expansions for the DWAD estimator (standardization and studentization are conducted considering both linear and quadratic terms) and theoretically shown that capturing both linear and quadratic terms gives a higher-order improvement on the accuracy of normal approximation even when the linear term is dominant, as well as the conditions on the tuning parameters and data-generating processes are weaken. 
    Although in the small bandwidth asymptotics of kernel-based non-linear semiparametric two-step estimators, it is known that the linear and bias terms dominate the quadratic term \citep{cattaneo2013generalized,cattaneo2018kernel}, capturing the quadratic term should improve the normal approximation error \cite[cf. the second paragraph of page 3]{cattaneo2024higher}. 

    In modern applications, the number of target parameters of statistical inference can be large, and one might wish to construct simultaneous confidence bands or conduct multiple testing with family-wise error rate or false discovery rate control. 
    Examples of such situations include cases where there are many outcomes, groups, or time points (or combinations thereof), and parameters are estimated separately for each outcome, group, and time point \citep[Section 1.1]{belloni2018high}; where economic theory implies a large number of testable conditions \citep[e.g.,][]{chernozhukov2019inference}; or where one seeks to perform uniform inference over tuning parameters for purposes of adaptive inference and sensitivity analysis/robustness checks \citep{horowitz2001adaptive,armstrong2018simple}. 
    See Sections \ref{sec:gof}, \ref{sec:DWAD} and \ref{sec:discussion} for more specific examples.

    Our Gaussian approximation results for high-dimensional $U$-statistics are broadly applicable to address the above situations and a range of related potential problems.
    To illustrate our developed Gaussian approximation results, we provide a toy example of small bandwidth asymptotics for estimating the average marginal densities of high-dimensional data (Section \ref{sec:app_ade}) and an application to an adaptive goodness-of-fit test against smooth alternatives (Section \ref{sec:gof}) and to the simultaneous inference on many DWADs with empirical data-analysis (Section \ref{sec:DWAD}). 
    The toy example not only provides an illustrative use case of our Gaussian approximation results, but also confirms that the bound on the approximation error is sharp enough to recover the weakest condition of small bandwidth asymptotics in the fixed-dimensional setting \citep{cattaneo2014small}. 
    Beyond the illustration, we make a notable contribution to a goodness-of-fit test of a prespecified distribution, which was recently investigated by \cite{li2024optimality}; see \cref{rmk:gof} for details.   
    See Section \ref{sec:discussion} for other specific examples of potential applications.

\paragraph{Related Literature and Technical Contributions:}
    In this paragraph, we explain theoretically related references and our contributions from a technical perspective.

    As a pioneering contribution, in \cite{CCK13}, \citeauthor*{CCK13} (CCK for short) established a Gaussian approximation result for the maximum of a sum of high-dimensional independent random vectors. 
    Since then, numerous extensions have been proposed in various directions, some of which address $U$-statistics and their generalizations \citep{chen2018gaussian,ChKa19,ChKa20,song2019approximating,song2023stratified,cheng2022gaussian,chiang2023inference,koike2023high}. 
    Nonetheless, existing CCK-type results for $U$-statistics are almost essentially concerned with the non-degenerate case. 
    The exceptions are \cite{ChKa19} and \cite{koike2023high}.
    While the former authors actually consider degenerate $U$-statistics, the focus is on randomized incomplete $U$-statistics which are approximated by linear terms.
    The latter considers essentially degenerate $U$-statistics whose kernels depend on $n$, but focuses on the case of homogeneous sums.
    To the best of our knowledge, high-dimensional Gaussian approximation in our setting has not been established so far.

    On the other hand, in the fixed-dimensional setting, the asymptotic normality of not necessarily Hoeffding non-degenerate $U$-statistics has been established by many authors and various sufficient conditions are known. 
    Among others, \cite{DoPe19} have recently derived an error bound for the normal approximation to a general symmetric $U$-statistic in terms of the so-called contraction kernels using the exchangeable pairs approach in Stein's method; see Theorem 5.2 ibidem and also \cite[Section 3.2]{dobler2023normal}. 
    Although a multivariate variant of their bounds potentially works in situations with growing dimensions, it is far from trivial how fast the dimension can grow with respect to the sample size.  
    
    In order to establish Gaussian approximation results for general symmetric $U$-statistics in high-dimensions, we build on the development of these two strands of literature.
    Specifically, we employ an analogous argument to the proof of Lemma A.1 of \cite{CCKK22} to develop a high-dimensional central limit theorem (CLT) via generalized exchangeable pairs (Theorem \ref{thm:gexch}) and make extensive use of some notions introduced by \cite{DoPe17,DoPe19}, especially contraction kernels and product formulae \citep[cf. Sections 2.5 and 2.6][]{DoPe19}.

    Whereas building upon these previous works, we make our own contributions toward establishing Gaussian approximation results.
    From a technical standpoint, the main contribution of this paper lies in the development of quite sharp maximal inequalities. 
    In particular, we extend Lemmas 8 and 9 of \cite{CCK15} in two directions: To $U$-statistics (Theorems \ref{thm:max-is} and \ref{lem:max-is}) and to martingales and non-negative adapted sequences (Lemmas \ref{max-rosenthal} and \ref{lem:nonneg-ada}).
    These results enable us to make our Gaussian approximation results (Theorem \ref{thm:order-2} and Corollary \ref{coro:order-2}) sharp enough to recover the weakest conditions known under small bandwidth asymptotics.
    Moreover, Theorem \ref{thm:max-is} by itself improves upon an existing maximal inequality (Corollary 5.5 in \citealp{ChKa20}) when applied to the present setting.
    This refinement may also hold in other settings and be of potential independent interest.
    See Remark \ref{rem:comparison_maximal} for details about this point.

    Our first main theorem (Theorem \ref{thm:clt-ustat}) covers a general setting, as it allows for $U$-statistics of arbitrary order $r$ and does not assume prior knowledge of the dominant component in the Hoeffding decomposition.
    However, such generality makes the bound on the Gaussian approximation error considerably complex.       
    To enhance applicability, we provide several additional results alongside a general result (Theorem \ref{thm:clt-ustat}). 
    Specifically, 
    (i) Theorem \ref{thm:order-2} presents a result for $r=2$ with a simple bound that is sufficiently sharp for practical purposes; and (ii) Corollary \ref{coro:order-2} serves a bound expressed in terms of moments of kernels rather than those of Hoeffding projections. This result holds under the same assumption as Theorem \ref{thm:order-2}.

\paragraph{Organization:}
The rest of the paper is organized as follows.
Sections \ref{sec:u-stat-notation} and \ref{subsec:main_results} introduce the formal setup and state the main theoretical results, respectively, and Section \ref{sec:app_ade} illustrates how to apply our results.
In Section \ref{sec:gof},  we apply our main results to a goodness-of-fit test.  
Section \ref{sec:discussion} discusses several concrete examples of potential applications.
\cref{subsec:hdgexch,subsec:maximal_inequality} present the two key building blocks of the proofs of main results, and the proofs of main results are in \cref{subsec:proof_clt_ustat,subsec:proof_coro_clt_ustat,subsec:proof_thm_order-2,subsec:proof_coro_order_2}.
\cref{sec:proof_gof,sec:proof-aux} give the proofs of results for the goodness-of-fit test and auxiliary results, respectively.

\paragraph{General notation and convention:}


For a positive integer $m$, we write $[m]:=\{1,\dots,m\}$. We also set $[0]:=\emptyset$ by convention. 
Given a vector $x\in\mathbb R^p$, its $j$-th component is denoted by $x_j$. 
Also, we set $|x|:=\sqrt{\sum_{j=1}^px_j^2}$ and $\|x\|_\infty:=\max_{j\in[p]}|x_j|$. 
For two vectors $x,y\in\mathbb R^p$, $x\cdot y$ denotes their inner product, i.e.~$x\cdot y=x^\top y$. 
Given a $p\times q$ matrix $A$, its $(j,k)$-th entry is denoted by $A_{jk}$. Also, we set $\|A\|_\infty:=\max_{j\in[p],k\in[q]}|A_{jk}|$. 
For two $p\times p$ matrices $A$ and $B$, $\langle A,B\rangle$ denotes their Frobenius inner product, i.e.~$\langle A,B\rangle=\trace(A^\top B)$. 
$\mcl R_p$ denotes the set of all rectangles in $\mathbb R^p$. 
For a normed space $\mathfrak X$, its norm is denoted by $\|\cdot\|_{\mathfrak X}$. 
We interpret $\max\emptyset$ as 0 unless otherwise stated. 
For two random variables $\xi$ and $\eta$, we write $\xi\lesssim\eta$ or $\eta\gtrsim\xi$ if there exists a \emph{universal} constant $C>0$ such that $\xi\leq C\eta$. 
Given parameters $\theta_1,\dots,\theta_m$, we use $C_{\theta_1,\dots,\theta_m}$ to denote positive constants, which depend only on $\theta_1,\dots,\theta_m$ and may be different in different expressions.


\section{Main results}\label{sec:main}

\subsection{$U$-statistics related notation}\label{sec:u-stat-notation}

Given a probability space $(\Omega,\mcl A,\pr)$, let $X_1,\dots,X_n$ be i.i.d.~random variables taking values in a measurable space $(S,\mcl S)$. 
Write $P$ for the common distribution of $X_i$. 
Given an integer $r\geq1$, we say that a function $\psi:S^r\to\mathbb R$ is \emph{symmetric} if $\psi(x_1,\dots,x_r)=\psi(x_{\sigma(1)},\dots,x_{\sigma(r)})$ for all $x_1,\dots,x_r\in S$ and $\sigma\in\symg_r$, where $\symg_r$ is the symmetric group of degree $r$. 
For an $\mcl S^{\otimes r}$-measurable symmetric function $\psi:S^r\to\mathbb R$, we set
\[
J_r(\psi)=J_{r,X}(\psi):=\sum_{1\leq i_1<\cdots<i_r\leq n}\psi(X_{i_1},\dots,X_{i_r})
=\frac{1}{r!}\sum_{(i_1,\dots,i_r)\in I_{n,r}}\psi(X_{i_1},\dots,X_{i_r}),
\]
where $I_{n,r}:=\{(i_1,\dots,i_r):1\leq i_1,\dots,i_r\leq n,~i_s\neq i_t\text{ for all }s\neq t\}$. 
Following \cite{DoPe19}, we call $J_r(\psi)$ the \emph{$U$-statistic of order $r$, based on $X=(X_i)_{i=1}^n$ and generated by the kernel $\psi$}. 
By convention, we set $J_0(\psi):=\psi$ when $r=0$ ($\psi$ is a constant in this case). 
Note that in statistics, the ``averaged'' version $U_r(\psi):=\binom{n}{r}^{-1}J_r(\psi)$ is usually referred to as a $U$-statistic because it is an \emph{unbiased} estimator for the parameter $\theta:=\E[\psi(X_1,\dots,X_r)]$. 
Since we frequently invoke technical tools developed in \cite{DoPe19}, we choose to work with the unaveraged version as in \cite{DoPe19}. 
Except for this, our notation is basically consistent with \cite{ChKa19,ChKa20}. 

For a symmetric kernel $\psi\in L^1(P^r)$ and $0\leq k\leq r$, we define a function $P^{r-k}\psi:S^k\to\mathbb R$ as
\[
P^{r-k}\psi(x_1,\dots,x_k)=\E[\psi(x_1,\dots,x_k,X_{k+1},\dots,X_{r})],\qquad x_1,\dots,x_k\in S.
\]
We say that $\psi$ is \emph{degenerate} if $P\psi=0$ $P^{r-1}$-a.s. 
We write $\pi_s\psi$ for the \emph{Hoeffding projection} of $\psi$ of order $s$, i.e.
\ben{\label{eq:hoef-proj}
\pi_s\psi(x_1,\dots,x_s)
=\sum_{k=0}^s(-1)^{s-k}\sum_{1\leq i_1<\cdots<i_k\leq s}P^{r-k}\psi(x_{i_1},\dots,x_{i_k}).
}
Note that \cite{DoPe19} use the notation $g_k$ and $\psi_s$ instead of $P^{r-k}\psi$ and $\pi_s\psi$, respectively. 
The \emph{Hoeffding decomposition} of $J_r(\psi)$ is given by
\ben{\label{eq:hoef-decomp}
J_r(\psi)=\sum_{s=0}^r\binom{n-s}{r-s}J_s(\pi_s\psi)
=\E[J_r(\psi)]+\sum_{s=1}^r\binom{n-s}{r-s}J_s(\pi_s\psi).
}
When $\psi\in L^2(P^r)$, the variance of $J_r(\psi)$ is decomposed as (cf.~Eqs.(2.8) and (2.10) in \cite{DoPe19})
\ban{
\Var[J_r(\psi)]
&=\sum_{s=1}^r\binom{n-s}{r-s}^2\Var[J_s(\pi_s\psi)]
=\sum_{s=1}^r\binom{n-s}{r-s}^2\binom{n}{s}\|\pi_s\psi\|_{L^2(P^s)}^2\label{u-anova}\\
&=\binom{n}{r}\sum_{s=1}^r\binom{r}{s}\binom{n-r}{r-s}\Var[P^{r-s}\psi(X_1,\dots,X_s)].\label{u-anova-2}
}

For two symmetric kernels $\psi\in L^2(P^{r}),\varphi\in L^2(P^{r'})$ and two integers $0\leq l\leq s\leq r\wedge r'$, we define the \emph{contraction kernel} $\psi\star_s^l\varphi:S^{r+r'-s-l}\to\mathbb R$ as
\bm{
(\psi\star^l_s\varphi)(y_1,\dots,y_{s-l},u_1,\dots,u_{r-s},v_1,\dots,v_{r'-s})\\
=\E\sbra{\psi(X_1,\dots,X_l,y_1,\dots,y_{s-l},u_1,\dots,u_{r-s})\varphi(X_1,\dots,X_l,y_1,\dots,y_{s-l},v_1,\dots,v_{r'-s})},
}
for every $(y_1,\dots,y_{s-l},u_1,\dots,u_{r-s},v_1,\dots,v_{r'-s})$ belonging to the set $A_0\subset S^{r+r'-s-l}$ such that the random variable in the expectation on the right-hand side is integrable, and we set it equal to zero otherwise. 
By Lemma 2.4(i) in \cite{DoPe19}, $\psi\star^l_s\varphi$ is well-defined in the sense that $P^{r+r'-s-l}(A_0)=0$. 
We refer to \cite[Section 2.5]{DoPe19} for more information on contraction kernels. 

For a function $f:S^r\to\mathbb R$ that is not necessarily symmetric, we set
\[
M(f):=\max_{(i_1,\dots,i_r)\in I_{n,r}}|f(X_{i_1},\dots,X_{i_r})|.
\]
We also define the \emph{symmetrization} of $f$ as the function $\wt f:S^r\to\mathbb R$ defined by
\[
\wt f(x_1,\dots,x_r):=\frac{1}{r!}\sum_{\sigma\in\symg_r}f(x_{\sigma(1)},\dots,x_{\sigma(r)}),\qquad
x_1,\dots,x_r\in S.
\]
$\wt f$ is evidently symmetric. 
Also, by Minkowski's inequality and Fubini's theorem
\ben{\label{sym-lp}
\|\wt f\|_{L^q(P^r)}\leq\|f\|_{L^q(P^r)}
}
for all $q\in[1,\infty]$. 
Moreover, by the triangle inequality,
\ben{\label{sym-M}
M(\wt f)\leq M(f).
}

\subsection{Main results} \label{subsec:main_results}

Let $p$ be a positive integer. We assume $p\geq3$ so that $\log p>1$. 
Also, let $r$ be a positive integer such that $r\leq n/4$. 
For every $j\in[p]$, let $\psi_{j}\in L^4(P^{r})$ be a symmetric kernel of order $r$ such that
$\sigma_j:=\sqrt{\Var[J_r(\psi_j)]}>0$. 
Define
\[
W:=\bra{J_r(\psi_1)-\E[J_r(\psi_1)],\dots,J_r(\psi_p)-\E[J_r(\psi_p)]}^\top.
\]
Our first main result is an explicit error bound on the normal approximation of $P(W\in A)$ uniformly over $A\in\mcl R_p$ for the general order $r$. 
To state the result concisely, we introduce some notation. 
For $a,b\in[r]$, we set
\ba{
\Delta_{1}(a,b)&:=\sum_{s=1}^{a\wedge b}\sum_{l=0}^{s\wedge(a+b-s-1)}\max_{0\leq u\leq a+b-l-s}\Delta_1(a,b;s,l,u),\\
\Delta_2(a)&:=\max_{0\leq s\leq a-1}\Delta_2(a;s),
}
where
\ba{
\Delta_1(a,b;s,l,u)
&:=n^{2r+\frac{l-s-a-b-u}{2}}(\log p)^{\frac{a+b-l-s+u}{2}}
\sqrt{\E\sbra{\max_{j,k\in[p]}\frac{M\bra{P^{a+b-l-s-u}(|\wt{\pi_a\psi_j\star^{l}_s\pi_b\psi_{k}}|^2)}}{\sigma_j^{2}\sigma_k^{2}}}}
}
and
\ba{
\Delta_2(a;s)
:=n^{4r-2a-2s-1}(\log p)^{2(a+s-1)}\E\sbra{\max_{j\in[p]}\frac{M\bra{P^{a-s-1}\bra{|\pi_a\psi_j|^2}}^2}{\sigma_j^4}}.
}
\begin{theorem}\label{thm:clt-ustat}
There exists a constant $C_r$ depending only on $r$ such that
\ben{\label{eq:clt-ustat}
\sup_{A\in\mathcal{R}_p}\left|\pr(W\in A)-\pr(Z\in A)\right|
\leq C_{r}\bra{\sqrt{\max_{a,b\in[r]}\Delta_{1}(a,b)\log^2 p}+\bra{\max_{a\in[ r]}\Delta_2(a)\log^5 p}^{1/4}},
}
where $Z\sim N(0,\Cov[W])$. 
\end{theorem}

Although the right-hand side of \eqref{eq:clt-ustat} consists of explicit analytical quantities of the kernels $\psi_j$, it contains many terms and their evaluations are often cumbersome. 
As we will see below, at least for the case $r=2$, we can drastically reduce the number of terms to be evaluated because most of the components of the first term are dominated by the second term.  
Set
\ba{
\Delta_1^{(0)}&:=n^2\max_{j\in[p]}\frac{\|\pi_2\psi_j\star_1^1\pi_2\psi_j\|_{L^2(P^2)}}{\sigma_j^2},&
\Delta_1^{(1)}&:=n^{\frac{5}{2}}\max_{j,k\in[p]}\frac{\|\pi_1\psi_j\star^{1}_1\pi_2\psi_{k}\|_{L^2(P)}}{\sigma_j\sigma_k},
}
and
\ba{
\Delta_{2,*}^{(1)}(1)
&:=n^5\max_{j\in[p]}\frac{\|\pi_1\psi_{j}\|_{L^4(P)}^4}{\sigma_j^4},&
\Delta_{2,*}^{(2)}(1)&:=n^{4}\E\sbra{\max_{j\in[p]}\frac{M(\pi_1\psi_j)^4}{\sigma_j^4}}\log p,
}
and
\ba{
\Delta_{2,*}^{(1)}(2)&:=n^2\max_{j\in[p]}\frac{\|\pi_2\psi_{j}\|_{L^4(P^2)}^4}{\sigma_j^4}\log^{3}p,&
\Delta_{2,*}^{(2)}(2)&:=n^3\max_{j\in[p]}\frac{\|P(|\pi_2\psi_{j}|^2)\|_{L^{2}(P)}^{2}}{\sigma_j^4}\log^{2}p,
\\
\Delta_{2,*}^{(3)}(2)&:=n\E\sbra{\max_{j\in[p]}\frac{M\bra{P\bra{|\pi_2\psi_j|^4}}}{\sigma_j^4}}\log^4 p,&
\Delta_{2,*}^{(4)}(2)&:=\E\sbra{\max_{j\in[p]}\frac{M\bra{\pi_2\psi_j}^4}{\sigma_j^4}}\log^5p,\\
\Delta_{2,*}^{(5)}(2)&:=n^{2}\E\sbra{\max_{j\in[p]}\frac{M\bra{P(|\pi_2\psi_j|^2)}^2}{\sigma_j^4}}\log^{3} p.
}
The next corollary states the bound in terms of $\Delta_1^{(\ell)}$ $(\ell=0,1)$, $\Delta_2(a)$ $(a=1,2)$ and $\Delta_{2,*}^{(\ell)}(2)$ $(\ell=1,5)$ for the case $r=2$. 
\begin{corollary}\label{coro:clt-ustat}
If $r=2$, there exists a universal constant $C$ such that
\ben{\label{eq:clt-ustat-2}
\sup_{A\in\mathcal{R}_p}\left|\pr(W\in A)-\pr(Z\in A)\right|
\leq C\bra{\sqrt{\Delta_1'}+\cbra{\bra{\Delta_{2}(1)+\Delta_{2}(2)+\Delta_{2,*}^{(1)}(2)}\log^5 p}^{1/4}},
}
where 
\begin{equation*}
    \Delta_1':=
\Delta_1^{(0)}\log^3 p
+\Delta_1^{(1)}\log^{5/2} p
+n^{3/2}\max_{j\in[p]}\frac{\|\pi_1\psi_j\|_{L^2(P)}}{\sigma_j}\bra{\Delta_{2,*}^{(5)}(2)\log^{9}p}^{1/4}.
\end{equation*}
\end{corollary}

In applications to small bandwidth asymptotics, we found that the bound of \cref{thm:clt-ustat}, and hence \cref{coro:clt-ustat}, is not sharp enough to recover the weakest possible condition on the lower bound of bandwidths (see \cref{rmk:ade}). 
This is caused by the second term on the right-hand side of \eqref{eq:clt-ustat} whose derivation relies on a somewhat crude argument similar in nature to a simple Gaussian approximation result of \cite{CCK13} (see Comment 2.5 ibidem). 
For the case $r=2$, we can refine this point, leading to the following result.  
\begin{theorem}\label{thm:order-2}
Suppose that $r=2$ and $\max_{j\in[p]}\|\psi_j\|_{L^q(P^2)}<\infty$ for some $q\in[4,\infty]$. 
Then there exists a universal constant $C$ such that
\ben{\label{eq:order-2}
\sup_{A\in\mathcal{R}_p}\left|\pr(W\in A)-\pr(Z\in A)\right|
\leq C\bra{\sqrt{\Delta_1'}+\cbra{\bra{\Delta_{2,q}(1)+\Delta_{2,q}(2)}\log^5 p}^{1/4}},
}
where 
$\Delta_{2,q}(1):=\Delta_{2,*}^{(1)}(1)+\Delta_{2,q}^{(2)}(1)$ and $\Delta_{2,q}(2):=\sum_{\ell=1}^3\Delta_{2,*}^{(\ell)}(2)+\sum_{\ell=4}^5\Delta_{2,q}^{(\ell)}(2)$ 
with
\ba{
\Delta_{2,q}^{(2)}(1)&:=n^{4+4/q}\norm{\max_{j\in[p]}\frac{|\pi_1\psi_{j}|}{\sigma_j}}_{L^{q}(P)}^4\log p,\\
\Delta_{2,q}^{(4)}(2)&:=n^{8/q}\norm{\max_{j\in[p]}\frac{|\pi_2\psi_{j}|}{\sigma_j}}_{L^q(P^2)}^4\log^5(np),\\
\Delta_{2,q}^{(5)}(2)&:=n^{2+4/q}\norm{\max_{j\in[p]}\frac{P(|\pi_2\psi_{j}|^2)}{\sigma_j^2}}_{L^{q/2}(P)}^2\log^{3} (np).
}
Here, we interpret $1/q$ as 0 when $q=\infty$. 
\end{theorem}

\begin{rmk}[Sufficient conditions for convergence of the bound]\label{rmk:order-2}
(i) Since $n^{3/2}\max_{j\in[p]}\|\pi_1\psi_j\|_{L^2(P)}/\sigma_j\lesssim1$ by \eqref{u-anova} and $\Delta_{2,*}^{(5)}(2)\leq\Delta_{2,q}^{(5)}(2)$ by \eqref{eq:max-out}, the right-hand side of \eqref{eq:order-2} converges to 0 once we verify the following conditions:
\ba{
\Delta_1^{(0)}\log^3p+\Delta_1^{(1)}\log^{5/2}p\to0,\\
\Delta_{2,q}(1)\log^5p+\sum_{\ell=1}^3\Delta_{2,*}^{(\ell)}(2)\log^5p+\Delta_{2,q}^{(4)}(2)\log^5p\to0,\\
\Delta_{2,q}^{(5)}(2)\log^{9}p\to0.
}
Moreover, since Lemma 2.4(vi) in \cite{DoPe19} gives
\[
\Delta_1^{(1)}\leq n^{3/2}\max_{j\in[p]}\frac{\|\pi_1\psi_j\|_{L^2(P)}}{\sigma_j}\sqrt{\Delta_1^{(0)}},
\]
the first condition can be replaced by the condition $\Delta_1^{(0)}\log^5p\to0$. 

\noindent(ii) If $\psi_j$ are all degenerate, then $\Delta_1'=\Delta_1^{(0)}\log^3p$ and $\Delta_{2,q}(1)=0$, so it suffices to verify
\[
\Delta_1^{(0)}\log^3p\to0\quad\text{and}\quad\Delta_{2,q}(2)\log^5p\to0.
\]
\end{rmk}

\begin{rmk}[Sub-Weibull case]
Since the constant $C$ in \cref{thm:order-2} does not depend on $q$, we can derive an adequate bound for the case of sub-Weibull kernels from \eqref{eq:order-2} with $q=\log n$ (cf.~Lemma A.6 in \cite{koike2023high}). 
The same remark applies to the next corollary. 
We omit the details. 
\end{rmk}

In applications, it is often convenient to directly work with kernels rather than their Hoeffding projections. 
The following corollary is useful for this purpose. 
\begin{corollary}\label{coro:order-2}
Under the assumptions of \cref{thm:order-2}, there exists a universal constant $C$ such that
\ben{\label{eq:coro:order-2}
\sup_{A\in\mathcal{R}_p}\left|\pr(W\in A)-\pr(Z\in A)\right|
\leq C\bra{\sqrt{\tilde\Delta_1'}+\cbra{\bra{\tilde\Delta_{2,q}(1)+\tilde\Delta_{2,q}(2)}\log^5 p}^{1/4}},
}
where
\ba{
\tilde\Delta_1'&:=
\tilde\Delta_1^{(0)}\log^3 p
+\tilde\Delta_1^{(1)}\log^{5/2}p
+n^{3/2}\max_{j\in[p]}\frac{\sqrt{\Var[P\psi_j(X_1)]}}{\sigma_j}\bra{\tilde\Delta_{2,q}^{(5)}(2)\log^{9}p}^{1/4},\\
\tilde\Delta_{2,q}(1)&:=\tilde\Delta_{2,*}^{(1)}(1)+\tilde\Delta_{2,q}^{(2)}(1),\qquad
\tilde\Delta_{2,q}(2):=\sum_{\ell=1}^3\tilde\Delta_{2,*}^{(\ell)}(2)+\sum_{\ell=4}^5\tilde\Delta_{2,q}^{(\ell)}(2),
}
with
\ba{
\tilde\Delta_1^{(0)}&:=n^2\max_{j\in[p]}\frac{\|\psi_j\star_1^1\psi_j\|_{L^2(P^2)}}{\sigma_j^2},&
\tilde\Delta_1^{(1)}
&:=n^{3/2}\max_{j\in[p]}\frac{\sqrt{\Var[P\psi_j(X_1)]}}{\sigma_j}\sqrt{\tilde\Delta_1^{(0)}},
}
and
\ba{
\tilde\Delta_{2,*}^{(1)}(1)
&:=n^5\max_{j\in[p]}\frac{\|P\psi_{j}\|_{L^4(P)}^4}{\sigma_j^4},&
\tilde\Delta_{2,q}^{(2)}(1)&:=n^{4+4/q}\norm{\max_{j\in[p]}\frac{|P\psi_{j}|}{\sigma_j}}_{L^{q}(P)}^4\log p,
}
and
\ba{
\tilde\Delta_{2,*}^{(1)}(2)&:=n^2\max_{j\in[p]}\frac{\|\psi_{j}\|_{L^4(P^2)}^4}{\sigma_j^4}\log^{3}p,&
\tilde\Delta_{2,*}^{(2)}(2)&:=n^3\max_{j\in[p]}\frac{\|P(\psi_{j}^2)\|_{L^{2}(P)}^{2}}{\sigma_j^4}\log^{2}p,\\
\tilde\Delta_{2,*}^{(3)}(2)&:=n\E\sbra{\max_{j\in[p]}\frac{M(P(\psi_{j}^4))}{\sigma_j^4}}\log^4 p,&
\tilde\Delta_{2,q}^{(4)}(2)&:=n^{8/q}\norm{\max_{j\in[p]}\frac{|\psi_{j}|}{\sigma_j}}_{L^q(P^2)}^4\log^5(np),\\
\tilde\Delta_{2,q}^{(5)}(2)&:=n^{2+4/q}\norm{\max_{j\in[p]}\frac{P(\psi_{j}^2)}{\sigma_j^2}}_{L^{q/2}(P)}^2\log^{3} (np).
}
\end{corollary}

\begin{rmk}[Sufficient conditions for convergence of the bound]\label{rmk:coro:order-2}
Since  $n^{3/2}\max_{j\in[p]}\sqrt{\Var[P\psi_j(X_1)]}/\sigma_j\lesssim1$ by \eqref{u-anova-2}, the right-hand side of \eqref{eq:coro:order-2} converges to 0 once we verify the following conditions:
\ba{
\tilde\Delta_1^{(0)}\log^5p\to0,\\
\tilde\Delta_{2,q}(1)\log^5p+\sum_{\ell=1}^3\tilde\Delta_{2,*}^{(\ell)}(2)\log^5p+\tilde\Delta_{2,q}^{(4)}(2)\log^5p\to0,\\
\tilde\Delta_{2,q}^{(5)}(2)\log^{9}p\to0.
}
\end{rmk}

\subsection{An illustration: Estimation of the average marginal densities} \label{sec:app_ade}

In this subsection, as an illustration of our developed Gaussian approximation results, we consider estimation of the average marginal densities of high-dimensional data. 
Notably, it turns out that our condition does not require the estimator to be asymptotically linear, and the lower bound condition on the bandwidth coincides with, up to a logarithmic factor, the weakest condition to ensure that the variance of the estimator converges to $0$ as $n\to\infty$.
This indicates that our high-dimensional Gaussian approximation holds under nearly the same conditions as small bandwidth asymptotics in the fixed-dimensional setting (cf. Theorem 1 in \citealp{cattaneo2014small}).

Let $X_1,\dots,X_n$ be i.i.d.~random vectors in $\mathbb R^p$ with common law $P$. 
We consider a high-dimensional setting such that $p=p_n\to\infty$ as $n\to\infty$. Note that this means that quantities related to $P$ possibly depend on $n$, although we omit this dependence for notational simplicity. 

For $i\in[n]$ and $j\in[p]$, we denote by $X_{ij}$ the $j$-th component of $X_i$. Assume that the law of $X_{ij}$ has an unknown density $f_j\in L^2(\mathbb R)$. Note that we do not assume that $P$ itself has density. 
We are interested in estimating the vector of the average marginal densities $\theta=(\theta_1,\dots,\theta_p)^\top$ with $\theta_j:=\E[f_j(X_{1j})]$. 
According to \cite{cattaneo2022average} (cf.~the first paragraph of page 1142), estimation of the average density is often viewed as a prototype of two-step semiparametric estimation in econometrics, so this would serve as illustrating how our theory works in such applications. 
We also remark that $\theta_j$ is equal to the integrated square of density $\int_{\mathbb R}f_j(t)^2dt$ and its estimation has been extensively studied in mathematical statistics; see e.g.~\cite{gine2008simple} and references therein. 

Following \cite{gine2008simple} and \cite{cattaneo2022average}, we consider the kernel-based leave-one-out cross-validation estimator for $\theta_j$:
\[
\hat\theta_{n,j}:=\frac{1}{\binom{n}{2}h_n}\sum_{1\leq i<k\leq n}K\bra{\frac{X_{ij}-X_{kj}}{h_n}},
\]
where $K:\mathbb R\to\mathbb R$ is a (fixed) kernel function and $h_n>0$ is a bandwidth parameter converging to $0$.
Following \cite{gine2008simple}, we impose the following conditions on the kernel:
\begin{assumption}[Kernel]
$K$ is bounded and symmetric. In addition,
\[
\int_{\mathbb R}K(u)du=1\quad\text{and}\quad\int_{\mathbb R}|uK(u)|du<\infty.
\]
\end{assumption}
Note that this condition particularly implies that for any $\gamma\in[0,1]$,
\[
\int_{\mathbb R}|u|^\gamma|K(u)|du\leq\|K\|_{L^\infty(\mathbb R)}+\int_{\mathbb R}|uK(u)|du<\infty.
\]
For the marginal densities $f_j$, we assume that they are bounded and contained in a Sobolev space as in \cite{gine2008simple}. 
Formally, for $f\in L^2(\mathbb R^d)$ and $\alpha>0$, we define
\[
\|f\|_{H^\alpha}:=\sqrt{\int_{\mathbb R^d}|\fourier{f}(\lambda)|^2(1+|\lambda|^2)^\alpha d\lambda},
\]
where $\fourier{f}$ denotes the Fourier transform of $f$; when $f\in L^1(\mathbb R^d)$, it is defined as
\[
\fourier{f}(\lambda)=\frac{1}{(2\pi)^{d/2}}\int_{\mathbb R^d}f(x)e^{-\sqrt{-1}\lambda\cdot x}dx,\quad \lambda\in\mathbb R^d,
\]
and we continuously extend it to $L^2(\mathbb R^d)$. 
See e.g.~\cite{Rudin1991} for details of these concepts. 
The only properties of the Sobolev space we need in this section are \cref{lem:sobolev} and Eq.\eqref{eq:bias} from \cite{gine2008simple} below. 

We impose the following conditions on the marginal densities. 
\begin{assumption}[Marginal densities]\label{ass:mar-dens}

    (i) There exist constants $R>0$ and $0<\alpha\leq1/2$ such that $\|f_j\|_{L^\infty(\mathbb R)}+\|f_j\|_{H^\alpha}\leq R$ for all $j\in[p]$.
    
    \noindent(ii) There exists a constant $b>0$ such that $\E[f_j(X_{1j})]\geq b$ and $\Var[f_j(X_{1j})]\geq b^2$ for all $j\in[p]$. 
\end{assumption}
Under \cref{ass:mar-dens}(i), Theorem 1 in \cite{gine2008simple} gives the following estimate of the bias:
\ben{\label{eq:bias}
\|\E[\hat\theta_n]-\theta\|_\infty=O(h_n^{2\alpha}).
}
\cref{ass:mar-dens}(ii) ensures that both the first- and second-order Hoeffding projections of $\psi_j$ defined below have non-zero asymptotic variances. 
Although it is presumably possible to modify the following arguments to remove this condition, we work with it for simplicity.  

We apply \cref{coro:order-2} to $\hat\theta_n:=(\hat\theta_{n,1},\dots,\hat\theta_{n,p})^\top$ with $q=\infty$.  
The corresponding kernels are
\[
\psi_j(x,y)=\frac{1}{h_n\binom{n}{2}}K\bra{\frac{x_{j}-y_{j}}{h_n}},\quad x,y\in\mathbb R^p.
\]
We begin by evaluating the order of the variances $\sigma_{j}^2:=\Var[\hat\theta_{j}]$, $j\in[p]$. 
Recall that by \eqref{u-anova-2},
\[
\sigma_{j}^2
=n(n-1)(n-2)\Var[P\psi_j(X_1)]
+\frac{n(n-1)}{2}\Var[\psi_j(X_1,X_2)].
\]
Let us evaluate the first term. Observe that for any integer $m\geq1$,
\ben{\label{ade-int1}
P(\psi_j^m)(x)
=\frac{1}{h_n^{m-1}\binom{n}{2}^m}\int_{\mathbb R}K(u)^mf_j(x_j+uh_n)du.
}
In particular, we have $\max_{j\in[p]}\|\binom{n}{2}P\psi_j(X_1)-f_j(X_{1j})\|_{L^2({P})}\to0$ as $n\to\infty$. 
In fact,
\ba{
\norm{\binom{n}{2}P\psi_j(X_1)-f_j(X_{1j})}_{L^2({ P})}^2
&=\int_{\mathbb R}\abs{\int_{\mathbb R}K(u)\{f_j(x_j+uh_n)-f_j(x_j)\}du}^2f_j(x_j)dx_j\\
&\leq R\|K\|_{L^1(\mathbb R)}\int_{\mathbb R}\int_{\mathbb R}|K(u)|\{f_j(x_j+uh_n)-f_j(x_j)\}^2dudx_j\\
&\leq 2^{2(1-\alpha)}R^2\|K\|_{L^1(\mathbb R)}\int_{\mathbb R}|K(u)||uh_n|^{2\alpha}du,
}
where the second line follows by Jensen's inequality and the third by \cref{lem:sobolev}. 
As a result, since $h_n \to 0$,
\[
\max_{j\in[p]}\abs{\Var\sbra{\binom{n}{2}P\psi_j(X_1)}-\Var[f_j(X_{1j})]}\to0
\]
as $n\to\infty$.
Meanwhile, \eqref{ade-int1} also yields
\ba{
P^2(\psi_j^2)
&=\frac{1}{h_n\binom{n}{2}^2}\int_{\mathbb R^2}K(u)^2f_j(y_j+uh_n)f_j(y_j)dudy_j.
}
Hence, 
\ban{
\abs{h_n\binom{n}{2}^2P^2(\psi_j^2)-\|K\|_{L^2(\mathbb R)}^2\E[f_j(X_{1j})]}
&\leq\int_{\mathbb R}\int_{\mathbb R}K(u)^2|f_j(y_j+uh_n)-f_j(y_j)|f_j(y_j)dudy_j\notag\\
&\leq\int_{\mathbb R}K(u)^2\sqrt{\int_{\mathbb R}|f_j(y_j+uh_n)-f_j(y_j)|^2dy_j}\|f_j\|_{L^2(\mathbb R)}du\notag\\
&\leq 2^{1-\alpha}\|K\|_{L^\infty(\mathbb R)}\|f_j\|_{L^2(\mathbb R)}\|f_j\|_{H^\alpha}\int_{\mathbb R}|K(u)||uh_n|^{\alpha}du,\label{ade-var-approx}
}
where the second line follows by the Schwarz inequality and the third by \cref{lem:sobolev}. 
Since $\max_{j\in[p]}\binom{n}{2}|P^2\psi_j|=O(1)$ by \eqref{ade-int1}, we conclude
\[
\max_{j\in[p]}\abs{h_n\binom{n}{2}^2\Var[\psi_j(X_1,X_2)]-\|K\|_{L^2(\mathbb R)}^2\E[f_j(X_{1j})]}\to0
\]
as $n\to\infty$. 
Therefore
$
\max_{j\in[p]}\sigma_{j}^{-2}=O(\{n^{-1} + n^{-2}h_n\}^{-1})= O\left( \min\{n, n^2h_n\} \right).
$
Next, we verify the conditions in \cref{rmk:coro:order-2}. 
By \eqref{ade-int1}, $|P(\psi_j^m)(x)|\lesssim R\|K\|_{L^\infty(\mathbb R)}^{m-1}h_n^{-m+1}n^{-2m}$ for any $x\in\mathbb R^p$ and integer $m\geq1$. 
Therefore, 
\ban{
\tilde\Delta_{2,*}^{(1)}(1)\log^5(np)
&=n^5\max_{j\in[p]}\frac{\|P\psi_{j}\|_{L^4(P)}^4}{\sigma_j^4}\log^5(np)
=O(n^{-1}\log^5(np)),\label{ade-cond1}\\
\tilde\Delta_{2,q}^{(2)}(1)\log^5(np)&
\leq n^{4}\norm{\max_{j\in[p]}\frac{|P\psi_{j}|}{\sigma_j}}_{L^{\infty}(P)}^4\log^6(np)
=O(n^{-2}\log^6(np)),\label{ade-cond2}\\
\tilde\Delta_{2,*}^{(1)}(2)\log^5(np)
&\leq n^2\max_{j\in[p]}\frac{\|\psi_{j}\|_{L^4(P^2)}^4}{\sigma_j^4}\log^{8}(np)
=O(n^{-2}h_n^{-1}\log^8(np)),\label{ade-cond3}\\
\tilde\Delta_{2,*}^{(2)}(2)\log^5(np)
&\leq n^3\max_{j\in[p]}\frac{\|P(\psi_{j}^2)\|_{L^{2}(P)}^{2}}{\sigma_j^4}\log^{7}(np)
=O(n^{-1}\log^7(np)),\label{ade-cond4}\\
\tilde\Delta_{2,*}^{(3)}(2)\log^5(np)
&\leq n\norm{\max_{j\in[p]}\frac{P(\psi_{j}^4)}{\sigma_j^4}}_{L^{\infty}(P)}\log^9(np)
=O(n^{-3}h_n^{-1}\log^9(np)),\label{ade-cond5}\\
\tilde\Delta_{2,q}^{(5)}(2)\log^9(np)
&=n^{2}\norm{\max_{j\in[p]}\frac{P(\psi_{j}^2)}{\sigma_j^2}}_{L^{\infty}(P)}^2\log^{12} (np)
=O(n^{-2}\log^{12}(np)).\label{ade-cond6}
}
Also, since $|\psi_j(x,y)|\lesssim \|K\|_{L^\infty(\mathbb R)}h_n^{-1}n^{-2}$ for all $x,y\in\mathbb R^p$,
\ban{
\tilde\Delta_{2,q}^{(4)}(2)\log^5(np)
&= \norm{\max_{j\in[p]}\frac{|\psi_{j}|}{\sigma_j}}_{L^\infty(P^2)}^4\log^{10}(np)
=O(n^{-4}h_n^{-2}\log^{10}(np)).\label{ade-cond7}
}

In addition, since
\ba{
\psi_j\star_1^1\psi_j(X_{1},X_{2})
=\frac{1}{h_n\binom{n}{2}^{2}}\int_{\mathbb R}K(u)K\bra{\frac{X_{1j}-X_{2j}}{h_n}+u}f(X_{1j}+uh_n)du,
}
we have
\ba{
\|\psi_j\star_1^1\psi_j\|_{L^2(P^2)}^2
&\leq \frac{R^2}{h_n^{2}\binom{n}{2}^{4}}\int_{\mathbb R^3}K(u)K\bra{\frac{x_1-x_{2}}{h_n}+u}^2f_j(x_1)f_j(x_2)dudx_1dx_2\\
&\leq \frac{R^3}{h_n\binom{n}{2}^{4}}\int_{\mathbb R^3}K(u)K(v)^2f_j(x_1)dudx_1dv 
\leq\frac{R^3\|K\|_{L^\infty(\mathbb R)}}{h_n\binom{n}{2}^{4}}.
}
Hence 
\ben{\label{ade-cond8}
\tilde\Delta_1^{(0)}\log^5p=n^2\max_{j\in[p]}\frac{\|\psi_j\star_1^1\psi_j\|_{L^2(P^2)}}{\sigma_j^2}\log^5p
=O(\sqrt h_n\log^5p).
}
Consequently, provided that
\ben{\label{ass:ade}
\log^7p=o(n),\qquad
\log^8(np)=o(n^2h_n),\qquad
h_n\log^{10}p=o(1),
}
we have
\[
\sup_{A\in\mathcal R_p}\abs{\pr(\hat\theta_n-\E[\hat\theta_n]\in A)-\pr(Z\in A)}\to0,
\]
where $Z\sim N(0,\Cov[\hat\theta_n])$. 
In view of \eqref{eq:bias}, if we additionally assume $(\min\{n,n^2h_n\})^{1/2} h_n^{2\alpha}\sqrt{\log p}\to0$, then Lemma 1 in \cite{chernozhukov2023high} gives
\[
\sup_{A\in\mathcal R_p}\abs{\pr(\hat\theta_n-\theta\in A)-\pr(Z\in A)}\to0.
\]
We emphasize that our condition does not require $nh_n\to\infty$ as $n\to\infty$ and hence $\hat\theta_n$ may not be asymptotically linear. 
Besides, the lower bound condition on the bandwidth is $n^2h_n/\log^8(np)\to\infty$, which coincides with, up to a logarithmic factor, the weakest condition to ensure $\Var[\hat\theta_{n,j}]\to0$ as $n\to\infty$ for each $j$. 

\begin{rmk}[Relation to \cite{cattaneo2014small}]
We can relate conditions \eqref{ade-cond1}--\eqref{ade-cond8} to those in the proof of \cite[Theorem 1]{cattaneo2014small} as follows. 
First, \eqref{ade-cond1} corresponds to Eq.(A.5) in \cite{cattaneo2014small}.
Next, \eqref{ade-cond3}, \eqref{ade-cond4} and \eqref{ade-cond8} are counterparts of Eqs.(A.7), (A.8) and (A.9) in \cite{cattaneo2014small}, respectively. 
Third, \eqref{ade-cond2} and \eqref{ade-cond6} can be seen as maximal versions of \eqref{ade-cond1} and \eqref{ade-cond4}, respectively. 
Finally, we can interpret \eqref{ade-cond5} and \eqref{ade-cond7} as maximal versions of \eqref{ade-cond3}.  
\end{rmk}

\begin{rmk}[Application of \cref{coro:clt-ustat}]\label{rmk:ade}
If we apply \cref{coro:clt-ustat} instead of \cref{coro:order-2}, we need to replace the second condition in \eqref{ass:ade} by $\log^9p=o(n^3h_n^2)$, which requires $n^3h_n^2\to\infty$ as $n\to\infty$. 
\end{rmk}

\section{Application 1: Adaptive goodness-of-fit tests}\label{sec:gof}

Let $X_1,\dots,X_n$ be i.i.d.~random vectors in $\mathbb R^d$ with common distribution $P$. 
Unlike \cref{sec:app_ade}, we assume that $P$ does not depend on $n$, and so does $d$. 
Assume that $P$ has density $f$. 
We aim to test whether $f$ is equal to a prespecified density function $f_0$ or not, based on the data $X_1,\dots,X_n$. 
Namely, we consider the following hypothesis testing problem:
\[
H_0:f=f_0\qquad\text{vs}\qquad H_1:f\neq f_0.
\]
Let $K:\mathbb R^d\to\mathbb R$ be a bounded positive definite function; recall that $K$ is said to be \emph{positive definite} if $(K(u_i-u_j))_{1\leq i,j\leq N}$ is a positive definite symmetric matrix for all $N\geq1$ and $u_1,\dots,u_N\in\mathbb R^d$. Note that $K$ is particularly symmetric. 
For every positive number $h>0$, write
\ba{
\varphi_h(x,y)=\frac{1}{h^d}K\bra{\frac{x-y}{h}},\quad x,y\in\mathbb R^d.
}
Then we define
\[
\hat\varphi_h(x,y)=\varphi_h(x,y)-P_0\varphi_h(x)-P_0\varphi_h(y)+P_0^2\varphi_h,
\]
where $P_0$ is the probability distribution on $\mathbb R^d$ with density $f_0$. 
A straightforward computation shows
\ben{\label{eq:mmd}
\E[\hat\varphi_h(X_1,X_2)]=\int_{\mathbb R^d\times\mathbb R^d}\varphi_h(x,y)\{f(x)-f_0(x)\}\{f(y)-f_0(y)\}dxdy,
}
which is equal to the squared \emph{maximum mean discrepancy} (MMD) between $P$ and $P_0$, based on the kernel $\varphi_h$ (see Eq.(10) in \cite{sriperumbudur2010hilbert}). 
In particular, $\E[\hat\varphi_h(X_1,X_2)]=0$ if and only if $f=f_0$ a.e., provided that $\varphi_h$ is a characteristic kernel in the sense of \cite[Definition 6]{sriperumbudur2010hilbert}. 
This suggests rejecting the null hypothesis when an estimator for $\E[\hat\varphi_h(X_1,X_2)]$ takes a too large value. Since the (averaged) $U$-statistic $U_2(\hat\varphi_h)=\binom{n}{2}^{-1}J_2(\hat\varphi_h)$ is an unbiased estimator for $\E[\hat\varphi_h(X_1,X_2)]$, it is natural to use a properly normalized version of $U_2(\hat\varphi_h)$ as a test statistic. This turns out to be $J_2(\hat\psi_h)$ with $\hat\psi_{h}:=\sqrt{h^d\binom{n}{2}^{-1}}\hat\varphi_h$ (cf.~\cref{mmd-var}). 
Recently, \cite{li2024optimality} have shown that this test is minimax optimal against smooth alternatives if $K$ is the Gaussian density and $h$ is chosen appropriately. 
To be precise, denote by $\pdf_d$ the set of probability density functions on $\mathbb R^d$. 
Fix a constant $R>0$. 
Given a constant $\alpha>0$ and a sequence $\rho_n$ of positive numbers tending to 0 as $n\to\infty$, we associate the sequence of alternatives as
\[
H_1(\rho_n;\alpha):=\cbra{f\in \pdf_d:\|f-f_0\|_{H^\alpha}\leq R,\|f-f_0\|_{L^2(\mathbb R^d)}\geq\rho_n}.
\]
In Theorem 2 of \cite{li2024optimality}, they have shown that if $\|f_0\|_{H^\alpha}<\infty$ and we choose $h=h_n$ so that $h_n\asymp n^{-2/(4\alpha+d)}$, the aforementioned test is consistent for the alternative $f\in H_1(\rho_n;\alpha)$ as long as $\rho_n/\rho_n^*(\alpha)\to\infty$, where $\rho_n^*(\alpha):=n^{-2\alpha/(4\alpha+d)}$. 
Moreover, if $\liminf_{n\to\infty}\rho_n/\rho_n^*(\alpha)<\infty$ and $\|f_0\|_{H^\alpha}<R$, there is no consistent test against $f\in H_1(\rho_n;\alpha)$ for some significance level by \cite[Theorem 3]{li2024optimality}; see also \cite{arias2018remember} for related results in the case of H\"older classes. 
An apparent problem of this test is that we should choose the bandwidth $h$ depending on $\alpha$ whose exact value is rarely known in practice. 
Therefore, one would wish to construct an \emph{adaptive} test in the sense that it does not require knowledge of $\alpha$ while keeping the power of the test as possible. 
We refer to \cite{ingster2000adaptive} and \cite[Section 8.1]{gine2016mathematical} for formal discussions of adaptive tests. 
To achieve this goal, \cite{li2024optimality} have considered the maximum of $J_2(\hat\psi_h)$ over a range of $h$ and showed that this test is adaptive to $\alpha\geq d/4$ up to a logarithmic factor; see Theorem 9 ibidem and also \cref{rmk:gof} below for a discussion. 
In this section, we use our theory to refine their result. 
Specifically, we consider the test statistic $
T_n:=\max_{h\in\mcl H_n}J_2(\hat\psi_h),
$
where
\[
\mcl H_n:=\cbra{\bar h_n/2^k:k=0,1,\dots,\lfloor\log_2(n^{2/d}/(\bar h_n\log^{5/d}n))\rfloor},
\]
and the sequence $\bar h_n$ is chosen so that $n^\delta\bar h_n\to\infty$ and $\bar h_n^\delta\log n\to0$ as $n\to\infty$ for any $\delta>0$. 
We can take $\bar h_n=e^{-\sqrt{\log n}}$ for example. 
We have defined $\mcl H_n$ so that the smallest bandwidth $\ul h_n:=\min\mcl H_n$ satisfies the following condition. 
\ben{\label{ass:bandwidth}
\log^5n=O(n^2\ul h_n^d)\quad\text{as }n\to\infty.
}
Note that unlike \cite{li2024optimality}, we take the maximum over a finite set of bandwidths. Apart from mathematical tractability, this is computationally attractive and is employed by several authors; see \cite{chetverikov2021adaptive} and references therein.  
For the kernel function $K$, we impose the following standard regularity assumption:
\begin{assumption}[Kernel]\label{ass:kernel-gof}
$K$ is bounded and positive definite. 
Also, $K\in L^1(\mathbb R^d)$ and $\int_{\mathbb R^d}K(u)du=1$.   
\end{assumption}
Note that we do not need to assume that the induced kernels $\varphi_h$ are characteristic in the sense of \cite[Definition 6]{sriperumbudur2010hilbert} because we consider the situation $h\to0$.

To compute quantiles of $T_n$, \cite{li2024optimality} suggest simulating $T_n$ under the null hypothesis. 
Although this is theoretically feasible, it is often computationally difficult or demanding to generate random variables from the general distribution $P_0$, even when $f_0$ is analytically tractable. 
For this reason, we suggest approximating the null distribution of $T_n$ using our theory. 
Let $Z^0=(Z^0_h)_{h\in\mcl H_n}$ be a centered Gaussian random vector with covariance matrix $((hh')^{d/2}P_0^2(\hat\varphi_{h}\hat\varphi_{h'}))_{h,h'\in\mcl H_n}$ and set $T_n^G:=\max_{h\in\mcl H_n}Z^0_h$. 
Also, denote by $\pr_f$ the probability measure on $(\Omega,\mcl A)$ under which the common distribution of $X_i$ has density $f$. 
\begin{proposition}\label{gof-ga-null}
Assume $\|f_0\|_{H^\gamma}<\infty$ for some $\gamma>0$. 
Under \cref{ass:kernel-gof}, we have
\[
\sup_{t\in\mathbb R}\abs{\pr_{f_0}(T_n\leq t)-\pr(T_n^G\leq t)}\to0\quad\text{as }n\to\infty.
\]
\end{proposition}
Since the covariance matrix of $Z^0$ is known, we can in principle compute quantiles of $T_n^G$ by simulation, and they can be used to construct (approximate) critical regions of the test. 
However, $Z^0$ is a high-dimensional random vector, so its simulation could be computationally demanding. 
Instead, we suggest a bootstrap procedure that does not require any simulation of multivariate random variables. 
Let $(\zeta_i)_{i=1}^n$ be i.i.d.~standard normal variables independent of the data $(X_i)_{i=1}^n$. 
We define a bootstrap version of $J_2(\hat\psi_h)$ as
\[
J_2^*(\psi_h):=\sum_{1\leq i<j\leq n}\zeta_i\zeta_j\psi_h(X_i,X_j),\quad\text{where }\psi_h:=\sqrt{h^d\binom{n}{2}^{-1}}\varphi_h.
\]
Here, we use $\psi_h$ instead of $\hat\psi_h$ for construction to make the mathematical analysis (slightly) simpler. 
Then we define the bootstrap test statistic as $T_n^*:=\max_{h\in\mcl H_n}J_2^*(\psi_h)$. 
Given a significance level $0<\tau<1$, let $\hat c_\tau$ be the $(1-\tau)$-th quantile of $T_n^*$ conditional on the data. That is,
\[
\hat c_\tau:=\inf\cbra{t\in\mathbb R:\pr^*(T_n^*\leq t)\geq1-\tau},
\]
where $\pr^*$ denotes the conditional probability given the data. 
\begin{theorem}[Size control]\label{gof-size}
Under the assumptions of \cref{gof-ga-null},  
$
\pr_{f_0}(T_n> \hat c_\tau)\to\tau
$ as $n\to\infty$. 
\end{theorem}
\cref{gof-size} suggests rejecting the null hypothesis if $T_n>\hat c_\tau$. 
The following result shows that this test is adaptive in the sense described in \cite{ingster2000adaptive}. 
\begin{theorem}[Adaptation]\label{gof-ada}
For every $\alpha>0$, define
\[
\rho_n^{ad}(\alpha):=\bra{\frac{\sqrt{\log\log n}}{n}}^{2\alpha/(4\alpha+d)}.
\]
Let $0<\alpha_0<\alpha_1$ and suppose that we have a family of sequences $\rho_n(\alpha)$ $(\alpha_0<\alpha<\alpha_1)$ such that $\inf_{\alpha_0<\alpha<\alpha_1}\rho_{n}(\alpha)/\rho_n^{ad}(\alpha)\to\infty$ as $n\to\infty$. Under the assumptions of \cref{gof-ga-null}, we have
\ba{
\sup_{\alpha_0<\alpha<\alpha_1}\sup_{f\in H_1(\rho_n(\alpha);\alpha)}\pr_f(T_n> \hat c_\tau)\to1\quad\text{as }n\to\infty.
}
\end{theorem}
Note that in \cref{gof-ada}, $f_0$ does not necessarily belong to the same Sobolev space as $f$.

\begin{rmk}[Comparison to \cite{li2024optimality}]\label{rmk:gof}
\cref{gof-ada} refines the result of \cite[Theorem 9]{li2024optimality} in three directions: (i) It does not require $\alpha\geq d/4$. (ii) The kernel function $K$ is not necessarily Gaussian. (iii) Distinguishable separation rates $\rho_n(\alpha)$ are smaller; \cite{li2024optimality} need to replace $\sqrt{\log\log n}$ in $\rho_n^{ad}(\alpha)$ by $\log\log n$. Note that the remaining $\sqrt{\log\log n}$ factor is not an artifact but is essential; see Theorem 1 in \cite{ingster2000adaptive}. 
We also mention that in the context of two-sample testing, \cite{schrab2023mmd} have addressed items (i) and (ii) but they additionally assume that the underlying densities are bounded; see \cite[page 54, footnote 10]{schrab2023mmd}. 
\end{rmk}

\section{Application 2: Inference on many density weighted average derivatives} \label{sec:DWAD}

\subsection{Jackknife multiplier bootstrap for many density weighted average derivatives}  \label{subsec:DWAD-theory}
Let $Z_{i} \coloneqq (X_{i}, Y_{i})~(i=1,\dots,n)$ be i.i.d.~observations of $Z=(X,Y)$, where $X$ is a random vector in $\mathbb{R}^d$ with density function $f$ and $Y$ is a random vector in $\mathbb{R}^p$.  
Our parameters of interest are the Density-Weighted Averaged Derivatives (DWADs) defined as follows.
For $j = 1,\dots, p$ and $k=1,\dots d$, 
\[
    \theta_{jk} \coloneqq \E\left[ f(X) \frac{\partial}{\partial X_k} g_j(X)\right], \quad \text{with} \quad g_j(x) \coloneqq \E[Y_{ij} \mid X = x].
\]
This parameter captures how much an infinitesimal change of $X_k$ affects $Y$ conditional on $X$ on average.
Although our Gaussian approximation can accommodate the situation where $d$ goes to infinity as long as Assumption \ref{as:DWAD-bandwdith} stated latter is satisfied with slight modifications of the assumptions and proof, we focus on the case under fixed $d$.
Meanwhile, we allow $p$ to diverge to infinity as $n\to\infty$.
\citet{powell1989semiparametric} proposed the following estimator for $\theta_{jk}$:
\[
    \hat{\theta}_{jk,h_n} \coloneqq \frac{1}{n}\sum_{i=1}^n\sum_{l\neq i}^n\psi_{jk,h_n}(Z_i,Z_l), \quad \text{with} \quad \psi_{jk,h_n}(Z_i,Z_l) \coloneqq \frac{-1}{(n-1)h_n^{d+1}} (Y_{ij}-Y_{lj})\partial_kK_{il,h_n} 
\]
where $K_{il,h_n} \coloneqq K((X_{ij} - X_{lj})/h_n)$ and $\partial_kK \coloneqq \frac{\partial}{\partial u_k} K(u_1, \dots, u_d)$ is a partial derivative of a kernel function $K$ satisfying \cref{as:DWAD-Kernel}.

To conduct simultaneous inference, we employ the Jackknife Multiplier Bootstrap (JMB) in the terminology of \cite{chen2018gaussian}. 
We could also conduct bootstrap inference under some appropriate conditions with the randomly reweighted bootstrap (adopted for a goodness-of-fit test in Section \ref{sec:gof}) or the $m$-out-of-$n$ bootstrap (considered by \citealp{cattaneo2014bootstrapping} under the fixed dimensional setting) but we do not pursue the possibility. 

Let $\{z_i\}_{i=1}^n$ be a sequence of random variables such that $z_i \overset{\text{i.i.d}}{\sim} N(0,1)$ and independent of $\{Z_i\}_{i=1}^n$. 
Then, the test statistic of JMB is defined as follows:
\begin{align*}
   & \hat{J}_{jk,h_n'}^{\sharp} \coloneqq\frac{1}{\hat{\sigma}_{jk,h_n'}}\left( \frac{1}{n}\sum_{i=1}^n z_i \left[ \sum_{l\neq i}^n \psi_{jk,h_n'}(Z_i,Z_l) - \hat{\theta}_{jk,h_n'} \right]\right) , \\
   & \quad \text{with} \quad \psi_{jk,h_n'}(Z_i,Z_l) \coloneqq \frac{-1}{(n-1)h_n'^{d+1}} (Y_{ij}-Y_{lj})\partial_k K_{il,h_n'} ,
\end{align*}  
where the bandwidth $h_n'$ satisfies $h_n' = 2^{1/(d+2)} h_n$ and $\hat{\sigma}_{jk,h_n'}$ is a standard error of $\hat{\theta}_{jk,h_n}$.
As Theorems 2(a), (b) and (c) in \cite{cattaneo2014small}, \cite{cattaneo2014small} proposed three types of consistent estimators of ${\sigma_{jk,h_n}^2} \coloneqq \Var[\hat{\theta}_{jk,h_n}]$. 
Among them, we adopt their Theorem 2(c) because Theorem 2(a) requires one more different bandwidth and Theorem 2(b) possibly takes negative values. 
The variance estimator of Theorem 2(c) in \cite{cattaneo2014small} is given by:
\[
    \hat{\sigma}^2_{jk,h_n'} \coloneqq \frac{1}{n}\sum_{i=1}^n \left[ \sum_{l\neq i}^n \psi_{jk,h_n'}(Z_i,Z_l) - \hat{\theta}_{jk,h_n'} \right]^2.
\]
We define the bootstrap statistic as $\hat{T}_n^\sharp \coloneqq \max_{(j,k,h_n) \in[p]\times[d]\times\mathcal{H}_n} |\hat{J}_{jk,h_n'}^{\sharp}|$, where $\mathcal{H}_n$ is a finite set of bandwidths.
Given a significance level $0<\tau<1$, the critical value is defined as the $(1-\tau)$-th quantile of $\hat{T}_n^\sharp$ conditional on data. 
That is, $\hat{c}^\sharp_\tau \coloneqq \inf\{t\in\mathbb{R} : \mathbb{P}^*(\hat{T}_n^\sharp \le t) \ge 1-\tau\}$.
We consider the validity of this critical value under the following assumptions. 

\begin{assumption}[Data Generating Process] \label{as:DWAD-DGP}
There exist constants $C>0$ and $0<\alpha\leq1/2$ independent of $n$ satisfying the following conditions:
\begin{enumerate}[label=(\roman*)]
    
    \item $f$ and $g_j$ are differentiable for all $j\in[p]$,  $\max_{k\in[d]}\|\partial_k f\|_{L^\infty(\mathbb{R}^d)}\leq C$ and $\max_{j\in[p],k\in[d]}\|\partial_k \{g_jf\}\|_{L^\infty(\mathbb{R}^d)} \leq C$.
    
    \item  $\|f\|_{H^\alpha}\leq C$, $\max_{j\in[p]} \|g_j\|_{H^\alpha}\leq C $, $\max_{j\in[p]} \|v_j\|_{H^\alpha}\leq C $, $\max_{k\in[d]}\|\partial_k f\|_{H^\alpha}\leq C$, $\max_{j\in[p]} \|g_jf\|_{H^\alpha}\leq C $ and
     $\max_{j\in[p],k\in[d]}\|\partial_k \{g_jf\}\|_{H^\alpha} \leq C$, where $v_j(x) \coloneqq \E[Y_j^2 \mid X=x]$.
    
    \item $\lim_{|x| \to \infty} |g_j(x)f(x)^2| \to 0$ for all $j\in[p]$. 
    
    \item $\max_{j\in[p]} |Y_j| \leq C$ almost surely.
    
    \item $\Var[\partial_k \{g_jf\}(X) - Y_j\partial_k f(X)]\geq1/C$  and $\E[\Var(Y_j\mid X) f(X)]\geq1/C$ for all $j\in[p]$ and $k\in[d]$.
    
\end{enumerate}
\end{assumption}

\begin{assumption}[Kernel] \label{as:DWAD-Kernel} 
\begin{itemize}\mbox{}
    \item[(i)] $K:\mathbb{R}^d \to \mathbb{R}$ is even, differentiable and $\int_{\mathbb{R}^d} K(u)du =1$.
    \item[(ii)]  $\int_{\mathbb{R}^d} |u||K(u)|du < \infty$, $ \|\partial_k K\|_{L^2(\mathbb{R}^d)} < \infty$ and $\|\partial_k K\|_{L^\infty(\mathbb{R}^d)}<\infty$, $\int_{\mathbb{R}^d} |u|^\alpha|\partial_k K(u)| du< \infty$ for all $k\in[d]$ and $0\le \alpha\le 1/2$. 
\end{itemize}
\end{assumption}

\begin{assumption}[Bandwidth and Dimensionality] \label{as:DWAD-bandwdith}
\begin{itemize}\mbox{}
Let $\overline{h}_n$ and $\underline{h}_n$ be the maximum and minimum of the bandwidth space $\mathcal{H}_n$ and $|\mathcal{A}_n| \coloneqq p \times d\times |\mathcal{H}_n|$. 
    \item[(i)] $\min\{n,n^2\bar{h}_n^{d+2}\}^{1/2} \overline{h}_n^{2\alpha} \sqrt{\log |\mathcal{A}_n|} \to 0$, $ \overline{h}_n^{d/2} \log^5 |\mathcal{A}_n| \to 0$, and $\bar{h}_n^{\alpha}\log^2|\mathcal{A}_n|\to 0$.
    \item[(ii)] $n^{-1}\log^7 (n|\mathcal{A}_n|) \to 0$, $n^{-2}\underline{h}_n^{-d} \log^8(n|\mathcal{A}_n|) \to 0$.
    \item[(iii)] The bandwidth $h_n'$ (used for the JMB and studentization) satisfies $h_n' = 2^{1/(d+2)} h_n$. 
\end{itemize}
\end{assumption}

Assumptions \ref{as:DWAD-DGP} and \ref{as:DWAD-Kernel} are standard conditions in the literature of the DWAD. 
The notable difference from the assumptions in \cite{cattaneo2014small} is that they consider the smoothness of functions associated with DGP in terms of the degrees of differentiability while we does in terms of Sobolev order $\alpha$.
As noted in the last paragraph of \cite[p.52]{gine2008simple}, the restriction $\alpha\le 1/2$ is imposed only for technical simplicity and can be dropped. 
In the empirical analysis in the next section, we allow $1/2<\alpha\le 1$ when constructing a bandwidth space.
Assumption \ref{as:DWAD-bandwdith}(i) and (ii) are almost the same as those in Section \ref{sec:app_ade}. 
Assumption \ref{as:DWAD-bandwdith}(iii) makes JMB valid and variance estimation consistent.



\begin{theorem} \label{thm:DWAD_size}
    Under the assumptions presented above, it holds that $\mathbb{P}(\hat{T}_n>\hat{c}^\sharp_\tau) \to \tau$ as $n\to\infty$, where $\hat{T}_n$ is defined as $\hat{T}_n \coloneqq \max_{(j,k,h_n)\in[p]\times[d]\times\mathcal{H}_n} \hat{\sigma}_{jk,h_n'}^{-1} |\hat{\theta}_{jk,h_n'} - \theta_{jk}|$. 
\end{theorem}

The proof of this section is in Appendix \ref{sec:proof_DWAD}.

\subsection{Empirical Application : Price Elasticities Matrices of food in India 2022}
In this section, we provide an empirical application of the Gaussian approximation result presented in \cref{subsec:DWAD-theory}. 
Specifically, we conduct statistical inference on high-dimensional price elasticities using the nonparametric approach proposed by \cite{deaton1998parametric}. 
Price elasticity is of interest for a wide range of purposes, including tax and subsidy reforms in public and development economics \citep{deaton1998parametric}, pricing, advertising and product display strategy in marketing science \citep{jiang2025online}, and many other applications.

We use the dataset of 2022-2023 Indian Household Consumption Expenditure Survey (HCES) by National Sample Survey (NSS). 
The cleaned dataset is available at Advait Moharir's Github\footnote{\url{https://github.com/advaitmoharir/hces_2022/tree/main}}.
We focus on ten goods throughout the analysis.
To control and capture heterogeneity across rural/urban areas and income levels, we stratify households in the dataset into eight mutually exclusive subsamples obtained by crossing place of residence (rural or urban) with income quantile, where income is proxied by Monthly Per-Capita Consumption Expenditure (MPCE).
We focus on this rural/urban split for simplicity, but a comprehensive empirical investigation would ideally stratify at finer geographic levels (such as state, village and so on) to accommodate local price and consumption variation.
The sample sizes of each groups are 38559 or 38560 for each quartile in rural area and are 25825 or 25826 for each quartile in urban area.
For our analysis we require, for each household $i$ and each good $j$, both the quantity consumed (denoted $Q_{ij}$) and a corresponding price measure (denoted $P_{ij}$).
As to the quantity, missing quantities are replaced by zero and then log-transformed as $\log(Q_{ij}+10^{-3})$ where the small constant $10^{-3}$ prevents taking the logarithm of zero.
Concerning the price, following \citet{deaton1998parametric}, we use the log of the household-specific unit value (the average price actually paid by the household $i$ for good $j$). 
In line with \citet[Section 4.3]{deaton1998parametric}, missing unit values are imputed by the median within each region-income group. While \cite{deaton1998parametric} employ village means, we use medians to mitigate the influence of outliers.

Following (14) and (15) of \cite{deaton1998parametric}, we estimate the price elasticities using DWAD. 
Although \cite{deaton1998parametric} estimate the average partial derivative of the household expenditure on good $j$ (i.e. $P_jQ_j$) with respect to the log-price of good $k$ (i.e. $\log P_k$), the resulting estimates themselves are not those of the price elasticities and the transformation of the estimates into those of the price elasticities leaves an even asymptotically non-negligible bias. 
Though the transformation still has a practical value as an approximation, we directly estimate the price elasticities instead. 
In particular, the average partial derivative of the log-household purchase of good $j$ (i.e. $\log Q_j$) with respect to the log-price of good $k$ (i.e. $\log P_k$):
\[
    \theta_{jk} \coloneqq \E\left[ f(\log P_1, \dots, \log P_{10}) \frac{\partial m_j(P_1, \dots, P_{10})}{\partial \log P_k}\right] 
\]
with $m_j(P_1, \dots, P_{10}) = \E[\log Q_j \mid P_1, \dots, P_{10}]$ and the expectation is taken with respect to the joint probability density of $(\log P_{1},\dots,\log P_{10})$. 

The critical value (at the significance level of $0.05$) is constructed via the JMB presented in the previous subsection with $499$ bootstrap replicates. 
For computational simplicity, we adopt  \cite{westfall1993resampling}'s method, which is, according to \cite{meinshausen2011asymptotic}, an asymptotically optimal single-step procedure when the number of hypothesis tends to infinity, and do not seek a potentially improved power with \cite{romano2005exact}'s step-down methods.
Note that subset pivotality holds for our procedure and therefore the family-wise error rate (FWER) is strongly controlled because the joint distribution of the test statistics corresponding to any subset of null hypotheses does not depend on (i) whether the complementary hypotheses are true or false and on (ii) parameters associated with those complementary hypotheses. 
When a more stringent level of power is desired, implementing a step-down procedure is a preferred approach (see \citealp[Section 5]{CCK13}).
As in Section 3, we take an agnostic stance on the smoothness level $\alpha$ of the data generating process (DGP) and conduct inference in an adaptive manner and construct the space of bandwidth candidates as $\mathcal{H}_n \coloneqq \{ h_\alpha = C_h \cdot n^{-2/(10+ 2\alpha - 1/10)} : \alpha= \{0,0.1,0.2,0.3,0.4,0.5, 0.6,0.7,0.8,0.9,1.0\}  \}$ with  $C_h \coloneqq (\max_{1\le i \le n} \log P_{ij} - \min_{1\le i \le n} \log P_{ij})_{1\le j \le 10}$. 
We specify the contraction rate of the bandwidth as $n^{-2/(10+ 2\alpha - 1/10)}$ in accordance with the MSE-optimal bandwidth by \cite{cattaneo2010robust}, which is proportional to $n^{-2/(10+ 2\alpha)}$.
The factor $-1/10$ is an adjustment term for undersmoothing. 
In our smoothness agnostic setting, the optimal constant derived via Taylor expansion by  \cite{cattaneo2010robust} is infeasible, so we set $C_h$ to the difference between the maximum and minimum of $\log P_{ij}$ following \cite[footnote 4]{chetverikov2021adaptive}. 
To avoid oversmoothing due to potential influential outliers, we calculate $C_h$ after trimming 0.5$\%$ from each tail of $\log P_{ij}$. 

Note that the target parameter for statistical inference is high-dimensional.
Specifically, the price elasticities are defined for every pair of goods, and since we focus on ten goods, this alone yields $10\times10=100$ parameters.
In addition, as mentioned above, we stratify the sample into eight sub-groups, so the dimension rises to $100\times8=800$.
Moreover, we consider eleven candidates of bandwidths and build the critical value to be the most conservative across them.
Taking all these factors together, we ultimately carry out inference on the parameters of $800\times 11 = 8,800$ dimensions.

Figure \ref{fig:empirical_result} reports the estimated price‐elasticity matrices for four household groups (rural–Q1, rural–Q4, urban–Q1, and urban–Q4) defined by location (rural and urban) and income level (lowest‐quartile Q1 and highest-quartile Q4). 
Heatmaps for the remaining groups and for alternative smoothness values $\alpha$ (84 additional heatmaps in total) are omitted here due to space constraints.
The insignificant estimates are black-outed.
A significant substitute relationship is found between rice and wheat in all four groups.
Also, the own-price elasticities for these staples become progressively smaller in absolute value as we move from low-income (Q1) to high-income (Q4) households, indicating that demand is less price‐elastic at higher income levels. 
Interestingly, for urban-Q1 households, rice demand is more price-elastic than wheat demand, whereas the reverse is observed in the other three groups. 
A rise in the price of cheese significantly increases milk consumption among Q1 households in both rural and urban areas, suggesting that milk serves as a low-cost substitute for cheese when budgets are tight. 
For Q4 households, such a relationship is statistically insignificant, consistent with weaker budget constraints. 
A variety of additional substitute and complement relationships and richer economic interpretations could be explored but we do not seek the possibility as the present paper focuses on methodological contributions rather than policy implications.

\begin{figure}[H]    
  \centering
  \begin{subfigure}{0.48\linewidth}
    \includegraphics[width=\linewidth, trim=0 0 0 0.73cm,clip]{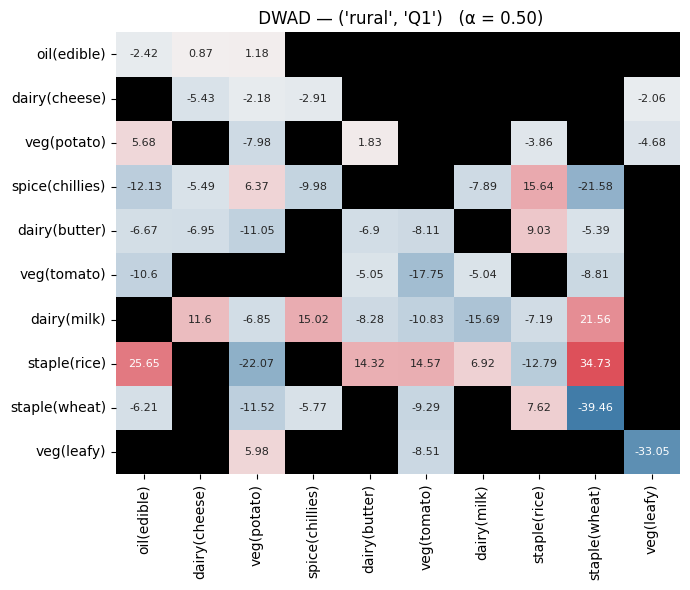}
    \caption{rural-Q1}
  \end{subfigure}\hfill      
  \begin{subfigure}{0.48\linewidth}
    \includegraphics[width=\linewidth, trim=0 0 0 0.73cm,clip]{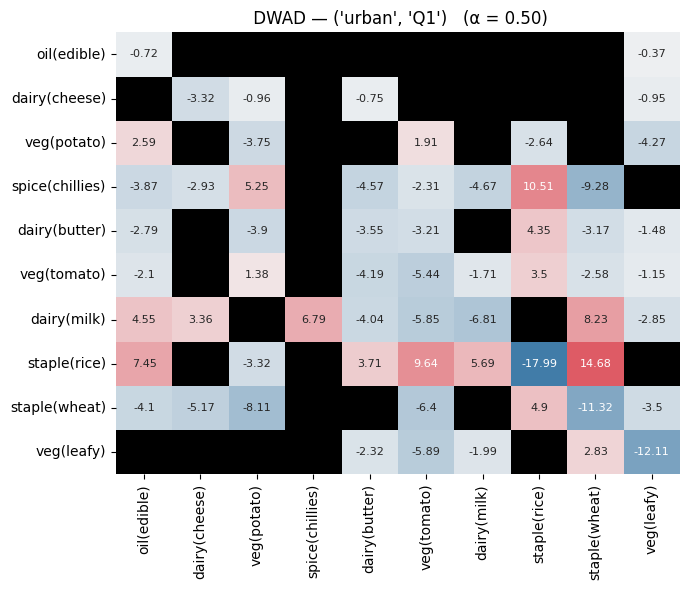}
    \caption{urban-Q1}
  \end{subfigure}

  \vspace{1em}               

  \begin{subfigure}{0.48\linewidth}
    \includegraphics[width=\linewidth, trim=0 0 0 0.73cm,clip]{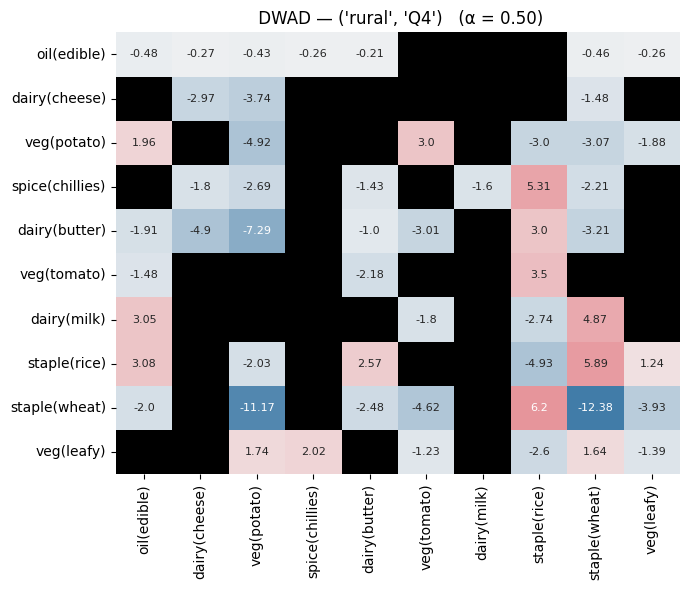}
    \caption{rural-Q4}
  \end{subfigure}\hfill
  \begin{subfigure}{0.48\linewidth}
    \includegraphics[width=\linewidth, trim=0 0 0 0.73cm,clip]{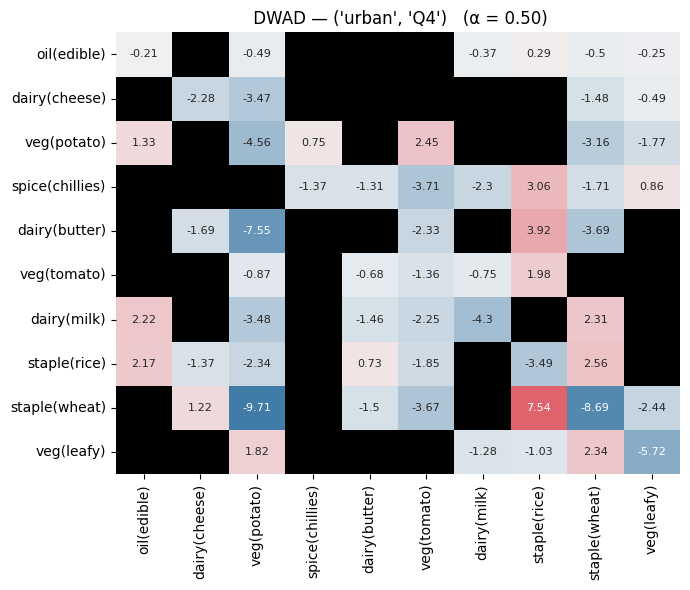}
    \caption{urban-Q4}
  \end{subfigure}

  \caption{Heatmaps of DWAD estimates for price-elasticities}
  \label{fig:empirical_result}
\end{figure}

\section{Conclusion and Discussion} \label{sec:discussion}
        In this study, we developed Gaussian approximation results for general symmetric $U$-statistics in the high-dimensional setting. 
        As an illustration, we considered small bandwidth asymptotics for estimating average marginal densities of high-dimensional data and an adaptive goodness-of-fit test against smooth alternatives, along with contributions to the literature on the applications. 
        Beyond the examples presented in the previous sections, our results have a wide range of potential applications, only a small portion of which are listed below.

        In Section \ref{sec:introduction}, we mentioned specification tests for parametric regression \citep{hardle1993comparing,zheng1996consistent}. 
        It is known that the test statistics of \cite{hardle1993comparing} and \cite{zheng1996consistent} with some specific weighting functions are approximated by certain degenerate second-order $U$-statistics;
        Also, beyond the second order $U$-statistics, \cite{fan1996consistent} considered testing the conditional mean independence and functional forms of partially linear models and single-index models, their proposed test statistics are approximated by certain degenerate fourth-order U-statistics.
        Both the examples provided here and the one treated in Section \ref{sec:gof} are illustrative in nature, and test statistics based on general symmetric $U$-statistics appear more broadly in specification tests.
        Moreover, in such cases, adaptive tests can be constructed in a straightforward manner by applying our Gaussian approximation results, as demonstrated in Section~\ref{sec:gof}.


       In \cref{sec:introduction}, we also mentioned DWAD and consider its application to the nonparametric inference on high-dimensional price elasticities in Section \ref{sec:DWAD}.
        Another use of DWADs is to estimate finite-dimensional parameters in the single-index model \citep{powell1989semiparametric}. 
        Letting $Z_i=(X_i,Y_i)$ $(i=1,\dots,n)$ be i.i.d.~observations of $Z=(X,Y)$, where $X$ is a random vector in $\mathbb R^d$ and $Y$ is a random variable,
        the semiparametric single index model is given by 
        \[
            Y_i = g(X_i^\top\theta) + \varepsilon_i, \quad \E[\varepsilon_i \mid X_i] = 0,
        \]
        where $g : \mathbb{R} \to \mathbb{R}$ is an unknown link function and $\theta\in\mathbb{R}^d$ is the parameter of interest. 
        Also, since $\tilde{g}'(X_i) \coloneqq \nabla_{X_i} g(X_i^\top\theta)  = g'(X_i^\top \theta)\theta$, it can be seen that $\E[f(X_i)\tilde{g}'(X_i)]$ is proportional to $\theta$:
        \[
            \E\left[ f(X_i) g'(X_i^\top \theta)\right] \theta = \E[f(X_i)\tilde{g}'(X_i)].
        \]
        Since the parameter of the single index model is identified only up to scale, an estimator for $\E[f(X_i)\tilde{g}'(X_i)]$ is also one of the estimators for $\theta$. 
        Although the single index model includes various limited dependent variable models, a useful special case is the semiparametric Type-I Tobit model and our developed Gaussian approximation results make it possible to examine cases with high-dimensional censored outcomes, such as top-coded incomes grouped by occupations in large labor markets and high-dimensional corner solutions in markets with numerous goods, without relying on the normality and homoscedasticity of the error term assumed in the standard method \citep{tobin1958corner,amemiya1973regression}. 
    
        One can also use DWADs to conduct statistical inference on the parameters of nonseparable models \citep[e.g.][]{altonji2012estimating} or to test whether marginal parameters satisfy the properties or conditions implied by economic theory \citep[e.g.][]{stoker1989tests,hardle1991empirical,deaton1998parametric,coppejans2005kernel,campbell2011competition,dong2022estimation}. 
        Let $X = (X_1^\top, X_2^\top)^\top$ and $Y=m(X_1, \varepsilon)$, where $m(\cdot)$ is an unknown function, and $\varepsilon$ is an unobservable random variable.
        Suppose that we are interested in estimating the marginal effect of $X_1$
        \[
            \theta \coloneqq \mathbb{E}\left[w(X_1,X_2) \frac{\partial}{\partial X_1}m(X_1,\varepsilon) \right],
        \]
        where $w(\cdot)$ is a known weight function. 
        Then, under the assumption that $X_1$ is independent of $\varepsilon$ conditional on $X_2$ together with some regularity conditions, paramter of interest in nonseparable model is given by
        \[
            \theta =  \E\left[w(X_1,X_2) \frac{\partial}{\partial X_1}g( X_1,X_2) \right],
        \]
        where $g(X) \coloneqq \mathbb{E}[Y\mid X]$. 
        $\theta$ captures the weighted average marginal effect of $X_1$. 
        This parameter includes other marginal parameters of potential interest in economics that are not always expressed as the parameter of nonseparable models.
        Although any strictly positive weight function can be used in testing the hypothesis ``$\E\left[ w(X_1, X_2) \partial m(X_1,\varepsilon)/\partial X_1\right] = \text{(some constant)} $'', choosing a density weight enables the complete removal of the random denominator problem.
        Among many marginal parameters of potential interest, as illustrative example where our Gaussian approximation results prove useful, we consider the inference on the high-dimensional price elasticities building on the idea of \cite{deaton1998parametric}.  
        Another potential application is to test whether some marginal parameters satisfy  conditions implied by economic theory.
        As a specific example, \cite{coppejans2005kernel} tested, using repeated cross-sectional data, whether a labor market is competitive by examining the hypothesis that the average wage equals the marginal wage with respect to working hours on average. 
        Notably, they conducted 36 separate tests for each of the 12 groups of occupations and 3 time points, and while they did not, similar tests based on various grouping criteria such as gender or income level, as well as those covering additional time points and occupations, could also be of interest, and that kind of multiple testing settings will be related to our developed Gaussian approximation results.

Finally, this study, in the current version, does not fully cover Gaussian approximations for weighted $U$-statistics. 
As a result, it cannot accommodate frameworks such as weak-many instrumental variables asymptotics \citep{chao2012asymptotic} or many covariates asymptotics \citep{cattaneo2018alternative,cattaneo2018inference}, due to the involvement of the inverse of a product of projection matrices in dominant terms in such settings.
We plan to address such situations in future research.

\clearpage

\appendix

\vspace{1cm}
\begin{center}
{\LARGE
{\bf Appendix}
}
\end{center}

\section{Proofs for Section \ref{sec:main}} \label{sec:proof_main}

Throughout the discussions, we will frequently use the following elementary inequality, sometimes without reference. For random variables $\xi_1,\dots,\xi_N$ and $1\leq m\leq q$,
\ben{\label{eq:max-out}
\norm{\max_{i\in[N]}|\xi_i|}_{L^m(\pr)}
\leq N^{1/q}\max_{i\in[N]}\norm{\xi_i}_{L^q(\pr)}.
}

We first introduce two main ingredients of the proof: High-dimensional CLTs via generalized exchangeable pairs and maximal inequalities. 
These results are proved later (see \cref{sec:proof-aux}). 
After that, we prove the main results presented in \cref{sec:main}. 

\subsection{High-dimensional CLTs via generalized exchangeable pairs} \label{subsec:hdgexch}

To effectively utilize the techniques developed in \cite{DoPe17,DoPe19}, it is convenient to have a high-dimensional CLT based on Stein's method of exchangeable pairs. 
While such a result has already been established in the literature (see Theorem 1.2 in \cite{fang2021high} for the non-degenerate covariance matrix case and Proposition 2 in \cite{cheng2022gaussian} for the possibly degenerate covariance matrix case), these results require the exchangeable pairs to satisfy the so-called approximate linear regression property (see Eq.(1.6) in \cite{fang2021high} and Eq.(10) in \cite{cheng2022gaussian}), which is not the case for the standard construction of exchangeable pairs for general symmetric $U$-statistics. 
For this reason, we develop the following new version, which can be seen as a variant of \cite[Theorem 7.1]{fang2023p} that concerns a bound in the $p$-Wasserstein distance. 
See also \cite{zhang2022berry} and \cite{dobler2023normal} for related results in the univariate setting. 
\begin{theorem}\label{thm:gexch}
Let $(Y,Y')$ be an exchangeable pair of random variables taking values in a measurable space $(E,\mcl E)$. Let $\mathsf{W}:E\to\mathbb R^p$ be an $\mcl E$-measurable function, and set $W:=\mathsf{W}(Y),W':=\mathsf{W}(Y')$ and $D:=W'-W$. 
Suppose that there exists an antisymmetric $\mcl E^{\otimes2}$-measurable function $\mathsf{G}:E^2\to\mathbb R^p$ in the sense that $\mathsf{G}(Y,Y')=-\mathsf{G}(Y',Y)$ and such that $G:=\mathsf{G}(Y,Y')$ satisfies
\ben{\label{eq:nlr}
\E[G\mid Y]=- (W+R)
} 
for some random vector $R$ in $\mathbb R^p$. 
Furthermore, let $\Sigma$ be a $p\times p$ positive semidefinite symmetric matrix such that $\ul\sigma:=\min_{j\in[p]}\sqrt{\Sigma_{jj}}>0$. 
Then, there exists a universal constant $C>0$ such that for any $\eps>0$, 
\besn{\label{eq:gech}
&\sup_{A\in\mathcal R_p}\left|\pr(W\in A)-\pr(Z\in A)\right|\\
&\leq \frac{C}{\ul\sigma}\bra{
\E\sbra{\|R^\eps\|_\infty}\sqrt{\log p}
+\eps^{-1}\E\sbra{\|V^\eps\|_\infty}(\log p)^{3/2}
+\eps^{-3}\E\sbra{\Gamma^\eps}(\log p)^{7/2}
+\eps\sqrt{\log p}
},
}
where $Z\sim N(0,\Sigma)$, $\beta=\eps^{-1}\log p$, 
\ba{
R^\eps&:=R+\E[G1_{\{\|D\|_\infty>\beta^{-1}\}}\mid Y],&
V^\eps&:=\frac{1}{2}\E[GD^\top1_{\{\|D\|_\infty\leq\beta^{-1}\}}\mid Y]-\Sigma,
}
and
\[
\Gamma^\eps:=\max_{j,k,l,m\in[p]}\E[|G_jD_kD_lD_m|1_{\{\|D\|_\infty\leq\beta^{-1}\}}\mid Y].
\]
\end{theorem}
%
\begin{rmk}[Comparison to \cite{cheng2022gaussian}]
When $G=\Lambda^{-1}D$ with $\Lambda$ a $p\times p$ invertible matrix, we can derive the following bound from Proposition 2 in \cite{cheng2022gaussian} and Nazarov's inequality \cite[Lemma A.1]{CCK17}: 
\ba{
&\sup_{A\in\mathcal R_p}\left|\pr(W\in A)-\pr(Z\in A)\right|\\
&\leq C_{\ul\sigma}\Biggl(
\eps^{-1}\E\sbra{\|R\|_\infty}
+\eps^{-1}(\log p)^{3/2}\E\sbra{\norm{\frac{1}{2}\E[(\Lambda^{-1}D)D^\top\mid W]-\Sigma}_\infty}\\
&\qquad\quad+\eps^{-3}(\log p)^{7/2}\E\sbra{\max_{j,k,l,m\in[p]}\E[|(\Lambda^{-1}D)_jD_kD_lD_m|\mid W]}\\
&\qquad\quad+\eps^{-4}(\log p)^{3}\E\sbra{\max_{j,k,l,m\in[p]}|(\Lambda^{-1}D)_jD_kD_lD_m|1_{\|D\|_\infty>\beta^{-1}}}
+\eps\sqrt{\log p}
\Biggr).
}
A simple computation shows that \eqref{eq:gech} implies a similar bound but replaces $\eps^{-1}$ in the first term and $\eps^{-4}(\log p)^3$ in the fourth term by $\sqrt{\log p}$ and $\eps^{-3}(\log p)^{7/2}$, respectively. 
Since the above bound is trivial if $\eps\sqrt{\log p}\geq1$ due to the last term, our bound is always better.  
\end{rmk}

Although \cref{thm:gexch} is per se new, its proof is essentially a minor modification of the proof of \cite[Lemma A.1]{CCKK22} that concerns sums of independent random vectors. 
The real new problem here is how to bound the quantities that appear on the right-hand side of \eqref{eq:gech}. 
In our application, we regard $X=(X_i)_{i=1}^n$ as a random element taking values in $(S^n,\mcl S^{\otimes n})$ and construct an exchangeable pair $(X,X')$ in a standard way. 
Then we apply \cref{thm:gexch} to $(Y,Y')=(X,X')$ with $G_j=\sum_{s=0}^rs^{-1}\binom{n-s}{r-s}\{J_{s,X'}(\pi_s\psi_j)-J_{s,X}(\pi_s\psi_j)\}$; see Step 1 of the proof of \cref{thm:clt-ustat} for details. 
We remark that \cite{dobler2023normal} has employed essentially the same construction to obtain 1-Wasserstein bounds in the univariate case. 
To bound the main term of $\E[\|V^\eps\|_\infty]$, which is $\E[\|V\|_\infty]$ with $V$ defined by \eqref{def-V}, we will utilize the fact that we can explicitly write down the Hoeffding decomposition of $\E[GD^\top\mid X]$ thanks to Proposition 2.6 in \cite{DoPe19} and Lemma 3.3 in \cite{DoPe17} (see Eq.\eqref{v-hoef}). Then, we can invoke sharp maximal inequalities for $U$-statistics developed in the next subsection to bound $\E[\|V\|_\infty]$ (see \cref{thm:max-is}). The detailed computation is found in Step 2 of the proof of \cref{thm:clt-ustat}. 
Meanwhile, the treatment of the remaining terms is more involved. 
For the case $r=2$, we develop a sufficiently sharp maximal inequality tailored to the present situation; see \cref{influence-mom}. 
For the general case, it seems hard to directly bound the remaining terms, especially $\E[\|\Gamma^\eps\|_\infty]$. 
For this reason, we instead use the following simplified bound. 
\begin{corollary}\label{coro:exch}
Under the assumptions of \cref{thm:gexch}, there exists a universal constant $C'>0$ such that 
\bm{
\sup_{A\in\mathcal{R}_p}\left|\pr(W\in A)-\pr(Z\in A)\right|\\
\leq \frac{C'}{\ul\sigma}\Biggl(
\E\sbra{\|R\|_\infty}\sqrt{\log p}
+\sqrt{\E\sbra{\|V\|_\infty}}\log p
+\bra{\E[\|G\|_\infty\|D\|_\infty^3]}^{1/4}(\log p)^{5/4}
\Biggr),
}
where
\ben{\label{def-V}
V:=\frac{1}{2}\E[GD^\top\mid Y]-\Sigma.
}
\end{corollary}
In our application, the first term of the above bound vanishes. 
Since $\E[\|G\|_\infty\|D\|_\infty^3]$ is essentially a maximal moment of degenerate $U$-statistics, we can again invoke \cref{thm:max-is} to bound it.

\subsection{Maximal inequalities} \label{subsec:maximal_inequality}

In order to obtain sharp bounds for quantities appearing in \cref{thm:gexch} and \cref{coro:exch}, we need to extend Lemmas 8 and 9 in \cite{CCK15} in two directions. These extensions would be of independent interest. 

The first direction is extensions to $U$-statistics. 
\begin{theorem}\label{thm:max-is}
Let $q\geq1$ and $\psi_j\in L^{q\vee2}(P^r)$ $(j\in[p])$ be degenerate, symmetric kernels of order $r\geq1$. 
Then there exists a constant $C_r$ depending only on $r$ such that
\ben{\label{eq:max-is-deg}
\norm{\max_{j\in[p]}|J_r(\psi_j)|}_{L^q(\pr)}
\leq C_r\max_{0\leq s\leq r}n^{\frac{r-s}{2}}(q+\log p)^{\frac{r+s}{2}}\norm{\max_{j\in[p]}M\bra{P^{r-s}(\psi_j^2)}}_{L^{1\vee\frac{q}{2}}(\pr)}^{1/2}.
}
\end{theorem}

\begin{theorem}\label{lem:max-is}
Let $q\geq1$ and $\psi_j\in L^{q}(P^r)$ $(j\in[p])$ be non-negative, symmetric kernels of order $r\geq0$. 
Then there exists a constant $c_r\geq1$ depending only on $r$ such that
\ben{\label{eq:max-is}
\norm{\max_{j\in[p]}J_r(\psi_j)}_{L^q(\pr)}
\leq c_r\max_{0\leq s\leq r}n^{r-s}(q+\log p)^s\norm{\max_{j\in[p]}M(P^{r-s}\psi_j)}_{L^q(\pr)}.
}
\end{theorem}
We can also regard these inequalities as extensions of Corollaries 2 and 1 in \cite{IbSh02} to maximal inequalities; see \cref{rmk:is}. 
\begin{rmk}[Comparison to \cite{ChKa20}] \label{rem:comparison_maximal}
\cite{ChKa20} have developed local maximal inequalities for $U$-processes indexed by general function classes satisfying certain uniform covering number conditions. 
Their results are particularly applicable to the finite function class $\mcl F:=\{\psi_1,\dots,\psi_p\}$. 
Specifically, since $\mcl F$ is VC type with characteristics $(p,1)$ for envelope $\max_{j\in[p]}|\psi_j|$ in the sense of \cite[Definition 2.1]{ChKa20}, under the assumptions of \cref{thm:max-is}, Corollary 5.5 in \cite{ChKa20} gives the following bound:
\ben{\label{ck-bound}
\E\sbra{\sup_{j\in[p]}|J_r(\psi_j)|}
\leq C_r\bra{n^{\frac{r}{2}}\sup_{j\in[p]}\|\psi_j\|_{L^2(P^r)}\log^{r/2}(np)+n^{\frac{r-1}{2}}\|M_r\|_{L^2(\pr)}\log^r(np)},
} 
where $M_r:=\max_{1\leq i\leq\lfloor n/r\rfloor}\max_{j\in[p]}|\psi_j(X_{(i-1)r+1},\dots,X_{ir})|$. 
Using \cite[Proposition 3]{kontorovich2023decoupling} and Jensen's inequality, one can show
\[
\max_{1\leq s\leq r}n^{\frac{r-s}{2}}(\log p)^{\frac{r+s}{2}}\norm{\max_{j\in[p]}M\bra{P^{r-s}(\psi_j^2)}}_{L^{1}(\pr)}^{1/2}\leq C_rn^{\frac{r-1}{2}}\|M_r\|_{L^2(\pr)}\log^r(np),
\]
so \eqref{eq:max-is-deg} with $q=1$ refines \eqref{ck-bound}. 
In applications, $\|M_r\|_{L^2(\pr)}$ is often comparable to $\|\max_{j\in[p]}M(\psi_j)\|_{L^2(\pr)}$, and the order of their coefficients improves from $O(n^{(r-1)/2}\log^r(np))$ in \eqref{ck-bound} to $O(\log^rp)$ in \eqref{eq:max-is-deg}. 
\end{rmk}

\begin{rmk}[(Sub-)optimality of the bounds]\label{rmk:is}
Since 
\[
\norm{\max_{j\in[p]}M(P^{r-s}\psi_j)}_{L^q(\pr)}\leq n^{s/q}\norm{\max_{j\in[p]}|P^{r-s}\psi_j|}_{L^q(\pr)}
\]
by \eqref{eq:max-out}, the bounds of Theorems \ref{thm:max-is} and \ref{lem:max-is} have the same dependence on $n$ and $\psi_j$ as those of Corollaries 2 and 1 in \cite{IbSh02}, respectively. 
Since the latter results are two-sided, our bound has a correct dependence on $n$ and $\psi_j$ in this sense.  
On the other hand, the dependence on $p$ and $q$ would be sub-optimal. 
For example, in the bound of \cref{thm:max-is}, the coefficient of the standard deviation component $n^{r/2}\max_{j\in[p]}\|\psi_j\|_{L^2(P^r)}$ is $(q+\log p)^{r/2}$, which should be $\sqrt{q+\log p}$ in view of the central limit theorem. 
In fact, when $r=2$ and $\psi_j$ are bounded, we can presumably derive a refined maximal inequality from \cite[Corollary 3.4]{gine2000exponential}. See also \cite{adamczak2006moment} and \cite{ChKu25} for extensions of this result to the cases of $r>2$ and sub-Weibull kernels, respectively. 
\end{rmk}

The second direction is extensions to martingales and non-negative adapted sequences. 
\begin{lemma}\label{max-rosenthal}
Let $(\xi_i)_{i=1}^N$ be a martingale difference sequence in $\mathbb R^p$ with respect to a filtration $\mathbf G=(\mcl G_i)_{i=0}^N$. 
There exists a universal constant $C$ such that
\ba{
&\norm{\max_{j\in[p]}\max_{n\in[N]}\abs{\sum_{i=1}^n\xi_{ij}}}_{L^m(\pr)}\\
&\leq C \bra{\norm{\max_{j\in[p]}\sqrt{\sum_{i=1}^N\E[\xi_{ij}^2\mid\mcl G_{i-1}]}}_{L^m(\pr)}\sqrt{m+\log p}
+\norm{\max_{i\in[N]}\|\xi_i\|_\infty}_{L^m(\pr)}(m+\log p)}
}
for any $m\geq1$. 
\end{lemma}

\begin{lemma}\label{lem:nonneg-ada}
Let $(\eta_i)_{i=1}^N$ be a sequence of random vectors in $\mathbb R^p$ adapted to a filtration $(\mcl G_i)_{i=1}^N$. Suppose that $\eta_{ij}\geq0$ and $\eta_{ij}\in L^1(\pr)$ for all $i\in[N]$ and $j\in[p]$. 
Then there exists a universal constant $C$ such that
\[
\E\sbra{\max_{j\in[p]}\sum_{i=1}^N\eta_{ij}}
\leq C\bra{\E\sbra{\max_{j\in[p]}\sum_{i=1}^N\E[\eta_{ij}\mid\mcl G_{i-1}]}
+\E\sbra{\max_{i\in[N]}\max_{j\in[p]}\eta_{ij}}\log p},
\]
where we set $\mcl G_0:=\{\emptyset,\Omega\}$.
\end{lemma}

We will use these inequalities to obtain the following estimates. They play a crucial role in the proof of \cref{thm:order-2}. 
\begin{lemma}\label{influence-mom}
Let $\psi_j\in L^4(P^2)$ $(j\in[p])$ be degenerate, symmetric kernels of order $2$. 
There exists a universal constant $C$ such that
\bmn{\label{influence-mom-1}
\E\sbra{\max_{i\in[n]}\max_{j\in[p]}\int_S\abs{\sum_{i'\in[n]:i'< i}\psi_j(X_{i'},x)}^4P(dx)}\\
\leq C\Biggl(n^2\max_{j\in[p]}\|P(\psi_j^2)\|_{L^{2}(P)}^{2}\log^{2}p
+n\max_{j\in[p]}\|\psi_j\|_{L^4(P^2)}^4\log^{3}p
+\E\sbra{\max_{j\in[p]}M\bra{P(\psi_{j}^4)}}\log^4 p
\Biggr)
}
and
\bmn{\label{influence-mom-2}
\E\sbra{\max_{j\in[p]}\sum_{i=1}^n\abs{\sum_{i'\in[n]:i'\neq i}\psi_j(X_{i'},X_i)}^4}\\
\leq C\Biggl(n^3\max_{j\in[p]}\|P(\psi_j^2)\|_{L^{2}(P)}^{2}\log^{2}p
+n^2\max_{j\in[p]}\|\psi_j\|_{L^4(P^2)}^4\log^{3}p
+n\E\sbra{\max_{j\in[p]}M\bra{P(\psi_{j}^4)}}\log^4 p
\\
+n^{2}\E\sbra{\max_{j\in[p]}M\bra{P(\psi_j^2)}^{2}}\log^{3} (np)
+\E\sbra{\max_{j\in[p]}M(\psi_j)^4}\log^5(np)
\Biggr).
}
\end{lemma}

\subsection{Proof of Theorem \ref{thm:clt-ustat}} \label{subsec:proof_clt_ustat}

For $\mathbf i=(i_1,\dots,i_r)\in I_{n,r}$, we write $X_{\mathbf i}=(X_{i_1},\dots,X_{i_r})$ for short. 
The following technical lemma is useful to simplify some estimates. 
\begin{lemma}\label{lem:mp-reduce}
    Let $\psi_j\in L^1(P^r)$ ($j\in[p]$) be symmetric kernels of order $r\geq1$. For any $1\leq l\leq r$,
\ben{\label{eq:mp-reduce}
\E\sbra{\max_{j\in[p]}M(P^l(\psi_j))}\leq \frac{r!}{(r-l)!}\E\sbra{\max_{j\in[p]}M(\psi_j)}.
}
\end{lemma}

\begin{proof}
Since $P^l\psi_j=P(P^{l-1}\psi_j)$, 
the claim for general $l$ follows from repeated applications of the claim for $l=1$. Hence, it suffices to consider the case $l=1$. 
Moreover, with $\psi^*:=\max_{j\in[p]}|\psi_j|$, we have $\max_{j\in[p]}M(P(\psi_j))\leq M(P^l(\psi^*))$ and $\max_{j\in[p]}M(\psi_j)=M(\psi^*)$; hence we may also assume $p=1$ and $\psi_1\geq0$ without loss of generality. 

Under the above assumptions, we shall prove \eqref{eq:mp-reduce} by induction on $r$. When $r=1$,
\ba{
\E[M(P(\psi_1))]
=P(\psi_1)=\E[\psi_1(X_1)]\leq\E[M(\psi_1)],
}
so \eqref{eq:mp-reduce} holds. 
Next, suppose that $r>1$ and \eqref{eq:mp-reduce} holds for any symmetric kernel $\psi_1$ of order less than $r$. 
Classifying whether an $r$-tuple $(i_1,\dots,i_r)\in I_{n,r}$ contains $n$ or not, we bound $M(P(\psi_1))$ as
\ban{
\E[M(P(\psi_1))]
&\leq\E\sbra{\max_{\mathbf i\in I_{n-1,r-1}}P(\psi_1)(X_{\mathbf i})}
+\E\sbra{\max_{\mathbf i\in I_{n-1,r-2}}P(\psi_1)(X_{\mathbf i},X_n)}\notag\\
&=:I+II,\label{eq:mp-reduce-1}
}
where we interpret $\max_{\mathbf i\in I_{i-1,r-2}}P(\psi_1)(X_{\mathbf i},\cdot)$ as $P(\psi_1)(\cdot)$ when $r=2$. 
Since $X_{n}$ is independent of $\mcl G:=\sigma(X_1,\dots,X_{n-1})$, we have
\ban{
I&=\E\sbra{\max_{\mathbf i\in I_{n-1,r-1}}\E[\psi_1(X_{\mathbf i},X_{n})\mid \mcl G]}
\leq\E\sbra{\max_{\mathbf i\in I_{n-1,r-1}}\psi_1(X_{\mathbf i},X_{n})}
=\E[M(\psi_1)],\label{eq:mp-reduce-2}
}
where the inequality is by Jensen's inequality. 
Meanwhile, we can rewrite $II$ as
\ba{
II=\int_S\E\sbra{\max_{\mathbf i\in I_{n-1,r-2}}P(\psi_1)(X_{\mathbf i},x)}P(dx).
}
Applying the assumption of the induction to the kernel $S^{r-1}\ni \mathbf y\mapsto\psi_1(\mathbf y,x)\in\mathbb R$ for $P$-a.s.~$x\in S$ gives
\ban{
II&\leq(r-1)\int_S\E\sbra{\max_{\mathbf i\in I_{n-1,r-1}}\psi_1(X_{\mathbf i},x)}P(dx)
=(r-1)\E\sbra{\max_{\mathbf i\in I_{n-1,r-1}}\psi_1(X_{\mathbf i},X_n)}\notag\\
&=(r-1)\E[M(\psi_1)],\label{eq:mp-reduce-3}
}
where the first equality follows from the fact that $(X_{\mathbf i})_{\mathbf i\in I_{n-1,r-1}}$ is independent of $X_n$. 
Combining \eqref{eq:mp-reduce-1}--\eqref{eq:mp-reduce-3} gives \eqref{eq:mp-reduce}. 
\end{proof}

\begin{proof}[Proof of \cref{thm:clt-ustat}]
Let $\varphi_j:=\psi_j/\sigma_j$ for $j\in[p]$ and set $\tilde W:=(J_r(\varphi_1)-\E[J_r(\varphi_1)],\dots,J_r(\varphi_p)-\E[J_r(\varphi_p)])^\top$. 
Then we have
\[
\sup_{A\in\mathcal{R}_p}\left|\pr(W\in A)-\pr(Z\in A)\right|
=\sup_{A\in\mathcal{R}_p}\left|\pr(\tilde W\in A)-\pr(\tilde Z\in A)\right|,
\]
where $\tilde Z\sim N(0,\Cov[\tilde W])$. 
Also, observe that $\Delta_1(a,b)$ and $\Delta_2(a)$ corresponding to $\varphi_j$ are the same as those corresponding to $\psi_j$, respectively. 
Consequently, replacing $\psi_j$ by $\varphi_j$, we may assume $\sigma_j=1$ for all $j\in[p]$ without loss of generality. 

For the rest of the proof, we proceed in three steps.
\vspace{-5mm}

\paragraph{Step 1.} 
Regarding $X=(X_i)_{i=1}^n$ as a random element taking values in the measurable space $(E,\mcl E)=(S^n,\mcl S^{\otimes n})$, we are going to apply \cref{coro:exch} to 
\[
\mathsf{W}(X):=\bra{J_{r,X}(\psi_{1})-\E[J_{r}(\psi_{1})],\dots,J_{r,X}(\psi_{p})-\E[J_{r}(\psi_{p})]}^\top.
\]
For this purpose, we need to construct an appropriate exchangeable pair $(X,X')$ and an antisymmetric function $\mathsf{G}$. 
Let $X^*=(X^*_i)_{i=1}^n$ be an independent copy of $X=(X_i)_{i=1}^n$. 
Also, let $\alpha$ be a random index uniformly distributed on $[n]$ and such that $X,X^*$ and $\alpha$ are independent. 
Then, define $X'=(X'_i)_{i=1}^n$ as $X_i':=X_i^*$ if $i=\alpha$ and $X_i':=X_i$ otherwise.
It is well-known that $(X,X')$ is an exchangeable pair. 
In addition, define a random vector $G=\mathsf{G}(X,X')$ in $\mathbb R^p$ as $G_j:=n\sum_{s=1}^rs^{-1}D_{j,s}$ for $j=1,\dots,p$, where
\[
D_{j,s}:=J_{s,X'}(\psi_{j,s})-J_{s,X}(\psi_{j,s})
\quad\text{with }
\psi_{j,s}:=\binom{n-s}{r-s}\pi_s\psi_j.
\]
$\mathsf{G}$ is antisymmetric by construction. 
Moreover, \eqref{eq:hoef-decomp} and Lemma 3.2 in \cite{DoPe17} give
\[
\E[G\mid X]=-W.
\]
Therefore, applying \cref{coro:exch} with $\Sigma=\Cov[W]$, we obtain 
\ben{\label{exch-applied}
\sup_{A\in\mathcal{R}_p}\left|\pr(W\in A)-\pr(Z\in A)\right|
\lesssim 
\sqrt{\E\sbra{\|V\|_\infty}}\log p
+\bra{\E[\|G\|_\infty\|D\|_\infty^3]}^{1/4}(\log p)^{5/4},
}
where $V$ and $D$ are defined in the same way as in \cref{coro:exch} with $(Y,Y')$ replaced by $(X,X')$. 
In Steps 2 and 3, we will show
\ban{
\E\sbra{\|V\|_\infty}
&\leq C_r\max_{a,b\in[r]}\Delta_{1}(a,b),\label{eq:v-bound}\\
\E[\|G\|_\infty\|D\|_\infty^3]
&\leq C_r\max_{a\in[r]}\Delta_2(a).\label{eq:d4-bound}
}
Inserting these bounds into \eqref{exch-applied} gives the desired result.
\vspace{-5mm}

\paragraph{Step 2.} 
In this step, we prove \eqref{eq:v-bound}. 
For $j,k\in[p]$, observe that
\[
G_jD_k=n\sum_{a,b=1}^ra^{-1}D_{j,a}D_{k,b}.
\]
For $a,b\in[r]$, $J_a(\psi_{j,a})J_b(\psi_{k,b})$ has the following Hoeffding decomposition by Proposition 2.6 in \cite{DoPe19}:
\[
J_a(\psi_{j,a})J_b(\psi_{k,b})=\sum_{t=0}^{2(a\wedge b)}J_{a+b-t}(\chi^{(j,k)}_{a+b-t}),
\]
where, for $t\in[2(a\wedge b)]$, 
\[
\chi^{(j,k)}_{a+b-t}=\chi^{(j,k,a,b)}_{a+b-t}
:=\sum_{s=\lceil t/2\rceil}^{t\wedge a\wedge b}\binom{n-a-b+t}{t-s}\binom{a+b-t}{a-s,b-s,2s-t}\pi_{a+b-t}\wt{\bra{\psi_{j,a}\star^{t-s}_s\psi_{k,b}}}
\]
is a degenerate, symmetric kernel. 
Hence, by Lemma 3.3 in \cite{DoPe17}
\ben{\label{v-hoef}
n\E[D_{j,a}D_{k,b}\mid X]=\sum_{t=1}^{2(a\wedge b)}tJ_{a+b-t}(\chi^{(j,k)}_{a+b-t}).
}
In addition,
\[
\E[W_jW_k]=\sum_{a=1}^r\E[J_a(\psi_{j,a})J_a(\psi_{k,a})]
=\sum_{a=1}^r\E[J_0(\chi^{(j,k,a,a)}_{0})].
\]
Consequently, we obtain
\ba{
2V_{jk}
&=\sum_{a=1}^ra^{-1}\sum_{t=1}^{2a-1}tJ_{2a-t}(\chi^{(j,k)}_{2a-t})
+\sum_{1\leq a<b\leq r}(a^{-1}+b^{-1})\sum_{t=1}^{2(a\wedge b)}tJ_{a+b-t}(\chi^{(j,k)}_{a+b-t}),
}
and thus
\ba{
\E[\|V\|_\infty]
&\leq C_r\bra{\sum_{a=1}^r\sum_{t=1}^{2a-1}\E\sbra{\max_{j,k\in[p]}|J_{2a-t}(\chi^{(j,k)}_{2a-t})|}
+\sum_{1\leq a<b\leq r}\sum_{t=1}^{2(a\wedge b)}\E\sbra{\max_{j,k\in[p]}|J_{a+b-t}(\chi^{(j,k)}_{a+b-t})|}}.
}
To bound the summands on the right-hand side, we are going to apply \cref{thm:max-is}. 
By the triangle inequality and \cite[Lemma 2.9]{DoPe19},
\ba{
|\chi^{(j,k)}_{a+b-t}|
&\leq C_{r}\frac{\sqrt{\binom{n}{a}}\sqrt{\binom{n}{b}}}{\sqrt{\binom{n}{a+b-t}}}\sum_{s=\lceil t/2\rceil}^{t\wedge a\wedge b}n^{t/2-s}|\pi_{a+b-t}\wt{\bra{\psi_{j,a}\star^{t-s}_s\psi_{k,b}}}|\\
&\leq C_{r}\sum_{s=\lceil t/2\rceil}^{t\wedge a\wedge b}n^{t-s}|\pi_{a+b-t}\wt{\bra{\psi_{j,a}\star^{t-s}_s\psi_{k,b}}}|.
}
Hence, for any $0\leq u\leq a+b-t$,
\ba{
&\E\sbra{\max_{j,k\in[p]}M\bra{P^{a+b-t-u}(|\chi^{(j,k)}_{a+b-t}|^2)}}\\
&\leq C_{r}\sum_{s=\lceil t/2\rceil}^{t\wedge a\wedge b}n^{2(t-s)}\E\sbra{\max_{j,k\in[p]}M\bra{P^{a+b-t-u}(|\pi_{a+b-t}\wt{\bra{\psi_{j,a}\star^{t-s}_s\psi_{k,b}}}|^2)}}\\
&\leq C_{r}\sum_{s=\lceil t/2\rceil}^{t\wedge a\wedge b}n^{2(t-s)}\E\sbra{\max_{j,k\in[p]}M\bra{P^{a+b-t-u}(|\wt{\psi_{j,a}\star^{t-s}_s\pi_b\psi_{k,b}}|^2)}},
}
where the last inequality follows from \eqref{eq:hoef-proj}, Jensen's inequality and \cref{lem:mp-reduce}. 
Hence, we obtain by \cref{thm:max-is}
\ba{
&\E\sbra{\max_{j,k\in[p]}|J_{a+b-t}(\chi^{(j,k)}_{a+b-t})|}\\
&\leq C_r\sum_{s=\lceil t/2\rceil}^{t\wedge a\wedge b}\max_{0\leq u\leq a+b-t}n^{\frac{a+b+t-u}{2}-s}(\log p)^{\frac{a+b-t+u}{2}}\sqrt{\E\sbra{\max_{j,k\in[p]}M\bra{P^{a+b-t-u}(|\wt{\psi_{j,a}\star^{t-s}_s\psi_{k,b}}|^2)}}}.
}
Noting that $|\psi_{j,s}|\leq n^{r-s}|\pi_s\psi_j|$, we deduce
\ba{
\E[\|V\|_\infty]
&\leq C_r\sum_{a=1}^r\sum_{t=1}^{2a-1}\sum_{s=\lceil t/2\rceil}^{t\wedge a}\max_{0\leq u\leq 2a-t}\Delta_1(a,a;s,t-s,u)\\
&\quad+C_r\sum_{1\leq a<b\leq r}\sum_{t=1}^{2(a\wedge b)}\sum_{s=\lceil t/2\rceil}^{t\wedge a\wedge b}\max_{0\leq u\leq a+b-t}\Delta_1(a,b,s,t-s,u)\\
&\leq C_r\max_{a,b\in[r]}\Delta_1(a,b).
}
\vspace{-5mm}

\paragraph{Step 3.} 
It remains to prove \eqref{eq:d4-bound}. 
Since
\ba{
\E[\|G\|_\infty\|D\|_\infty^3]
\leq n\E\sbra{\max_{j\in[p]}\bra{\sum_{a=1}^r|D_{j,a}|}^4}
\leq C_r\max_{a\in[r]}n\E\sbra{\max_{j\in[p]}D_{j,a}^4},
}
it suffices to prove
\ben{\label{aim:d4-bound}
n\E\sbra{\max_{j\in[p]}D_{j,a}^4}
\leq C_{r}\Delta_2(a)
}
for all $a\in[r]$. 
Observe that
\ba{
D_{j,a}
&=\frac{1}{(a-1)!}\sum_{\begin{subarray}{c}
\mathbf i=(i_1,\dots,i_{a-1})\in I_{n,a-1}\\
i_s\neq\alpha\text{ for all }s\in[a-1]
\end{subarray}}\cbra{\psi_{j,a}(X_{\mathbf i},X_\alpha^*)-\psi_{j,a}(X_{\mathbf i},X_\alpha)}.
}
Therefore, noting the fact that $X_1,\dots,X_n$ are i.i.d., we obtain
\ban{
\E\sbra{\max_{j\in [p]}D_{j,a}^4}
&\leq\frac{16}{n}\sum_{i=1}^n\E\sbra{\max_{j\in [p]}\abs{\frac{1}{(a-1)!}\sum_{\begin{subarray}{c}
\mathbf i=(i_1,\dots,i_{a-1})\in I_{n,a-1}\\
i_s\neq i\text{ for all }s\in[a-1]
\end{subarray}}\psi_{j,a}(X_{\mathbf i},X_i)}^4}\notag\\
&= 16\E\sbra{\max_{j\in [p]}\abs{\frac{1}{(a-1)!}\sum_{\mathbf i\in I_{n-1,a-1}}\psi_{j,a}(X_{\mathbf i},X_n)}^4}.\label{d4-order-r}
}
Observe that conditional on $X_n$,
\[
\frac{1}{(a-1)!}\sum_{\mathbf i\in I_{n-1,a-1}}\psi_{j,a}(X_{\mathbf i},X_n)
\]
is a degenerate $U$-statistic of order $a-1$, based on $(X_i)_{i=1}^{n-1}$. 
Hence, \cref{thm:max-is} gives
\ba{
&\E\sbra{\max_{j\in [p]}\abs{\frac{1}{(a-1)!}\sum_{\mathbf i\in I_{n-1,a-1}}\psi_{j,a}(X_{\mathbf i},X_n)}^4\mid X_n}\\
&\leq C_{r}\max_{0\leq s\leq a-1}n^{2(a-1-s)}(\log p)^{2(a-1+s)}\E\sbra{\max_{j\in[p]}\max_{\mathbf i\in I_{n-1,s}}P^{a-1-s}\bra{\psi_{j,a}^2}(X_{\mathbf i},X_n)^2\mid X_n}.
}
Combining this with \eqref{d4-order-r} and $|\psi_{j,s}|\leq n^{r-s}|\pi_s\psi_j|$ gives \eqref{aim:d4-bound}.
\end{proof}

\subsection{Proof of Corollary \ref{coro:clt-ustat}}
\label{subsec:proof_coro_clt_ustat}

We need the following technical estimate to simplify the first term on the right-hand side of \eqref{eq:clt-ustat}. 
\begin{lemma}\label{lem:delta}
Under the assumptions of \cref{thm:order-2}, there exists a universal constant $C$ such that
\ban{
\Delta_1(1,1)\log^2p&\leq C\sqrt{\Delta_{2,*}(1)\log^5p},\label{lem:delta-eq11}\\
\Delta_1(2,2)\log^2p&\leq C\bra{\Delta_1^{(0)}\log^3 p
+\sqrt{\Delta_{2,*}(2)\log^5p}},\label{lem:delta-eq22}
}
and
\bmn{
\Delta_1(1,2)\log^2p\leq
C\Biggl(
\Delta_{1}^{(1)}\log^{5/2}p
+n^{3/2}\max_{j,k\in[p]}\frac{\|\pi_1\psi_j\|_{L^2(P)}}{\sigma_j}\bra{\Delta_{2,*}^{(5)}(2)\log^{9}p}^{1/4}\\
+\sqrt{\bra{\Delta_{2,*}(1)+\Delta_{2,*}(2)}\log^5p}
\Biggr),\label{lem:delta-eq12}
}
where
$\Delta_{2,*}(1):=\sum_{\ell=1}^2\Delta_{2,*}^{(\ell)}(1)$ and 
$\Delta_{2,*}(2):=\sum_{\ell=1}^5\Delta_{2,*}^{(\ell)}(2)$. 
\end{lemma}
The proof of this lemma is deferred to \cref{proof:delta}. 

\begin{proof}[Proof of \cref{coro:clt-ustat}]
First, for any $f\in L^m(P)$ with $m\geq1$, we have $\|f\|_{L^m(P)}^m=\E[|f(X_1)|^m]\leq\E[M(f)^m]$. 
Hence we have $\Delta_{2,*}^{(1)}(1)\leq\Delta_2(1)$ and $\Delta_{2,*}^{(2)}(2)\leq\Delta_2(2;0)$. 
In particular,
\ba{
\{\Delta_2(1)+\Delta_{2,*}^{(1)}(2)\}\log^5p
&\geq\max_{j\in[p]}\frac{n^5\|\pi_1\psi_j\|_{L^2(P)}^4+n^2\|\pi_2\psi_j\|_{L^2(P^2)}^4}{\sigma_j^4}\log^5p\\
&\geq\max_{j\in[p]}\frac{\bra{n^3\|\pi_1\psi_j\|_{L^2(P)}^2+n^2\|\pi_2\psi_j\|_{L^2(P^2)}^2}^2}{\sigma_j^4}\frac{\log^5p}{2n^2}
\geq\frac{\log^5p}{2n^2},
}
where the last inequality follows by \eqref{u-anova}. 
Thus, the claim asserted is trivial if $\log p>n$; hence it suffices to consider the case $\log p\leq n$. 
In this case, we have $\Delta_{2,*}^{(2)}(1)\leq\Delta_2(1)$, $\Delta_{2,*}^{(5)}(2)\leq\Delta_2(2;0)$ and $\Delta_{2,*}^{(4)}(2)\leq\Delta_2(2;1)$ by definition. 
Meanwhile, \cref{lem:mp-reduce} gives $\Delta_{2,*}^{(3)}\leq2\Delta_{2}(2;1)$. 
Therefore, \cref{lem:delta} gives 
\ba{
\max_{a,b\in[2]}\Delta_2(a,b)\lesssim \Delta_1'+\sqrt{\cbra{\Delta_2(1)+\Delta_2(2)+\Delta_{2,*}^{(1)}(2)}\log^5p}.
}
Inserting this bound into \eqref{eq:clt-ustat} gives the desired result. 
\end{proof}

\subsection{Proof of Theorem \ref{thm:order-2}}
\label{subsec:proof_thm_order-2}

In this proof, we use the same notation as in the proof of \cref{thm:clt-ustat}. 
First, by the same reasoning as in the proof of \cref{thm:clt-ustat}, we may assume $\sigma_j=1$ for all $j\in[p]$ without loss of generality. 
Next, since $n^{1/\log n}=e\leq en^{1/q}$ and the $L^q$-norm with respect to a probability measure is non-decreasing in $q\in[1,\infty]$, the asserted claim for $q>\log n$ follows from the one for $q=\log n$. 
Hence, we may assume $q\leq\log n$ without loss of generality. 

Now, using \cref{thm:gexch} instead of \cref{coro:exch} in the proof of \cref{thm:clt-ustat}, we obtain for any $\eps>0$
\ba{
&\sup_{A\in\mathcal{R}_p}\left|\pr(W\in A)-\pr(Z\in A)\right|\\
&\lesssim 
\E\sbra{\|R^\eps\|_\infty}\sqrt{\log p}
+\eps^{-1}\E\sbra{\|V^\eps\|_\infty}(\log p)^{3/2}
+\eps^{-3}\E\sbra{\Gamma^\eps}(\log p)^{7/2}
+\eps\sqrt{\log p},
}
where $R^\eps,V^\eps$ and $\Gamma^\eps$ are defined in the same way as in \cref{thm:gexch} with $R=0$ and $(Y,Y')$ replaced by $(X,X')$. 
Since
\[
\|R^\eps\|_\infty
\leq n\max_{j\in[p]}\sum_{s=1}^2\E[|D_{j,s}|1_{\{\|D\|_\infty>\beta^{-1}\}}\mid X]
\]
and
\[
\|V^\eps-V\|_\infty
\leq n\max_{j\in[p]}\sum_{s=1}^2\E\sbra{|D_{j,s}|^21_{\{\|D\|_\infty>\beta^{-1}\}}\mid X},
\]
Young's inequality for products gives
\ba{
\|R^\eps\|_\infty
&\leq n\max_{j\in[p]}\sum_{s=1}^2\bra{\frac{\beta^3}{4}\E[|D_{j,s}|^4\mid X]+\frac{3}{4\beta}\E[1_{\{\|D\|_\infty>\beta^{-1}\}}\mid X]}
}
and
\ba{
\|V^\eps-V\|_\infty
&\leq n\max_{j\in[p]}\sum_{s=1}^2\bra{\frac{\beta^3}{2}\E[|D_{j,s}|^4\mid X]+\frac{1}{2\beta^3}\E[1_{\{\|D\|_\infty>\beta^{-1}\}}\mid X]}.
}
Hence we have
\ba{
&\E\sbra{\|R^\eps\|_\infty}\sqrt{\log p}
+\eps^{-1}\E\sbra{\|V^\eps\|_\infty}(\log p)^{3/2}\\
&\lesssim \frac{n\sqrt{\log p}}{\beta\wedge\beta^3}\pr(\|D\|_\infty>\beta^{-1})
+\eps^{-1}\E\sbra{\|V\|_\infty}(\log p)^{3/2}
+\eps^{-3}\E[\Gamma_1+\Gamma_2](\log p)^{7/2},
}
where $\Gamma_s:=n\max_{j\in[p]}\E[|D_{j,s}|^4\mid X]$ for $s=1,2$. 
Also, we have
\ba{
\Gamma^\eps
\leq n\max_{j\in[p]}\E\sbra{\bra{\sum_{s=1}^r|D_{j,s}|}^4\mid X}
\leq 8(\Gamma_1+\Gamma_2).
}
Consequently, we obtain
\besn{\label{order2-aim}
&\sup_{A\in\mathcal{R}_p}\left|\pr(W\in A)-\pr(Z\in A)\right|\\
&\lesssim
\frac{n\sqrt{\log p}}{\beta\wedge\beta^3}\pr(\|D\|_\infty>\beta^{-1})
+\eps^{-1}\E\sbra{\|V\|_\infty}(\log p)^{3/2}
+\eps^{-3}\E\sbra{\Gamma_1+\Gamma_2}(\log p)^{7/2}
+\eps\sqrt{\log p}.
}
In the remaining proof, we will bound the quantities on the right-hand side and then choose $\eps$ appropriately. 
Recall that we already show (cf.~Eq.\eqref{eq:v-bound})
\ben{\label{order2-v-bound}
\E\sbra{\|V\|_\infty}\lesssim\max_{a,b\in[2]}\Delta_1(a,b).
}
Also, by construction
\ben{\label{ass-var}
1=\Var[W_j]=\sum_{s=1}^r\Var[J_s(\psi_{j,s})]=\sum_{s=1}^r\binom{n}{s}\|\psi_{j,s}\|_{L^2(P^s)}^2
}
and
\ben{\label{psi-pi}
|\psi_{j,1}|\leq n|\pi_1\psi_j|,\qquad
|\psi_{j,2}|\leq |\pi_2\psi_j|.
}

\paragraph{Step 1.} 
In this step, we bound $\E[\Gamma_1]$ and $\E[\Gamma_2]$. 
Observe that
\ba{
\Gamma_1
&=\max_{j\in[p]}\sum_{i=1}^n\E\sbra{\abs{\psi_{j,1}(X_{i}^*)-\psi_{j,1}(X_{i})\}}^4\mid X}
\leq8\max_{j\in[p]}\bra{n\|\psi_{j,1}\|_{L^4(P)}^4+\sum_{i=1}^n\psi_{j,1}(X_{i})^4}.
}
By Lemma 9 in \cite{CCK15} and \eqref{eq:max-out},
\ba{
\E\sbra{\max_{j\in[p]}\sum_{i=1}^n\psi_{j,1}(X_{i})^4}
\lesssim n\max_{j\in[p]}\|\psi_{j,1}\|_{L^4(P)}^4+n^{4/q}\norm{\max_{j\in[p]}|\psi_{j,1}|}_{L^{q}(P)}^4\log p.
}
Combining these bounds with \eqref{psi-pi} gives
\ben{\label{gamma1-bound}
\E[\Gamma_1]
\lesssim\Delta_{2,q}(1).
}
Next, observe that
\ba{
\Gamma_2
&=\max_{j\in[p]}\sum_{i=1}^n\E\sbra{\abs{\sum_{i'\in[n]:i'\neq i}\{\psi_{j,2}(X_{i'},X_{i}^*)-\psi_{j,2}(X_{i'},X_{i})\}}^4\mid X}\\
&\leq8\bra{\max_{j\in[p]}\sum_{i=1}^n\E\sbra{\abs{\sum_{i'\in[n]:i'\neq i}\psi_{j,2}(X_{i'},X_{i}^*)}^4\mid X}
+\max_{j\in[p]}\sum_{i=1}^n\abs{\sum_{i'\in[n]:i'\neq i}\psi_{j,2}(X_{i'},X_{i})}^4}\\
&=:8(\Gamma_{2,1}+\Gamma_{2,2}).
}
Since $(X_i^*)_{i=1}^n$ is an i.i.d.~sequence with the common law $P$ and independent of $X$, 
\ba{
\E[\Gamma_{2,1}]
&=\E\sbra{\max_{j\in[p]}\sum_{i=1}^n\int_S\abs{\sum_{i'\in[n]:i'\neq i}\psi_{j,2}(X_{i'},x)}^4P(dx)}\\
&\leq\sum_{i=1}^n\E\sbra{\max_{j\in[p]}\int_S\abs{\sum_{i'\in[n]:i'\neq i}\psi_{j,2}(X_{i'},x)}^4P(dx)}\\
&=n\E\sbra{\max_{j\in[p]}\int_S\abs{\sum_{i=1}^{n-1}\psi_{j,2}(X_{i},x)}^4P(dx)},
}
where the last equality follows from the fact that $(X_i)_{i=1}^n$ is i.i.d. 
Therefore, \cref{influence-mom}, \eqref{eq:max-out} and \eqref{psi-pi} give
\ban{
\E[\Gamma_2]
&\lesssim n^3\max_{j\in[p]}\|P(\psi_{j,2}^2)\|_{L^{2}(P)}^{2}\log^{2}p
+n^2\max_{j\in[p]}\|\psi_{j,2}\|_{L^4(P^2)}^4\log^{3}p\notag\\
&\quad+n\E\sbra{\max_{j\in[p]}M\bra{P(\psi_{j,2}^4)}}\log^4 p
+n^{2+4/q}\norm{\max_{j\in[p]}P(\psi_{j,2}^2)}_{L^{q/2}(P)}^2\log^{3} (np)\notag\\
&\quad+n^{8/q}\norm{\max_{j\in[p]}|\psi_{j,2}|}_{L^q(P^2)}^4\log^5(np)
\leq\Delta_{2,q}(2).\label{gamma2-bound}
}

\paragraph{Step 2.} In this step, we bound $\pr\bra{\|D\|_\infty>\beta^{-1}}$. 
By Markov's inequality,
\ba{
\pr\bra{\|D\|_\infty>\beta^{-1}}
\leq\beta^q\E[\|D\|_\infty^q]
\leq(2\beta)^q\bra{\E\sbra{\max_{j\in[p]}|D_{j,1}|^q}
+\E\sbra{\max_{j\in[p]}|D_{j,2}|^q}}.
}
By definition, \eqref{eq:max-out} and \eqref{psi-pi},
\ba{
\E\sbra{\max_{j\in[p]}|D_{j,1}|^q}
&=\frac{1}{n}\sum_{i=1}^n\E\sbra{\max_{j\in[p]}|\psi_{j,1}(X_i^*)-\psi_{j,1}(X_i)|^q}\\
&\leq2^qn^q\norm{\max_{j\in[p]}|\pi_1\psi_{j}|}_{L^q(P)}^q
\leq2^q\bra{\frac{\Delta_{2,q}^{(2)}(1)^{1/4}}{n^{1/q}(\log p)^{1/4}}}^{q}.
}
Also, noting that $(X_i)_{i=1}^n$ is i.i.d., we have
\ba{
\E\sbra{\max_{j\in[p]}|D_{j,2}|^q}
&=\frac{1}{n}\sum_{i'=1}^n\E\sbra{\max_{j\in[p]}\abs{\sum_{i:i\neq i'}\{\psi_{j,2}(X_{i},X_{i'})-\psi_{j,2}(X_{i},X_{i'}^*)\}}^q}\\
&\leq2^q\E\sbra{\max_{j\in[p]}\abs{\sum_{i=1}^{n-1}\psi_{j,2}(X_{i},X_{n})}^q}.
}
Since $(X_i)_{i=1}^{n-1}$ is centered and independent conditional on $X_n$, \cref{max-rosenthal} together with the assumption $q\leq\log n$ and \eqref{psi-pi} imply that there exists a universal constant $C_1$ such that
\ba{
&\E\sbra{\max_{j\in[p]}|D_{j,2}|^q}\\
&\leq C_1^q\bra{\E\sbra{\bra{\log (np)\max_{j\in[p]}\sum_{i=1}^{n-1}P(\psi_{j,2}^2)(X_{n})}^{q/2}}+\log^q(np)\E\sbra{\max_{i\in[n]}\max_{j\in[p]}|\psi_{j,2}(X_i,X_n)|^q}}\\
&\leq C_1^q\bra{\bra{n\log(n p)\norm{\max_{j\in[p]}P(|\pi_2\psi_{j}|^2)}_{L^{q/2}(P)}}^{q/2}+n\log^q(np)\norm{\max_{j\in[p]}|\pi_2\psi_{j}|}_{L^q(P^2)}^q}\\
&\leq 2C_1^q\bra{\frac{\{\Delta_{2,q}^{(5)}(2)+\Delta_{2,q}^{(4)}(2)\}^{1/4}}{n^{1/q}\log^{1/4}(np)}}^q.
}
Consequently, there exists a universal constant $C_2>0$ such that
\ba{
n\pr\bra{\|D\|_\infty>\beta^{-1}}\leq \eps^{-q}(\log p)^q\bra{C_2\frac{\{\Delta_{2,q}(1)+\Delta_{2,q}(2)\}^{1/4}}{\log^{1/4}(np)}}^q.
}

\paragraph{Step 3.} In this step, we choose the value of $\eps$ appropriately and complete the proof. 
Let
\ba{
\eps=\sqrt{\max_{a,b\in[2]}\Delta_1(a,b)\log p}
+C_2\bra{\bra{\Delta_{2,q}(1)+\Delta_{2,q}(2)}\log^3p}^{1/4}
}
so that
\ben{\label{eps-bound}
\eps^{-1}\max_{a,b\in[2]}\Delta_1(a,b)(\log p)^{3/2}
+\eps^{-3}\bra{\Delta_{2,q}(1)+\Delta_{2,q}(2)}(\log p)^{7/2}
\lesssim\eps\sqrt{\log p}.
}
Also, Step 2 gives 
$
n\pr\bra{\|D\|_\infty>\beta^{-1}}\leq1.
$
If $\beta=\eps^{-1}\log p<1$, then $\eps>1$, and the asserted bound is trivially valid for any $C\geq1$. 
Hence, it suffices to consider the case $\beta\geq1$. 
Then, 
\[
\frac{n\sqrt{\log p}}{\beta\wedge\beta^3}\pr(\|D\|_\infty>\beta^{-1})
\leq\frac{\eps}{\sqrt{\log p}}\leq\eps\sqrt{\log p}.
\]
Combining this with \eqref{order2-aim}--\eqref{order2-v-bound} and \eqref{gamma1-bound}--\eqref{eps-bound} gives 
\ba{
\sup_{A\in\mathcal{R}_p}\left|\pr(W\in A)-\pr(Z\in A)\right|
&\lesssim \sqrt{\max_{a,b\in[2]}\Delta_1(a,b)\log^2 p}+\cbra{\bra{\Delta_{2,q}(1)+\Delta_{2,q}(2)}\log^5 p}^{1/4}.
}
Now the desired result follows by \cref{lem:delta}.\qed

\subsection{Proof of Corollary \ref{coro:order-2}} \label{subsec:proof_coro_order_2}

\begin{lemma}\label{drop-pi}
There exists a universal constant $C$ such that
\ba{
\|\pi_2\psi\star_1^1\pi_2\psi\|_{L^2(P^2)}&\leq C\bra{\|\psi\star_1^1\psi\|_{L^2(P^2)} + \|\psi\|_{L^2(P^2)}\|P\psi\|_{L^2(P)}}
}
for any $\psi\in L^2(P^2)$. 
\end{lemma}

\begin{proof}
By Lemma 5.7 in \cite{dobler2022functional},
\besn{\label{dkp-applied}
&\|\pi_2\psi\star_1^1\pi_2\psi\|_{L^2(P^2)}\\
&\lesssim \max_{a,b\geq0,a+b\leq2}\|P^{2-a}\psi\star^0_0P^{2-b}\psi\|_{L^2(P^{a+b})}
\vee\|P\psi\|_{L^2(P)}^2\vee\|\psi\star^1_1P\psi\|_{L^2(P)}\vee\|\psi\star^1_1\psi\|_{L^2(P^{2})}.
}
Lemma 2.4(v) in \cite{DoPe19} gives
$
\|\psi\star^1_1P\psi\|_{L^2(P)}
\leq \|\psi\|_{L^2(P^2)}\|P\psi\|_{L^2(P)}
$
and
\ba{
\max_{a,b\geq0,a+b\leq2}\|P^{2-a}\psi\star^0_0P^{2-b}\psi\|_{L^2(P^{a+b})}
&\leq\max_{a,b\geq0,a+b\leq2}\|P^{2-a}\psi\|_{L^2(P^a)}\|P^{2-b}\psi\|_{L^2(P^b)}\\
&\leq \|\psi\|_{L^2(P^2)}\|P\psi\|_{L^2(P)},
}
where the last inequality follows by Jensen's inequality. 
Inserting these bounds into \eqref{dkp-applied} gives the desired result. 
\end{proof}

\begin{proof}[Proof of \cref{coro:order-2}]
Again, by the same reasoning as in the proof of \cref{thm:clt-ustat}, we may assume $\sigma_j=1$ for all $j\in[p]$ without loss of generality. 

First, \eqref{eq:hoef-proj}, Jensen's inequality and \cref{lem:mp-reduce} yield $\Delta_{2,q}(1)\lesssim\tilde\Delta_{2,q}(1)$, $\Delta_{2,q}(2)\lesssim\tilde\Delta_{2,q}(2)$ and $\Delta_{2,*}^{(5)}(2)\lesssim\tilde\Delta_{2,q}^{(5)}(2)$. 
Next, observe that $\|\psi_j\|_{L^2(P^2)}^2=\|P(\psi_j^2)\|_{L^1(P)}$. Hence, combining \cref{drop-pi} with the Lyapunov and AM-GM inequalities gives
\ba{
\Delta_1^{(0)}\log^3p
&\lesssim\tilde\Delta_1^{(0)}\log^3p+\frac{n^{3/2}}{2}\|P(\psi_j^2)\|_{L^2(P)}\log^{7/2}p
+\frac{n^{5/2}}{2}\|P\psi_j\|_{L^4(P)}^2\log^{5/2}p\\
&\leq\tilde\Delta_1^{(0)}\log^3p+\sqrt{\cbra{\tilde\Delta_{2,*}^{(2)}(2)+\tilde\Delta_{2,*}^{(1)}(1)}\log^5p}.
}
Third, since $\E[\pi_1\psi_j(X_1)]=0$, inserting the expression \eqref{eq:hoef-proj} in $\pi_2\psi_k$ gives
\ba{
\pi_1\psi_j\star_1^1\pi_2\psi_k(v)
&=\E[\pi_1\psi_j(X_1)\{\psi_k(X_1,v)-P\psi_k(X_1)\}]
=\pi_1\psi_j\star_1^1\psi_k(v)-P(\pi_1\psi_j\star_1^1\psi_k).
}
Hence $\|\pi_1\psi_j\star_1^1\pi_2\psi_k\|_{L^2(P)}\leq\|\pi_1\psi_j\star_1^1\psi_k\|_{L^2(P)}$. 
Thus, Lemma 2.4(vi) in \cite{DoPe19} gives
\ba{
\Delta_1^{(1)}\leq n^{5/2}\max_{j,k\in[p]}\|\pi_1\psi_j\|_{L^2(P)}\|\psi_k\star_1^1\psi_k\|_{L^2(P)}^{1/2}
=n^{3/2}\max_{j\in[p]}\|\pi_1\psi_j\|_{L^2(P)}\sqrt{\tilde\Delta_1^{(0)}}.
}
Finally, observe that $\|\pi_1\psi_j\|_{L^2(P)}^2=\Var[P\psi_j]$ for all $j\in[p]$ by definition. 
Combining these bounds shows that $\sqrt{\Delta_1'}$ is bounded by the right-hand side of \eqref{eq:coro:order-2} up to a universal constant. 
Now, the desired result follows by inserting the obtained bounds into \eqref{eq:order-2}. 
\end{proof}

\section{Proofs for Section \ref{sec:gof}} \label{sec:proof_gof}

Before starting the discussion, we introduce some notation. 
Throughout this section, we abbreviate $\|\cdot\|_{L^q(\mathbb R^d)}$ to $\|\cdot\|_{L^q}$ for $q\in[1,\infty]$. 
Note that $\|K\|_{L^q}<\infty$ for all $q\in[1,\infty]$ under \cref{ass:kernel-gof}. 
For any $g\in L^1(\mathbb R^d)$, we write $P_g$ for the signed measure on $\mathbb R^d$ with density $g$. That is, $P_g(B)=\int_Bg(x)dx$ for any Borel set $B\subset\mathbb R^d$. 
Then, for any symmetric bounded function $\psi:\mathbb R^d\times\mathbb R^d\to\mathbb R$, we define a function $P_g\psi:\mathbb R^d\to\mathbb R$ as $P_g\psi(x)=\int_{\mathbb R^d}\psi(x,y)P_g(dy)$, $x\in\mathbb R^d$. 
For $f\in\pdf_d$ and a symmetric kernel $\psi\in L^1(P_f^2)$, we denote by $\pi_2^f\psi$ the second-order Hoeffding projection of $\psi$ under $P_f$. That is,
\[
\pi_2^f\psi(x,y)=\psi(x,y)-P_f\psi(x)-P_f\psi(y)+P_f^2\psi,\qquad x,y\in\mathbb R^d.
\]
We omit the superscript $f$ when no confusion can arise. 
For every $h>0$, we define a function $K_h:\mathbb R^d\to\mathbb R$ as $K_h(t)=h^{-d}K(t/h)$, $t\in\mathbb R^d$.

\subsection{Proof of Proposition \ref{gof-ga-null}}

For later use, we prove a slightly generalized version of \cref{gof-ga-null}. Set
\[
H^\alpha_{R,b}:=\{f\in \pdf_d:\|f\|_{H^\alpha}\leq R,\|f\|_{L^2}^2\geq b\}
\]
for every $b>0$. 
\begin{proposition}\label{gof-ga}
Let $\alpha>0$ and $b>0$. Under \cref{ass:kernel-gof},
\[
\sup_{f\in H^\alpha_{R,b}}\sup_{t\in\mathbb R}\abs{\pr_{f}\bra{\max_{h\in\mcl H_n}J_2(\pi_2^f\psi_h)\leq t}-\pr_{f}\bra{\max_{h\in\mcl H_n}Z_h\leq t}}\to0\quad\text{as }n\to\infty,
\]
where $Z=(Z_h)_{h\in\mcl H_n}$ is a centered Gaussian random vector such that
\ba{
\E_f[Z_{h}Z_{h'}]=(hh')^{d/2}P_f^2(\pi_2^f\varphi_{h}\pi_2^f\varphi_{h'})
}
for all $f\in H^\alpha_{R,b}$ and $h,h'\in\mcl H_n$. 
\end{proposition}

Since $\hat\psi_h=\pi_2^{f_0}\psi_h$ for every $h>0$, \cref{gof-ga-null} is an immediate consequence of \cref{gof-ga} with $R=\|f_0\|_{H^\alpha}$, $\alpha=\gamma$ and $b=\|f_0\|_{L^2}^2$. 


Turning to the proof of \cref{gof-ga}, we begin by proving a few technical estimates.  
\begin{lemma}\label{lem:sobolev}
Let $f\in L^2(\mathbb R^d)$ satisfy $\|f\|_{H^\alpha}<\infty$ for some $0<\alpha\leq1$. Then, for any $u\in\mathbb R^d$,
\[
\int_{\mathbb R^d}|f(x+u)-f(x)|^2dx\leq2^{2(1-\alpha)}\|f\|_{H^\alpha}^2|u|^{2\alpha}.
\]
\end{lemma}

\begin{proof}
By the Plancherel theorem and the inequality $|e^{\sqrt{-1}t}-1|\leq2\wedge|t|$ for any real number $t$, 
\ba{
\int_{\mathbb R^d}|f(x+u)-f(x)|^2dx
&=\int_{\mathbb R^d}|\fourier{f}(\lambda)\{e^{\sqrt{-1}\lambda\cdot u}-1\}|^2d\lambda\\
&\leq2^{2(1-\alpha)}\int_{\mathbb R^d}|\fourier{f}(\lambda)|^2|\lambda\cdot u|^{2\alpha}d\lambda
\leq2^{2(1-\alpha)}\|f\|_{H^\alpha}^2|u|^{2\alpha}.
}    
This completes the proof. 
\end{proof}

\begin{lemma}\label{lem:gof}
For any $g,g_1\in L^1(\mathbb R^d)\cap L^2(\mathbb R^d)$, $h>0$ and integer $m\geq1$, we have the following:
\begin{enumerate}[label=(\alph*)]

\item\label{gof-infty} $\|P_g(\varphi_h^m)\|_{L^\infty}\leq h^{-(m-1/2)d}\|K\|_{L^{2m}}^m\|g\|_{L^2}$.

\item\label{gof-full} $|\int_{\mathbb R^{d}}|P_g(\varphi_h^m)(x)|P_{g_1}(dx)|\leq h^{-(m-1)d}\|K\|_{L^m}^m\|g\|_{L^2}\|g_1\|_{L^2}$.

\item\label{gof-l2} $\|P_g(\varphi_h^m)\|_{L^2(P_f)}\leq h^{-(m-3/4)d}\|K\|_{L^m}^{m/2}\|K\|_{L^{2m}}^{m/2}\|g\|_{L^2}\|f\|_{L^2}^{1/2}$ for any $f\in\pdf_d$.

\end{enumerate}
\end{lemma}

\begin{proof}
Observe that for any $x\in\mathbb R^d$, 
\ben{\label{gof-int1}
P_g(\varphi_h^m)(x)
=\frac{1}{h^{md}}\int_{\mathbb R^d}K\bra{\frac{x-y}{h}}^mg(y)dy
=\frac{1}{h^{(m-1)d}}\int_{\mathbb R^d}K(u)^mg(x+uh)du.
}
Hence, the Schwarz inequality gives
\ba{
|P_g(\varphi_h^m)(x)|\leq\frac{1}{h^{md}}\sqrt{\int_{\mathbb R^d} K\bra{\frac{x-y}{h}}^{2m}dy}\|g\|_{L^2}
\leq h^{-(m-1/2)d}\|K\|_{L^{2m}}^m\|g\|_{L^2}.
}
This shows \ref{gof-infty}. 
Next, using \eqref{gof-int1} again, we obtain
\ba{
&\abs{\int_{\mathbb R^{2d}}|P_g(\varphi_h^m)(x)|P_{g_1}(dx)}\\
&\leq\frac{1}{h^{(m-1)d}}\int_{\mathbb R^{d}}|K(u)|^m\bra{\int_{\mathbb R^d}|g(x+uh)g_1(x)|dx}du
\leq\frac{\|K\|_{L^m}^m\|g\|_{L^2}\|g_1\|_{L^2}}{h^{(m-1)d}},
}
where the last inequality follows from the Schwarz inequality. 
This shows \ref{gof-full}. 
Finally, since $\|P_g(\varphi_h^m)\|_{L^2(P_f)}^2\leq\|P_g(\varphi_h^m)\|_{L^\infty}\int_{\mathbb R^{d}}|P_g(\varphi_h^m)(x)|P_{f}(dx)$, \ref{gof-infty} and \ref{gof-full} give \ref{gof-l2}. 
\end{proof}

\begin{lemma}\label{mmd-var}
Let $h,h'>0$ and $f\in \pdf_d\cap L^2(\mathbb R^d)$. Then
\besn{\label{mmd-var-eq1}
&(hh')^{d/2}\abs{P_f^2(\pi_2^f\varphi_{h}\pi_2^f\varphi_{h'})-P_f^2(\varphi_{h}\varphi_{h'})}\\
&\leq6\|K\|_{L^1}\|K\|_{L^2}\|f\|_{L^2}^3(h\wedge h')^{d/2}
+(hh')^{d/2}\|K\|_{L^1}^2\|f\|_{L^2}^4.
}
Moreover, there exists a constant $c>0$ depending only on $K$ such that if $\|f\|_{H^\alpha}<\infty$ for some $0<\alpha\leq1$,
\ben{\label{mmd-var-eq2}
\frac{\|K\|_{L^2}^2}{2}\bra{\|f\|_{L^2}^2-c\|f\|_{H^\alpha}^2h^\alpha}
\leq h^d\|\varphi_h\|_{L^2(P_{f}^2)}^2\leq \|K\|_{L^2}^2\|f\|_{L^2}^2.
}
\end{lemma}

\begin{proof}
We may assume $h\geq h'$ without loss of generality. 
A straightforward computation shows
\ba{
&P_f^2(\pi_2\varphi_{h}\pi_2\varphi_{h'})-P_f^2(\varphi_{h}\varphi_{h'})\\
&=-2\E_f[\varphi_h(X_1,X_2)P_f\varphi_{h'}(X_1)]
-2\E_f[\varphi_{h'}(X_1,X_2)P_f\varphi_{h}(X_1)]\\
&\quad+2\E_f[P_f\varphi_h(X_1)P_f\varphi_{h'}(X_1)]
+(P^2_f\varphi_h)(P^2_f\varphi_{h'})\\
&=:2I+2II+2III+IV.
}
\cref{lem:gof}\ref{gof-infty}--\ref{gof-full} give $|III|\leq h^{-d/2}\|K\|_{L^1}\|K\|_{L^2}\|f\|_{L^2}^3$ and $|IV|\leq\|K\|_{L^1}^2\|f\|_{L^2}^4$. 
Meanwhile, by \eqref{gof-int1},
\ba{
|I|
&\leq\int_{\mathbb R^{3d}}|K_{h}(x-y)||K(u)|f(x+uh)f(x)f(y)dudxdy.
}
Using the Schwarz inequality twice, we obtain
\ba{
\int_{\mathbb R^{2d}}f(x+uh)|K_{h}(x-y)|f(x)f(y)dxdy
&\leq\|K_{h}\|_{L^2}\|f\|_{L^2}\int_{\mathbb R^{2d}}f(x+uh)f(x)dx\\
&\leq h^{-d/2}\|K\|_{L^2}\|f\|_{L^2}^3.
}
Thus, $|I|\leq h^{-d/2}\|K\|_{L^1}\|K\|_{L^2}\|f\|_{L^2}^3$. 
Further, another application of \eqref{gof-int1} gives
\ba{
|II|
&\leq\int_{\mathbb R^{3d}}|K_{h}(x-y)||K(u)|f(x+uh')f(x)f(y)dudxdy.
}
Hence the above argument also shows $|II|\leq h^{-d/2}\|K\|_{L^1}\|K\|_{L^2}\|f\|_{L^2}^3$. 
All together, we complete the proof of \eqref{mmd-var-eq1}. 

Next, we prove \eqref{mmd-var-eq2}. 
The upper bound follows from \cref{lem:gof}\ref{gof-full}. 
Meanwhile, since $\int_{|u|\leq a}K(u)^2du\to\|K\|_{L^2}^2$ as $a\to\infty$, there exists a constant $a\geq1$ such that $\int_{|u|\leq a}K(u)^2du\geq\|K\|_{L^2}^2/2$. 
Then, using \eqref{gof-int1}, we obtain
\ba{
h^d\|\varphi_h\|_{L^2(P_f^2)}^2
&\geq\int_{|u|\leq a}\bra{\int_{\mathbb R^{d}}K(u)^2f(x)f(x+uh)dx}du.
}
A similar argument to the derivation of \eqref{ade-var-approx} gives
\ba{
&\abs{\int_{|u|\leq a}\bra{\int_{\mathbb R^{d}}K(u)^2f(x)\{f(x+uh)-f(x)\}dx}du}\\
&\leq 2^{1-\alpha}\|f\|_{H^\alpha}\|f\|_{L^2}\int_{|u|\leq a}K(u)^2|uh|^{\alpha}du
\leq2a\|f\|_{H^\alpha}^2\|K\|_{L^2}^2h^\alpha,
}
where we used $a\geq1$ and $\alpha\leq1$ for the last inequality. 
Consequently, 
\ba{
h^d\|\varphi_h^2\|_{L^2(P_f^2)}^2
&\geq\frac{\|K\|_{L^2}^2}{2}\|f\|_{L^2}^2-2a\|f\|_{H^\alpha}^2\|K\|_{L^2}^2h^\alpha.
}
This gives \eqref{mmd-var-eq2} for $c=4a$.
\end{proof}

\begin{proof}[Proof of \cref{gof-ga}]
Since $H^\alpha_{R,b}\subset H^{\alpha'}_{R,b}$ if $\alpha'\leq\alpha$, we may assume $\alpha<d/4$ without loss of generality. 
We apply \cref{thm:order-2} to $W:=(J_2(\pi_2\psi_h))_{h\in\mcl H_n}$. 
Observe that \cref{mmd-var} gives 
\ben{\label{gof-minvar}
\sup_{f\in H^\alpha_{R,b},h\in\mcl H_n}\|J_2(\pi_2\psi_h)\|_{L^2(\pr_f)}^{-1}
=\sup_{f\in H^\alpha_{R,b},h\in\mcl H_n}(h^{d/2}\|\pi_2\varphi_h\|_{L^2(P_f^2)})^{-1}
=O(1).
}
Then, since $\pi_2\psi_j$ are degenerate, we obtain
\ben{\label{gof-ga-aim}
\sup_{f\in H^\alpha_{R,b}}\sup_{A\in\mcl R_{|\mcl H_n|}}|\pr_f(W\in A)-\pr_f(Z\in A)|\to0,
}
once we verify the following conditions:
\ba{
\delta_0&:=n^2\sup_{f\in H^\alpha_{R,b}}\max_{h\in\mcl H_n}\|\pi_2\psi_h\star_1^1\pi_2\psi_h\|_{L^2(P_f^2)}\log^3|\mcl H_n|\to0,\\
\delta_1&:=n^2\sup_{f\in H^\alpha_{R,b}}\max_{h\in\mcl H_n}\|\pi_2\psi_{h}\|_{L^4(P_f^2)}^4\log^{8}|\mcl H_n|\to0,\\
\delta_2&:=n^3\sup_{f\in H^\alpha_{R,b}}\max_{h\in\mcl H_n}\|P_f(|\pi_2\psi_{h}|^2)\|_{L^{2}(P_f)}^{2}\log^{7}|\mcl{H}_n|\to0,\\
\delta_3&:=n\sup_{f\in H^\alpha_{R,b}}\E_f\sbra{\max_{h\in\mcl H_n}M(P_f(|\pi_2\psi_{h}|^4))}\log^9 |\mcl H_n|\to0,\\
\delta_4&:=\sup_{f\in H^\alpha_{R,b}}\norm{\max_{h\in\mcl H_n}|\pi_2\psi_{h}|}_{L^\infty(P_f^2)}^4(\log^5n)\log^5|\mcl H_n|\to0,\\
\delta_5&:=n^{2}\sup_{f\in H^\alpha_{R,b}}\norm{\max_{h\in\mcl H_n}P_f(|\pi_2\psi_{h}|^2)}_{L^{\infty}(P_f)}^2(\log^{3} n)\log^5|\mcl H_n|\to0.
}
Here, $|\mcl H_n|$ denotes the number of elements in $\mcl H_n$. 
The claim of \cref{gof-ga} follows applying \eqref{gof-ga-aim} to $A=(-\infty,t]^{|\mcl H_n|}$, $t\in\mathbb R$. 

First, since $K$ is bounded, $\delta_4=O(n^{-4}\ul h_n^{-2d}(\log^5n)\log^5|\mcl H_n|)$. 
Next, \cref{lem:gof}\ref{gof-infty} gives $\delta_3=O(n^{-3}\ul h_n^{-3d/2}\log^9|\mcl H_n|)$ and $\delta_5=O(n^{-2}\ul h_n^{-d}(\log^{3}n)\log^5|\mcl H_n|)$. 
Third, \cref{lem:gof}\ref{gof-full} gives $\delta_1=O(n^{-2}\ul h_n^{-d}\log^8|\mcl H_n|)$. 
Fourth, \cref{lem:gof}\ref{gof-l2} yields 
$
\delta_2
=O(n^{-1}\ul h_n^{-d/2}\log^7|\mcl H_n|).
$
Therefore, we have $\delta_\ell\to0$ for all $\ell\in[5]$ by \eqref{ass:bandwidth} and $|\mcl H_n|=O(\log n)$. 

It remains to prove $\delta_0\to0$. \cref{drop-pi} gives $\delta_0\lesssim\delta_{00}+\delta_{01}$, where
\ba{
\delta_{00}&:=n^2\sup_{f\in H^\alpha_{R,b}}\max_{h\in\mcl H_n}\|\psi_h\star_1^1\psi_h\|_{L^2(P_f^2)}\log^3|\mcl H_n|,\\
\delta_{01}&:=n^2\sup_{f\in H^\alpha_{R,b}}\max_{h\in\mcl H_n}\|\psi_h\|_{L^2(P_f^2)}\|P_f\psi_{h}\|_{L^2(P_f)}\log^{3}|\mcl H_n|.
}
\cref{lem:gof}\ref{gof-full}--\ref{gof-l2} yield $\delta_{01}
=O(\bar h_n^{d/4}\log^3|\mcl H_n|)=o(1)
$. 

To bound $\delta_{00}$, fix $f\in H^\alpha_{R,b}$ and $h\in\mcl H_n$ arbitrarily. 
A straightforward computation shows
\ba{
\|\varphi_h\star_1^1\varphi_h\|_{L^2(P_f)}^2
&=\int_{\mathbb R^{4d}}\varphi_h(z,x)\varphi_h(z,y)\varphi_h(w,x)\varphi_h(w,y)f(x)f(y)f(z)f(w)dxdydzdw\\
&=\int_{\mathbb R^{2d}}\varphi_h(0,x)\varphi_h(w,0)I_h(x,w)dxdw,
}
where 
\ba{
I(x,w)&:=\int_{\mathbb R^{2d}}\varphi_h(z,y)\varphi_h(w+y,x+z)f(x+z)f(y)f(z)f(w+y)dydz\\
&=\int_{\mathbb R^{2d}}K_h(z-y)K_h(z-y+x-w)f(x+z)f(y)f(z)f(w+y)dydz.
}
Let $m:=2/(1-2\alpha/d)>2$. Then we have $\|f\|_{L^m}\leq C_{d,\alpha}\|f\|_{H^\alpha}$ by Sobolev's inequality (see e.g.~Theorem 6.5 in \cite{di2012hitchhiker}). 
Hence, with $m':=1/(2-4/m)=d/(4\alpha)$, we have by Young's convolution inequality (see e.g.~Theorem 2.24 in \cite{adams2003sobolev})
\ba{
|I(x,w)|&\leq\bra{\int_{\mathbb R^{d}}f(x+z)^{m/2}f(z)^{m/2}dz}^{2/m}\bra{\int_{\mathbb R^{d}}f(y+w)^{m/2}f(y)^{m/2}dy}^{2/m}\\
&\quad\times\bra{\int_{\mathbb R^{d}}|K_h(t)|^{m'}|K_h(t+x-w)|^{m'}dt}^{1/m'}\\
&\leq C_{d,\alpha}^4\|f\|_{H^\alpha}^4\bra{\int_{\mathbb R^{d}}|K_h(t)|^{m'}|K_h(t+x-w)|^{m'}dt}^{1/m'}.
}
Hence
\ba{
&\|\varphi_h\star_1^1\varphi_h\|_{L^2(P_f)}^2\\
&\leq C_{d,\alpha}^4\|f\|_{H^\alpha}^4\int_{\mathbb R^{2d}}|K_h(x)K_h(w)|\bra{\int_{\mathbb R^{d}}|K_h(t)|^{m'}|K_h(t+x-w)|^{m'}dt}^{1/m'}dxdw\\
&=C_{d,\alpha}^4\|f\|_{H^\alpha}^4\int_{\mathbb R^{2d}}|K(u)K(v)|\bra{\int_{\mathbb R^{d}}|K_h(t)|^{m'}|K_h(t+(u-v)h)|^{m'}dt}^{1/m'}dudv.
}
Using Young's convolution inequality again, we deduce
\ba{
\|\varphi_h\star_1^1\varphi_h\|_{L^2(P_f)}^2
&\leq\|K\|_{L^{m/2}}^2\bra{\int_{\mathbb R^{2d}}|K_h(t)|^{m'}|K_h(t+sh)|^{m'}dtds}^{1/m'}\\
&=h^{-2d+d/m'}\|K\|_{L^{m/2}}^2\|K\|_{L^{m'}}^2.
}
Consequently, 
\ben{\label{gof:star11}
\delta_{00}=\sup_{f\in\mcl H^\alpha_{R,b}}\max_{h\in\mcl H_n}h^d\|\varphi_h\star_1^1\varphi_h\|_{L^2(P_f^2)}\log^3|\mcl H_n|
=O\bra{\bar h_n^{2\alpha}\log^3|\mcl H_n|}=o(1).
}
This completes the proof. 
\end{proof}

\subsection{Proof of Theorem \ref{gof-size}}

\cref{gof-size} is an immediate consequence of the following Gaussian approximation result for $T_n^*$. 
\begin{proposition}\label{gof-boot}
Let $\alpha>0$ and $b>0$. Under \cref{ass:kernel-gof},
\[
\sup_{f\in H^\alpha_{R,b}}\E_f\sbra{\sup_{t\in\mathbb R}\abs{\pr^*(T_n^*\leq t)-\pr_{f}\bra{\max_{h\in\mcl H_n}Z_h\leq t}}}\to0,
\]
where $Z$ is the same as in \cref{gof-ga}.
\end{proposition}
Combining this result with \cref{gof-ga-null}, \cref{mmd-var} and \cite[Proposition 3.2]{koike2019mixed}, we obtain the conclusion of \cref{gof-size}.

The proof of \cref{gof-boot} relies on a Gaussian approximation result for maxima of Gaussian quadratic forms \cite[Theorem 3.1]{koike2019gaussian}. 
Although this result suffices for our purpose, we record a refined version for future reference. 
\begin{lemma}[High-dimensional CLT for Gaussian quadratic forms]\label{gqf}
Let $\zeta$ be a centered Gaussian vector in $\mathbb R^n$. 
Also, for every $j\in[p]$, let $M_j$ be an $n\times n$ symmetric matrix and define a random vector $W$ in $\mathbb R^p$ as $W_j:=\zeta^\top M_{j}\zeta-\E[\zeta^\top M_{j}\zeta]$, $j\in[p]$. 
In addition, let $Z$ be a centered Gaussian vector in $\mathbb R^p$ such that $\ul\sigma:=\min_{j\in[p]}\|Z_j\|_{L^2(\pr)}>0$. 
Then there exists a universal constant $C$ such that
\ba{
\sup_{A\in\mcl R_p}|\pr(W\in A)-\pr(Z\in A)|
\leq\frac{C}{\ul\sigma}\bra{\sqrt{\|\Cov[W]-\Cov[Z]\|_\infty\log^2p}+\bra{\max_{j\in[p]}\kappa_4(W_j)\log^6p}^{1/4}},
}
where $\kappa_4(W_j):=\E[W_j^4]-3(\E[W_j^2])^2$ is the fourth cumulant of $W_j$. 
\end{lemma}

\begin{proof}
In view of Proposition 3.7 in \cite{NPS14}, the desired result follows from the proof of Theorem 3.1 in \cite{koike2019gaussian} once we replace Theorem 2.1 and Corollary 2.1 there by Theorem 3.2 in \cite{CCKK22}. 
\end{proof}

\begin{proof}[Proof of \cref{gof-boot}]
We apply \cref{gqf} to $W:=(J_2^*(\psi_h))_{h\in\mcl H_n}$ conditional on the data. 
Recall that we have \eqref{gof-minvar}. 
Also, note that $\kappa_4(\zeta^\top M\zeta)=48\trace(M^4)$ for any $\zeta\sim N(0,I_p)$ and $p\times p$ symmetric matrix $M$ (cf.~Eq.(11) of \cite{dalalyan2011second}). 
Then, the claim asserted follows once we verify the following conditions:
\ba{
I&:=\sup_{f\in H^\alpha_{R,b}}\E_f\sbra{\max_{h,h'\in\mcl H_n}\abs{\E^*[J_2^*(\psi_h)J_2^*(\psi_{h'})]-(hh')^{d/2}P_f^2(\pi_2\varphi_h\pi_2\varphi_{h'})}}\log^2|\mcl H_n|\to0,\\
II&:=\sup_{f\in H^\alpha_{R,b}}\E_f\sbra{\max_{h\in\mcl H_n}\sum_{i,j=1}^n\bra{\sum_{k:k\neq i,j}\psi_h(X_i,X_k)\psi_h(X_j,X_k)}^2}\log^6|\mcl H_n|\to0.
}
First we prove $I\to0$. In view of \eqref{mmd-var-eq1}, it suffices to prove
\[
I':=\sup_{f\in H^\alpha_{R,b}}\E_f\sbra{\max_{h,h'\in\mcl H_n}\abs{\E^*[J_2^*(\psi_h)J_2^*(\psi_{h'})]-(hh')^{d/2}P_f^2(\varphi_h\varphi_{h'})}}\log^2|\mcl H_n|\to0.
\]
For any $h,h'\in\mcl H_n$, observe that $\E^*[J_2^*(\psi_h)J_2^*(\psi_{h'})]=(hh')^{d/2}\binom{n}{2}^{-1}J_2(\varphi_h\varphi_{h'})$. Hence, by \eqref{eq:max-out},
\ba{
I'&\leq |\mcl H_n|\sup_{f\in H^\alpha_{R,b}}\max_{h,h'\in\mcl H_n}\sqrt{(hh')^{d}\binom{n}{2}^{-2}\Var_f[J_2(\varphi_h\varphi_{h'})]}\log^2|\mcl H_n|,
}
where $\Var_f[\cdot]$ denotes the variance with respect to $\pr_f$.  
For any $f\in H^\alpha_{R,b}$, \eqref{u-anova-2} gives
\ba{
\Var_f[J_2(\varphi_h\varphi_{h'})]
&=\binom{n}{2}\bra{2(n-2)\Var[P(\varphi_h\varphi_{h'})(X_1)]+\Var[\varphi_h(X_1,X_2)\varphi_{h'}(X_1,X_2)]}\\
&\leq n^3\E[P(\varphi_h\varphi_{h'})(X_1)^2]+n^2\E[\varphi_h(X_1,X_2)^2\varphi_{h'}(X_1,X_2)^2].
}
Thus, by the AM-GM inequality,
\ba{
\max_{h,h'\in\mcl H_n}(hh')^{d}\Var[J_2(\varphi_h\varphi_{h'})]
\leq n^3\max_{h\in\mcl H_n}h^{2d}\E[P(\varphi_h^2)(X_1)^2]+n^2\max_{h\in\mcl H_n}h^{2d}\E[\varphi_h(X_1,X_2)^4].
}
Combining this with \cref{lem:gof}\ref{gof-full}--\ref{gof-l2} gives
\ba{
I'=O\bra{(\log n)\sqrt{n^{-1}\ul h_n^{-d/2}+n^{-2}\ul h_n^{-d}}\log^2(\log n)}=o(1),
}
where the last equality follows from \eqref{ass:bandwidth}. 

Next we prove $II\to0$. 
A straightforward computation shows
\ba{
&\sum_{i,j=1}^n\bra{\sum_{k:k\neq i,j}\varphi_h(X_i,X_k)\varphi_h(X_j,X_k)}^2\\
&=\sum_{(i,j,k,l)\in I_{n,4}}\varphi_h(X_i,X_k)\varphi_h(X_j,X_k)\varphi_h(X_i,X_l)\varphi_h(X_j,X_l)\\
&\quad+2\sum_{(i,j,k)\in I_{n,3}}\varphi_h(X_i,X_k)^2\varphi_h(X_j,X_k)^2
+\sum_{(i,j)\in I_{n,2}}\varphi_h(X_i,X_j)^4.
}

Hence, by \eqref{eq:max-out},
\ba{
II&\leq|\mcl H_n|\sup_{f\in H^\alpha_{R,b}}\max_{h\in\mcl H_n}h^{2d}\bra{\|\varphi_h\star_1^1\varphi_h\|_{L^2(P_f)}^2
+2n^{-1}\|P_f(\varphi_h^2)\|_{L^2(P_f)}^2+n^{-2}\|\varphi_h\|_{L^4(P_f)}^4}\log^6|\mcl H_n|\\
&=O\bra{(\log n)\bra{\bar h_n^{4\alpha}+n^{-1}\ul h_n^{-d/2}+n^{-2}\ul h_n^{-d}}\log^6(\log n)}=o(1),
}
where the second line follows from \eqref{gof:star11} and \cref{lem:gof}\ref{gof-full}--\ref{gof-l2}. 
This completes the proof. 
\end{proof}

\subsection{Proof of Theorem \ref{gof-ada}}

The following lemma extends Lemma 15 in \cite{li2024optimality} to general kernel functions. 
\begin{lemma}\label{ly-lem15}
Let $g\in L^2(\mathbb R^d)$ satisfy $\|g\|_{H^\alpha}\leq R$ for some $\alpha>0$. 
Under \cref{ass:kernel-gof}, there exists a constant $c>0$ depending only on $K$ such that
\ba{
\int_{\mathbb R^{2d}} \varphi_h(x,y)g(x)g(y)dxdy\geq \frac{\|g\|_{L^2}^2}{2}
}
for any $0<h\leq c\bra{\|g\|_{L^2}/(2R)}^{1/\alpha}$.
\end{lemma}

\begin{proof}
Observe that
\ba{
\int \varphi_h(x,y)g(x)g(y)dxdy
=h^{-d}\int K(y/h)g(x)g(x+y)dxdy
=(2\pi)^{d/2}\int \fourier{K}(h\lambda)|\fourier{g}(\lambda)|^2d\lambda.
}
Since $\int K(u)du=1$, we have $(2\pi)^{d/2}\fourier{K}(\lambda)\to1$ as $\lambda\to0$ by the dominated convergence theorem. 
Thus, there exists a constant $c>0$ depending only on $K$ such that $|(2\pi)^{d/2}\fourier{K}(\lambda)-1|\leq1/3$ for any $|\lambda|\leq c$. 
Meanwhile, by the proof of Lemma 15 in \cite{li2024optimality},
\[
\int_{|\lambda|\leq T}|\fourier{g}(\lambda)|^2d\lambda\geq\frac{3}{4}\|g\|_{L^2}^2,
\]
where $T=(2R/\|g\|_{L^2})^{1/\alpha}$. 
In addition, since $K$ is a positive definite function, $\fourier K\geq0$. 
Consequently, if $|Th|\leq c$,
\ba{
\int \varphi_h(x,y)g(x)g(y)dxdy
\geq(2\pi)^{d/2}\int_{|\lambda|\leq T} \fourier{K}(h\lambda)|\fourier{g}(\lambda)|^2d\lambda
\geq\frac{1}{2}\|g\|_{L^2}^2.
}
This completes the proof. 
\end{proof}

\begin{proof}[Proof of \cref{gof-ada}]
For every $\alpha>0$, set
$
h_{n}(\alpha):=\max\{h\in\mcl H_n:h\leq\rho_n^{ad}(\alpha)^{1/\alpha}\}.
$
Note that the maximum always exists for sufficiently large $n$ and has the same order as $\rho_n^{ad}(\alpha)^{1/\alpha}$ by the construction of $\mcl H_n$. 
Then, it suffices to prove
$
\pr_{f_n}(J_2(\hat\psi_{h_{n}(\alpha_n)})\leq\hat c_\tau)\to0
$
for any sequences $\alpha_n\in(\alpha_0,\alpha_1)$ and $f_n\in H_1(\rho_n(\alpha_n);\alpha_n)$. 
First, since $n\rho_n^{ad}(\alpha)^{\frac{d}{2\alpha}+2}=\sqrt{\log\log n}$ for any $\alpha>0$, we have
\ben{\label{cond-bandwidth}
\inf_{\alpha_0<\alpha<\alpha_1}\frac{nh_{n}(\alpha)^{d/2}\rho_{n}(\alpha)^2}{\sqrt{\log\log n}}\to\infty.
}
Next, since $h_n(\alpha_n)/\rho_n(\alpha_n)^{1/\alpha}\to0$, 
we have by \eqref{eq:mmd} and \cref{ly-lem15}
\ben{\label{eq:alt0}
\E_{f_n}[J_2(\hat\psi_{h_n(\alpha_n)})]\geq\frac{nh_n(\alpha_n)^{d/2}}{4}\|f_n-f_0\|_{L^2}^2
}
for sufficiently large $n$. 
Hence
\ben{\label{eq:alt}
\frac{\E_{f_n}[J_2(\hat\psi_{h_n(\alpha_n)})]}{\sqrt{\log\log n}}\to\infty.
}
Now, a straightforward computation shows
\ba{
J_2(\hat\psi_{h_{n}(\alpha_n)})
=J_2(\pi_2^{f_n}\psi_{h_{n}(\alpha_n)})
+S_n
+\E_{f_n}[J_2(\hat\psi_{h_n(\alpha_n)})],
}
where
\[
S_{n}:=(n-1)\sum_{i=1}^n\bra{P_{f_n-f_0}\psi_{h_{n}(\alpha_n)}(X_i)-\E_{f_n}[P_{f_n-f_0}\psi_{h_{n}(\alpha_n)}(X_i)]}.
\]
Hence,
\ba{
\pr_{f_n}(J_2(\hat\psi_{h_{n}(\alpha_n)})\leq\hat c_\tau)
&\leq\pr_{f_n}(a_n<\hat c_\tau)
+\pr_{f_n}(a_n<|J_2(\pi_2^{f_n}\psi_{h_{n}(\alpha_n)}|))
+\pr_{f_n}(a_n<|S_{n}|)\\
&=:I+II+III,
}
where $a_n:=\E_{f_n}[J_2(\hat\psi_{h_{n}(\alpha_n)})]/4$. 
Let us bound $I$. By the definition of $\hat c_\tau$,
$I=\pr_{f_n}(\pr^*(T_n^*> a_n)>\tau).$ Hence, Markov's inequality gives $I\leq\tau^{-1}\E_{f_n}[\pr^*(T_n^*> a_n)]$. 
Recall that $\|f_0\|_{H^\gamma}<\infty$ for some $\gamma>0$. 
Hence, $f_n\in H^{\alpha_0\wedge\gamma}_{R_1,b}$ with $R_1:=R+\|f_0\|_{H^\gamma}$ and $b:=\|f_0\|_{L^2}^2/2$ for sufficiently large $n$, so $\E_{f_n}[\pr^*(T_n^*> a_n)]=\pr_{f_n}(\max_{h\in\mcl H_n}Z_h> a_n)+o(1)$ by \cref{gof-boot}. 
Since $\E_{f_n}[\max_{h\in\mcl H_n}Z_h]=O(\sqrt{\log|\mcl H_n|})$ by \cite[Lemma 2.3.4]{gine2016mathematical} and \cref{mmd-var}, we obtain $I\to0$ by \eqref{eq:alt}. 
Next, observe that $II\leq\pr_{f_n}(\max_{h\in\mcl H_n}J_2(\pi_2^{f_n}\psi_h)>a_n)=\pr_{f_n}(\max_{h\in\mcl H_n}Z_h>a_n)+o(1)$, where the equality follows from \cref{gof-ga}. Hence, the same argument as above gives $II\to0$. 
Finally, since $X_i\overset{i.i.d.}{\sim}P_{f_n}$ under $\pr_{f_n}$, 
\ba{
\E_{f_n}[S_{n}^2]\leq n^3\|P_{f_n-f_0}\psi_{h_{n}(\alpha_n)}\|_{L^2(P_{f_n})}^2\leq C_Knh_{n}(\alpha_n)^{d/2}\|f_n-f_0\|_{L^2}^2\|f_n\|_{L^2},
}
where the second inequality follows from \cref{lem:gof}\ref{gof-l2}. Thus, by \eqref{eq:alt0}, for sufficiently large $n$, 
\ba{
a_n^{-2}\E_{f_n}[S_{n}^2]
\leq C_K\frac{\|f_n\|_{L^2}}{nh_{n}(\alpha_n)^{d/2}\|f_n-f_0\|_{L^2}^2}
\leq\frac{C_K(R+\|f_0\|_{L^2})}{nh_{n}(\alpha_n)^{d/2}\rho_n(\alpha_n)^2},
}
where the second inequality is due to $f_n\in H_1(\rho_n(\alpha_n);\alpha_n)$. 
Therefore, $III\to0$ by Markov's inequality and \eqref{cond-bandwidth}. 
Consequently, we complete the proof. 
\end{proof}

\section{Proof for Section \ref{sec:DWAD}} \label{sec:proof_DWAD}
\subsection{Proof of Theorem \ref{thm:DWAD_size}}


\paragraph{Notations:}
Recall that, for notational simplicity, we define $\mathcal{A}_n \coloneqq [p]\times[d]\times\mathcal{H}_n$ and $|\mathcal{A}_n| \coloneqq p\times d\times |\mathcal{H}_n|$. 
Also, for each $(j,k,h_n) \in\mathcal{A}_n$, we define
\begin{align*}
        & J_{jk,h_n} \coloneqq \frac{\hat{\theta}_{jk, h_n} - \theta_{jk}}{\sigma_{jk,h_n}}, 
        \quad 
        \hat{J}_{jk,h_n} \coloneqq \frac{\hat{\theta}_{jk, h_n} - \theta_{jk}}{\hat{\sigma}_{jk,h_n'}},\\
        & \hat{J}_{jk,h_n'}^{\sharp} \coloneqq  \frac{1}{\hat{\sigma}_{jk,h_n'}} \left(  \frac{1}{n}\sum_{i=1}^n z_i \left[ \sum_{l\neq i}^n \psi_{jk,h'_n}(Z_i,Z_l) - \hat{\theta}_{jk,h_n'} \right] \right).
\end{align*}
Also we define
\begin{align*}
    & J \coloneqq (J_{jk,h_n})_{1\le j \le p, 1\le k\le d, h_n\in\mathcal{H}_n}, ~~ \hat{J} \coloneqq (\hat{J}_{jk,h_n})_{1\le j \le p, 1\le k\le d, h_n\in\mathcal{H}_n}, ~~ \hat{J}^\sharp \coloneqq (\hat{J}_{jk,h_n}^\sharp)_{1\le j \le p, 1\le k\le d, h_n\in\mathcal{H}_n}.
\end{align*}

Throughout this section, we use $C$ to denote positive constants, which possibly depend only on quantities independent of $n$. Specifically, we suppose that $C$ possibly depends only on $r$, $d$, $\int_{\mathbb{R}^d} |u|K(u)du$, $\sup_{k\in[d]}\|\partial_k K\|_{L^2(\mathbb{R}^d)}$ and $\sup_{k\in[d]}\|\partial_k K\|_{L^\infty(\mathbb{R}^d)}$ as well as the constant in Assumption \ref{as:DWAD-DGP} and may be different in different expressions.

\paragraph{Proof of \cref{thm:DWAD_size}:}

To show the theorem, we show the convergence of 
\[ \sup_{t\in\mathbb{R}} \left| \mathbb{P}\left( \max_{(j,k,h_n)\in\mathcal{A}_n } | \hat{J}_{jk,h_n} | \le t\right) - \mathbb{P}^*\left( \max_{(j,k,h_n)\in\mathcal{A}_n } | \hat{J}_{jk,h_n}^\sharp | \le t \right) \right|.\]
From the triangular inequality, 
\begin{align*}
    &  \sup_{t\in\mathbb{R}} \left| \mathbb{P}\left( \max_{(j,k,h_n)\in\mathcal{A}_n } |\hat{J}_{jk,h_n}| \le t\right) - \mathbb{P}^*\left( \max_{(j,k,h_n)\in\mathcal{A}_n } |\hat{J}_{jk,h_n}^\sharp| \le t \right) \right| \\
    & \le  
    \sup_{t\in\mathbb{R}} \left|\mathbb{P}\left(  \max_{(j,k,h_n)\in\mathcal{A}_n } |J_{jk,h_n}| \le t \right) - \mathbb{P}\left(   \max_{(j,k,h_n)\in\mathcal{A}_n } |G_{jk,h_n}| \le t\right)\right| \\
    & \quad + \sup_{t\in\mathbb{R}} \left|\mathbb{P}^*\left(  \max_{(j,k,h_n)\in\mathcal{A}_n } |\hat{J}^\sharp_{jk,h_n}| \le t \right) - \mathbb{P}\left(   \max_{(j,k,h_n)\in\mathcal{A}_n } |G_{jk,h_n}| \le t\right)\right| \\
    & \quad + \sup_{t\in\mathbb{R}} \left|\mathbb{P}\left(  \max_{(j,k,h_n)\in\mathcal{A}_n } |\hat{J}_{jk,h_n}| \le t \right) - \mathbb{P}\left(   \max_{(j,k,h_n)\in\mathcal{A}_n } |J_{jk,h_n}| \le t\right)\right| \eqqcolon (\rho^\sharp.I) + (\rho^\sharp.II) + (\rho^\sharp.III),
\end{align*}
where $G \coloneqq (G_{jk,h_n})_{1\le j \le p, 1\le k\le d, h_n\in\mathcal{H}_n} \sim N(0, \Cov(J))$. 
Given the convergence of these terms, Theorem \ref{thm:DWAD_size} follows from the proof of \citep[Proposition 3.2]{koike2019mixed}.

\paragraph{Convergence of $(\rho^\sharp.I)$:}
To show the convergence of $(\rho^\sharp.I)$, it is sufficient to show the convergence of  $\sup_{A\in\mathcal{R}}\left|\mathbb{P}(J \in A) - \mathbb{P}\left( G \in A\right)\right|$.
The convergence is an immediate consequence of the approximation results in Corollary \ref{coro:order-2}. 

Before the evaluation we show the boundedness of some quantities.
\begin{lemma} \label{lem:norm-DWAD}
Under Assumption \cref{as:DWAD-DGP}, $\max_{j\in[p]}\|g_j\|_{L^\infty(\mathbb{R}^d)}$, $\max_{j\in[p]}\|v_j\|_{L^\infty(\mathbb{R}^d)}$,  $\|f\|_{L^\infty(\mathbb{R}^d)}$, $\|f\|_{L^2(\mathbb{R}^d)}$ and $\max_{j\in[p]}\|g_jf\|_{L^2(\mathbb{R}^d)}$ are bounded.
\end{lemma}
\begin{proof}
The boundedness of $\max_{j\in[p]}\|g_j\|_{L^\infty(\mathbb{R}^d)}$ and $\max_{j\in[p]}\|v_j\|_{L^\infty(\mathbb{R}^d)}$ immediately follows from Assumption \ref{as:DWAD-DGP}(iv).
Also, the boundedness of $\|f\|_{L^2(\mathbb{R}^d)}$ and $\max_{j\in[p]}\|g_jf\|_{L^2(\mathbb{R}^d)}$ follows from Assumption  \ref{as:DWAD-DGP}(ii) and Plancherel's Theorem.

Finally, we show the boundedness of $\|f\|_{L^\infty(\mathbb{R}^d)}$.
For some function $m(x)\in L^2(\mathbb{R}^d)$, assume $\int_{\mathbb{R}^d}m(x)dx = 1$ and $\int_{\mathbb{R}^d} |x||m(x)|dx < \infty$.
Then,
\begin{align*}
    & f(x) = \int_{\mathbb{R}^d} f(x)m(y)dy  = \int_{\mathbb{R}^d}f(x-y)m(y)dy + \int_{\mathbb{R}^d} \{f(x) - f(x-y)\}m(y)dy.
\end{align*}
Applying Young's inequality and the mean value theorem, we have
\begin{align*}
    \|f\|_{L^\infty(\mathbb{R}^d)} \le \|f\|_{L^2(\mathbb{R}^d)} \|m\|_{L^2(\mathbb{R}^d)} + \||\nabla f|\|_{L^\infty(\mathbb{R}^d)} \int_{\mathbb{R}^d} |y||m(y)|dy.
\end{align*} 
Therefore, $\|f\|_{L^\infty(\mathbb{R}^d)} $ is bounded under Assumption \ref{as:DWAD-DGP}. 
\end{proof}

In order to apply Corollary \ref{coro:order-2}, we evaluate $\sigma_{jk,h_n}^2$ at first.
From \eqref{u-anova-2}, we can see that 
\begin{align*}
    \sigma_{jk,h_n}^2 
    &= \frac{1}{n^2} n(n-1)(n-2) \Var[P\psi_{jk,h_n}(Z_i)] + \frac{1}{n^2} \frac{n(n-1)}{2} \Var[\psi_{jk,h_n}(Z_i,Z_k)].
\end{align*}
The fourth line of Eq.(3.15) in \cite{powell1989semiparametric} states that
\begin{align}
    P\psi_{jk,h_n}(Z_i) &= \frac{1}{n-1}\int_{\mathbb{R}^d} K(u)[-Y_{ij}\partial_kf(X_i+uh_n) + \partial_k \{g_j(X_i + uh_n)f(X_i+uh_n)\}]du. \label{eq:DWAD_hajek}
\end{align}
Also, Eq.(3.13) in \cite{powell1989semiparametric} states that
\begin{align}
    & P^2(\psi_{jk,h_n}^2) \nonumber\\
    &= \frac{1}{(n-1)^2h_n^{d+2}}\int_{\mathbb{R}^d\times \mathbb{R}^d} \{\partial_k K(u)\}^2 [v_j(x) + v_j(x+uh_n) - 2g_j(x)g_j(x+uh_n)]f(x)f(x+uh_n) dxdu. \label{eq:DWAD-ex-squared-kernel}
\end{align}
Therefore, from Assumption \ref{as:DWAD-DGP} and \cref{lem:norm-DWAD}, we have
\begin{align*}
    & \left\|(n-1)P\psi_{jk,h_n}(Z_i) - [-Y_{ij}\partial_kf(X_i) + \partial_k \{g_j(X_i)f(X_i)\}]\right\|_{L^2(P)}^2 \\
    & \le  2\int_{\mathbb{R}^d} v_j(x) \left| \int_{\mathbb{R}^d} K(u)\{\partial_kf(x+uh_n) - \partial_kf(x)\} du\right|^2 f(x)dx \\
    & \quad +   2 \int_{\mathbb{R}^d} \left| \int_{\mathbb{R}^d} K(u)[ \partial_k \{g_j(x + uh_n)f(x+uh_n)\} -  \partial_k \{g_j(x)f(x)\}] du\right|^2 f(x)dx \\
    & \le 2 \|f\|_{L^\infty(\mathbb{R}^d)} \|v\|_{L^\infty(\mathbb{R}^d)} \|K\|_{L^1(\mathbb{R}^d)} \int_{\mathbb{R}^d}\int_{\mathbb{R}^d} |K(u)| \{\partial_kf(x+uh_n) - \partial_kf(x)\}^2dudx \\
    & \quad + 2 \|f\|_{L^\infty(\mathbb{R}^d)}  \|K\|_{L^1(\mathbb{R}^d)} \int_{\mathbb{R}^d}\int_{\mathbb{R}^d} |K(u)|[ \partial_k \{g_j(x + uh_n)f(x+uh_n)\} -  \partial_k \{g_j(x)f(x)\}]^2 dudx \\
    & \le 2^{2(1-\alpha)+1}\|f\|_{L^\infty(\mathbb{R}^d)} \|v\|_{L^\infty(\mathbb{R}^d)} \|K\|_{L^1(\mathbb{R}^d)} \|\partial_k f\|_{H^\alpha}^2 \int_{\mathbb{R}^d} |K(u)||uh_n|^{2\alpha}du \\
    & \quad + 2^{2(1-\alpha)+1}\|f\|_{L^\infty(\mathbb{R}^d)}\|K\|_{L^1(\mathbb{R}^d)} \|\partial_k \{gf\}\|_{H^\alpha}^2 \int_{\mathbb{R}^d} |K(u)||uh_n|^{2\alpha}du,
\end{align*}
where the final inequality follows from \cref{lem:sobolev}.
Therefore
\begin{align}
    &\max_{j\in[p]}\left| \Var[(n-1)P\psi_j(Z_i)] - \Var[-Y_{ij}\partial_kf(X_i) + \partial_k \{g_j(X_i)f(X_i)\}]\right| \nonumber\\
    & = \max_{j\in[p]}\left| \Var[(n-1)P\psi_j(Z)] - \Var[\partial_k \{g_jf\}(X) -Y_{j}\partial_kf(X)]\right|\to 0. \label{eq:DWAD-variance1}
\end{align}
as $n\to\infty$.
Also, since $f(x)f(x+uh_n) = f(x)^2 + f(x)\{f(x+uh_n) - f(x)\}$, from Assumption \ref{as:DWAD-DGP} and \cref{lem:norm-DWAD}, we have
\begin{align*}
    & \left| (n-1)^2h_n^{d+2}P^2(\psi_{jk,h_n}^2) - \|\partial_kK\|_{L^2(\mathbb{R}^d)}^2 \E[\{2v_j(X_i) - 2g_j(X_i)^2\}f(X_i)] \right| \\
    & \lesssim \int_{\mathbb{R}^d}\int_{\mathbb{R}^d} \{\partial_kK(u)\}^2 |v_j(x+uh_n) - v_j(x) | f(x)^2 dudx \\
    & \quad + \int_{\mathbb{R}^d}\int_{\mathbb{R}^d} \{\partial_kK(u)\}^2 | v_j(x) + v_j(x+uh_n)||f(x+uh_n)-f(x)|f(x)dudx \\
    & \quad + \int_{\mathbb{R}^d}\int_{\mathbb{R}^d} \{\partial_kK(u)\}^2 |g_j(x+uh_n) - g_j(x)|g_j(x)f(x)^2 dudx \\
    & \quad + \int_{\mathbb{R}^d}\int_{\mathbb{R}^d} \{\partial_kK(u)\}^2 |f(x+uh_n) - f(x)|g_j(x)g_j(x+uh_n)f(x)dudx \\
    & \le  \|f\|_{L^\infty(\mathbb{R}^d)}\|f\|_{L^2(\mathbb{R}^d)} \int_{\mathbb{R}^d}  \{\partial_kK(u)\}^2du\sqrt{\int_{\mathbb{R}^d}|v_j(x+uh_n) - v_j(x)|^2 dx}   \\
    & \quad + 2\|v_j\|_{L^\infty(\mathbb{R}^d)}  \|f\|_{L^2(\mathbb{R}^d)} \int_{\mathbb{R}^d} \{\partial_kK(u)\}^2 du \sqrt{\int_{\mathbb{R}^d}|f(x+uh_n) - f(x)|^2 dx} \\
    & \quad + \|f\|_{L^\infty(\mathbb{R}^d)}\|g_jf\|_{L^2(\mathbb{R}^d)}\int_{\mathbb{R}^d} \{\partial_kK(u)\}^2du \sqrt{\int_{\mathbb{R}^d}|g_j(x+uh_n) - g_j(x)|^2 dx} \\
    & \quad + \|g_j\|_{L^\infty(\mathbb{R}^d)}\|g_jf\|_{L^2(\mathbb{R}^d)} \int_{\mathbb{R}^d} \{\partial_kK(u)\}^2 du \sqrt{\int_{\mathbb{R}^d}|f(x+uh_n) - f(x)|^2 dx} \\
    & \le 2^{2(1-\alpha)}\|f\|_{L^\infty(\mathbb{R}^d)}\|f\|_{L^2(\mathbb{R}^d)} \|\partial_kK\|_{L^\infty(\mathbb{R}^d)} \|v_j\|_{H^\alpha} \int_{\mathbb{R}^d}  |\partial_kK(u)||uh_n|^\alpha du \\
    & \quad + 2^{2(1-\alpha) + 1} \|v_j\|_{L^\infty(\mathbb{R}^d)}  \|f\|_{L^2(\mathbb{R}^d)} \|\partial_kK\|_{L^\infty(\mathbb{R}^d)} \|f\|_{H^\alpha} \int_{\mathbb{R}^d}  |\partial_kK(u)||uh_n|^\alpha du \\
    & \quad + 2^{2(1-\alpha)}\|f\|_{L^\infty(\mathbb{R}^d)}\|g_jf\|_{L^2(\mathbb{R}^d)} \|\partial_kK\|_{L^\infty(\mathbb{R}^d)} \|g_j\|_{H^\alpha}\int_{\mathbb{R}^d}  |\partial_kK(u)||uh_n|^\alpha du \\
    & \quad + 2^{2(1-\alpha)}\|g_j\|_{L^\infty(\mathbb{R}^d)}\|g_jf\|_{L^2(\mathbb{R}^d)}\|\partial_kK\|_{L^\infty(\mathbb{R}^d)} \|f\|_{H^\alpha}\int_{\mathbb{R}^d}  |\partial_kK(u)||uh_n|^\alpha du
\end{align*}
where the second inequality follows from Schwarz inequality and the final inequality follows from \cref{lem:sobolev}.
Therefore
\begin{align}
    & \max_{j\in[p]} \left| (n-1)^2h_n^{d+2}P^2(\psi_{jk,h_n}^2) - \|\partial_kK\|_{L^2(\mathbb{R}^d)}^2 \E[\{2v_j(X_i) - 2g_j(X_i)^2\}f(X_i)] \right| \nonumber\\
    & = \max_{j\in[p]} \left| (n-1)^2h_n^{d+2}P^2(\psi_{jk,h_n}^2) - 2\|\partial_kK\|_{L^2(\mathbb{R}^d)}^2 \E[\Var(Y_j\mid X) f(X)] \right|\to 0. \label{eq:DWAD-variance-2}
\end{align}
From \eqref{eq:DWAD-variance1}, \eqref{eq:DWAD-variance-2}, Assumption \ref{as:DWAD-DGP}(v) and Assumption \ref{as:DWAD-Kernel}(ii) we have $\max_{(j,k,h_n)\in\mathcal{A}_n}\sigma_{jk,h_n}^{-2} = O\left( \min\{n, n^2\ol{h}_n^{d+2}\}\right)$.

Next, we verify the conditions in Remark \ref{rmk:coro:order-2}. 
Note that, from \eqref{eq:DWAD_hajek}, we can see that
\begin{align}
    & |P\psi_{jk,h_n} |\lesssim  n^{-1}(|Y_j| + 1),  \label{eq:eval-DWAD-hajek}
\end{align}
for each $j\in[p]$ and $k\in[d]$.
In what follows, we use this evaluation.
In view of \eqref{eq:eval-DWAD-hajek}, it holds that
\begin{align*}
    & \Tilde{\Delta}_{2,*}^{(1)} \log^5(n|\mathcal{A}_n|) = n^5\max_{(j,k,h_n)\in\mathcal{A}_n} \frac{\|P(n^{-1}\psi_{jk,h_n})\|_{L^4(P)}^4}{\sigma^4_{jk,h_n}}  \log^5(n|\mathcal{A}_n|) = O(n^{-1}  \log^5(n|\mathcal{A}_n|)), \\
    & \Tilde{\Delta}_{2,q}^{(2)}(1)\log^5(n|\mathcal{A}_n|) \le n^4 \left\| \max_{(j,k,h_n)\in\mathcal{A}_n} \frac{|P(n^{-1}\psi_{jk,h_n})|}{\sigma_{jk,h_n}} \right\|^4_{L^\infty(P)}\log^6(n|\mathcal{A}_n|)  = O(n^{-2} \log^6 (n|\mathcal{A}_n|)).
\end{align*}
By the proof of Lemma 4 in \cite{Ro95} states that $P^2(\psi_{jk,h_n}^4) = O(n^{-4}h_n^{-3d -4})$ for each $j\in[p]$ and $k\in[d]$, we can see that
\begin{align*}
    & \Tilde{\Delta}_{2,*}^{(1)}(2)\log^5(n|\mathcal{A}_n|)  \le n^2 \max_{(j,k,h_n)\in\mathcal{A}_n}  \frac{\|n^{-1}\psi_{jk,h_n}\|^4_{L^4(P^2)}}{\sigma_{jk,h_n}^4} \log^8(n|\mathcal{A}_n|) = O(n^{-2}\underline{h}_n^{-d} \log^8(n|\mathcal{A}_n|)).
\end{align*}
With a slight modification of Lemma 4 in \cite{NiRo00}, we can see that $P(\psi_{jk,h_n}^m) \lesssim (|Y_{ij}|^m+1)n^{-m}h_n^{-(m-1)d-m}$ for $2\le m\le 4$ and each $j\in[p]$ and $k\in[d]$, so it holds that
\begin{align*}
    &\Tilde{\Delta}_{2,*}^{(2)}(2)\log^5(n|\mathcal{A}_n|) \le n^3\max_{(j,k,h_n)\in\mathcal{A}_n}\frac{\|P(n^{-2}\psi_{jk,h_n}^2)\|^2_{L^2(P)}}{\sigma_{jk,h_n}^4} \log^7(n|\mathcal{A}_n|) = O(n^{-1} \log^7(n|\mathcal{A}_n|)),\\
    &\Tilde{\Delta}_{2,*}^{(3)}(2)\log^5(n|\mathcal{A}_n|) \le n\left\| \max_{(j,k,h_n)\in\mathcal{A}_n} \frac{P(n^{-4}\psi_{jk,h_n}^4)}{\sigma_{jk,h_n}^4} \right\|_{L^\infty(P)} \log^9(n|\mathcal{A}_n|) = O(n^{-3}\underline{h}_n^{-d}\log^9(n|\mathcal{A}_n|)),\\
    &\Tilde{\Delta}_{2,q}^{(5)}(2)\log^9(n|\mathcal{A}_n|) = n^2 \left\|\max_{(j,k,h_n)\in\mathcal{A}_n} \frac{P(n^{-2}\psi_{jk,h_n}^2)}{\sigma_{jk,h_n}^2}\right\|^2_{L^\infty(P)} \log^{12}(n|\mathcal{A}_n|) = O(n^{-2} \log^{12}(n|\mathcal{A}_n|)).
\end{align*}
Also, since $|\psi_{jk,h_n}| \lesssim  (|Y_{ij} + |Y_{lj}|)n^{-1}h_n^{-(d+1)}$ for each $j\in[p]$ and $k\in[d]$,
\begin{align*}
    \Tilde{\Delta}_{2,q}^{(4)}(2)\log^5(n|\mathcal{A}_n|) = \left\| \max_{(j,k,h_n)\in\mathcal{A}_n} \frac{|n^{-1}\psi_{jk,h_n}|}{\sigma_{jk,h_n}}\right\|_{L^\infty(P^2)}^4 \log^{10}(n|\mathcal{A}_n|) = O(n^{-4}\underline{h}_n^{-2d} \log^{10}(n|\mathcal{A}_n|)).
\end{align*}
In addition, since $ \psi_{jk,h_n} \star_1^1 \psi_{jk,h_n}(Z_2, Z_3) = \E[\psi_{jk,h_n}(Z_1, Z_2)  \psi_{jk,h_n}(Z_1, Z_3) \mid Z_2, Z_3] $, from Lemma 6 in \cite{NiRo00}, 
we have $\|\psi_{jk,h_n} \star_1^1 \psi_{jk,h_n}(Z_1, Z_2)\|_{L(P^2)} = O(n^{-2} h^{-(d+4)/2})$ for each $j\in[p]$ and $k\in[d]$.
Hence,
\begin{align*}
    \Tilde{\Delta}_1^{(0)} \log^5 |\mathcal{A}_n| = n^2 \max_{(j,k,h_n)\in\mathcal{A}_n} \frac{\|n^{-1}\psi_{jk,h_n} \star_1^1 n^{-1}\psi_{jk,h_n}(Z_1, Z_2)\|_{L^2(P^2)}}{\sigma_{jk,h_n}^2} \log^5|\mathcal{A}_n| = O( \overline{h}^{d/2} \log^5 |\mathcal{A}_n|) .
\end{align*}
Consequently, under Assumption \ref{as:DWAD-bandwdith}(ii), we have 
\begin{align*}
    \sup_{A\in\mathcal{R}_p}\left| \mathbb{P}\left( \Tilde{J} \in A\right) - \mathbb{P}(G\in A) \right| \to 0,
\end{align*}
where $\Tilde{J}$ is defined as $\Tilde{J} \coloneqq (\Tilde{J}_{jk,h_n})_{1\le j \le p, 1\le k\le d, h_n\in\mathcal{H}_n}$ with $\tilde{J}_{jk,h_n} \coloneqq (\hat{\theta}_{jk,h_n} - \E[\hat{\theta}_{jk,h_n}])/\sigma_{jk,h_n}$.

Finally, we evaluate the bias term.
Eq.(3.19) in \cite{powell1989semiparametric} gives
\begin{align*}
    \mathbb{E}[\hat{\theta}_{jk,h_n}] 
    &=  -2\int_{\mathbb{R}^d \times \mathbb{R}^d} K(u)g_j(x_1)f(x_1)\partial_k f(x_1+uh_n)dx_1du
\end{align*}
and by definition of $\theta_{jk}$, we can see that
\begin{align*}
    \theta_{jk} = \int_{\mathbb{R}^d} \partial_kg_j(x) f^2(x)  dx = -2\int_{\mathbb{R}^d} g_j(x)f(x)\partial_k f(x) dx = -2\int_{\mathbb{R}^d \times \mathbb{R}^d} K(u)g(x)f(x)\partial_k f(x) dxdu.
\end{align*}
So, we have
\begin{align*}
    \mathbb{E}[\hat{\theta}_{jk,h_n}]  - \theta_j &= -2 \int_{\mathbb{R}^d}  \int_{\mathbb{R}^d}K(u)  \{\partial_k f(x+uh_n) - \partial_k f(x) \}g(x)f(x)dxdu \\
    &= -2 \int_{\mathbb{R}^d} K(u) [(\partial_k f * g_jf)(uh_n) - (\partial_k f * g_jf)(0)] du
\end{align*}
where $*$ denotes convolution.
If $\partial_kf $ and $g_jf$ satisfy Assumption \ref{as:DWAD-DGP}, Lemma 1 in \cite{gine2008simple} gives
\begin{align*}
    |\mathbb{E}[\hat{\theta}_{jk,h_n}]  - \theta_{jk}| &\lesssim h_n^{2\alpha} \|\partial_k f\|_{H^\alpha} \|g_jf\|_{H^\alpha} \int K(u) |u|^{2\alpha} du.
\end{align*}
Therefore, it holds that
\begin{align}
    \max_{(jk,h_n)\in\mathcal{A}_n} \left|\frac{\mathbb{E}[\hat{\theta}_{jk,h_n}]  - \theta_{jk}}{\sigma_{jk,h_n}}\right| = O(\min\{n,n^2\bar{h}_n^{d+2}\}^{1/2} \bar{h}_n^{2\alpha}). \label{eq:DWAD_bias}
\end{align}
Then, under Assumption \ref{as:DWAD-bandwdith}(i), Lemma 1 in \cite{chernozhukov2023high} gives
\begin{align}
    \sup_{A\in\mathcal{R}_p}\left| \mathbb{P}(J\in A) - \mathbb{P}(G\in A) \right| \to 0. \label{eq:DWAD_normal}
\end{align}

\paragraph{Convergence of $(\rho^\sharp.II)$:}
To show the convergence of $(\rho^\sharp.II)$, it is sufficient to show the convergence of $\sup_{A\in\mathcal{R}}\left|\mathbb{P}^*(\hat{J}^\sharp \in A) - \mathbb{P}\left( G\in A\right)\right|$.
In this paragraph, following \cite[Section 4]{chetverikov2021adaptive}, we index the elements in $\mathcal{A}_n$ by $a = 1,\dots, |\mathcal{A}_n|$ for notational simplicity.
Note that $a \mapsto (j_a, k_a, h_{n,a})$ is a one-to-one mapping from $\{1,\dots, |\mathcal{A}_n|\}$ to $\mathcal{A}_n$. 
We also use the notation $\psi'_a \coloneqq \psi_{j_ak_a, h_{n,a}'}$ and $\hat{\theta}'_a \coloneqq \hat{\theta}_{j_ak_a,h_a'}$ to stress the difference of bandwidths.

Since $\min_{a\in\mathcal{A}_n} \Var[\hat{\theta}_a/\sigma_a] = 1 >0$, Proposition 2.1 in \cite{CCKK22} gives
\begin{align*}
    (\rho^\sharp.II) \le C \left( \Delta^\sharp \log^2 |\mathcal{A}_n|\right)^{1/2}, \quad \text{with} \quad \Delta^\sharp \coloneqq \max_{(a_1,a_2)\in\mathcal{A}_n^2} |\Cov(J)_{a_1, a_2} - \Cov^*(\hat{J}^\sharp)_{a_1,a_2}|,
\end{align*}
where $\Cov(J)_{a_1,a_2}$ and $\Cov^*(J^\sharp)_{a_1,a_2}$ are the $(a_1,a_2)$-th elements of the covariance matrices of $J$ and $\hat{J}^\sharp$, respectively.
In order to confirm the convergence of $(\rho^{\sharp}.II)$, we show 
\begin{align*}
    & (\Delta^\sharp.I) \coloneqq \max_{(a_1,a_2)\in\mathcal{A}_n^2} |\Cov(J)_{a_1, a_2} - \Cov^*(J^\sharp)_{a_1,a_2}| = o_p(\log^{-2} |\mathcal{A}_n|), \\
    & (\Delta^\sharp.II) \coloneqq \max_{(a_1,a_2)\in\mathcal{A}_n^2} |\Cov^*(J^\sharp)_{a_1, a_2} - \Cov^*(\hat{J}^\sharp)_{a_1,a_2}| =  o_p(\log^{-2} |\mathcal{A}_n|).
\end{align*}
In terms of $ (\Delta^\sharp.I)$, it suffice to show that
\begin{align*}
    & (\Delta^\sharp.I.I) \\
    & \coloneqq \max_{(a_1,a_2)\in\mathcal{A}_n^2}\frac{\left| \sigma_{a_1}\sigma_{a_2}\Cov^*(J^\sharp)_{a_1,a_2} -  \{n \Cov(P\psi'_{a_1}, P\psi'_{a_2}) + \Cov(\psi'_{a_1}, \psi'_{a_2})\}\right|}{\sigma_{a_1}\sigma_{a_2}} =o_p\left( \frac{1}{\log^{2}|\mathcal{A}_n|}\right), \\
    & (\Delta^\sharp.I.II) \\
    & \coloneqq  \max_{(a_1,a_2)\in\mathcal{A}_n^2} \left| \frac{\{n \Cov(P\psi'_{a_1}, P\psi'_{a_2}) + \Cov(\psi'_{a_1}, \psi'_{a_2})\}}{\sigma_{a_1}\sigma_{a_2}}  - \Cov(J)_{a_1,a_2}  \right| = o\left( \frac{1}{\log^{2}|\mathcal{A}_n|}\right). 
\end{align*}
In what follows, we often use 
\begin{align}
    \frac{1}{\sigma_{a_1}\sigma_{a_2}} \asymp \min \{n,n^2h_{n,a_1}^{(d+2)/2}h_{n,a_2}^{(d+2)/2}\}. \label{eq:asymp-variance}
\end{align}

First, we show the convergence of $(\Delta^\sharp.I.II)$.
From a variant of \eqref{u-anova-2}, we have
\begin{align}
    & \Cov(J)_{a_1,a_2} \nonumber \\
    &= \frac{1}{\sigma_{a_1}\sigma_{a_2}}\Cov\left( \hat\theta_{a_1}, \hat{\theta}_{a_2}\right) \nonumber \\
    &= \frac{1}{n^2}n(n-1)(n-2) \frac{\Cov(P\psi_{a_1}, P\psi_{a_2})}{\sigma_{a_1}\sigma_{a_2}} + \frac{1}{n^2} \frac{n(n-1)}{2} \frac{\Cov(\psi_{a_1}, \psi_{a_2})}{\sigma_{a_1}\sigma_{a_2}}. \label{eq:DWAD_COV_J} 
\end{align}
From \eqref{eq:eval-DWAD-hajek} and Lemma \ref{lem:sobolev}, 
\begin{align}
    & \max_{(a_1,a_2)\in\mathcal{A}_n^2}\left|\frac{\Cov(P\psi_{a_1}, P\psi_{a_2})}{\sigma_{a_1}\sigma_{a_2}}\right| = O(n^{-1}), \label{eq:DWAD_COV_J1}
\end{align}  
and from the Cauchy-Schwarz inequality and \eqref{eq:DWAD-ex-squared-kernel},
\begin{align}
     \max_{(a_1,a_2)\in\mathcal{A}_n^2}\left|\frac{\Cov(\psi_{a_1}, \psi_{a_2})}{\sigma_{a_1}\sigma_{a_2}}\right| 
    & \le \max_{a\in\mathcal{A}_n}\frac{\Var[\psi_{a}]}{\sigma_{a}^2} \le \max_{a\in\mathcal{A}_n} \frac{P^2\psi_{a}^2}{\sigma_{a}^2} = O\left( 1 \right). \label{eq:DWAD_COV_J2}
\end{align}
Therefore, we can approximate $\Cov(J)_{a_1,a_2}$ by $n\sigma^{-1}_{a_1}\sigma^{-1}_{a_2}\Cov(P\psi_{a_1}, P\psi_{a_2}) + \frac{1}{2} \sigma^{-1}_{a_1}\sigma^{-1}_{a_2} \Cov(\psi_{a_1}, \psi_{a_2})$ uniformly in $(a_1,a_2)\in\mathcal{A}_n^2$:
\begin{align*}
    & \max_{(a_1,a_2)\in\mathcal{A}_n^2} \left|\Cov(J)_{a_1,a_2}  - \frac{\left\{ n\Cov(P\psi_{a_1}, P\psi_{a_2}) + \frac{1}{2} \Cov(\psi_{a_1}, \psi_{a_2})\right\}}{\sigma_{a_1}\sigma_{a_2}}\right| \\
    & \asymp \max_{(a_1,a_2)\in\mathcal{A}_n^2} \frac{\left| \Cov(P\psi_{a_1}, P\psi_{a_2}) + \frac{1}{2n} \Cov(\psi_{a_1}, \psi_{a_2}) \right|}{\sigma_{a_1}\sigma_{a_2}}  = O\left( \frac{1}{n}\right).
\end{align*} 
This implies that
\begin{align*}
    (\Delta^\sharp.I.II) & \lesssim \max_{(a_1,a_2)_\in\mathcal{A}_n^2} \frac{n}{\sigma_{a_1}\sigma_{a_2}}\left|\Cov(P\psi'_{a_1}, P\psi'_{a_2}) - \Cov(P\psi_{a_1}, P\psi_{a_2})\right|\\
    &\quad\quad + \max_{(a_1,a_2)_\in\mathcal{A}_n^2}\frac{1}{\sigma_{a_1}\sigma_{a_2}} \left|\Cov(\psi_{a_1}', \psi'_{a_2}) - \frac{1}{2}\Cov(\psi_{a_1}, \psi_{a_2})\right|.
\end{align*} 
In addition, from \eqref{eq:DWAD_hajek}, slight modification of \eqref{eq:DWAD-ex-squared-kernel} and Lemma \ref{lem:sobolev}, we can see that
\begin{align*}
    & \max_{(a_1,a_2)\in\mathcal{A}_n^2} \frac{n}{\sigma_{a_1}\sigma_{a_2}}\left| \Cov(P\psi_{a_1}', P\psi'_{a_2}) - \Cov(P\psi_{a_1}, P\psi_{a_2})\right| \\
    & \quad \lesssim\max_{(a_1,a_2)\in\mathcal{A}_n^2}  \frac{1}{n}\frac{1}{\sigma_{a_1}\sigma_{a_2}}\{|h_{n,a_1}' - h_{n,a_1}|^{\alpha} + |h_{n,a_2}' - h_{n,a_2}|^{\alpha}\} = o\left( \frac{1}{ \log^2|\mathcal{A}_n|} \right), \\
    & \max_{(a_1,a_2)\in\mathcal{A}_n^2} \frac{1}{\sigma_{a_1}\sigma_{a_2}}\left|\Cov(\psi'_{a_1}, \psi'_{a_2}) - \frac{1}{2}\Cov(\psi_{a_1}, \psi_{a_2})\right|\\
    & \quad \lesssim \max_{(a_1,a_2)\in\mathcal{A}_n^2} \frac{1}{n^2}\frac{1}{\sigma_{a_1}\sigma_{a_2}}\left|\left( \frac{1}{h_{n,a_1}'^{d+2}} - \frac{1}{2h_{n,a_1}^{d+2}}\right) + \left( \frac{1}{h_{n,a_2}'^{d+2}} - \frac{1}{2h_{n,a_2}^{d+2}}\right) \right| \\
    & \quad\quad +  \max_{(a_1,a_2)\in\mathcal{A}_n^2} \frac{1}{n^2h_{n,a_1}^{(d+2)/2}h_{n,a_2}^{(d+2)/2}}\frac{1}{\sigma_{a_1}\sigma_{a_2}}\{|h_{n,a_1}' - h_{n,a_1}|^{\alpha} + |h_{n,a_2}' - h_{n,a_2}|^{\alpha}\} = o\left( \frac{1}{\log^2|\mathcal{A}_n|} \right) .
\end{align*}
The first convergence follows from \cref{eq:asymp-variance} and $\bar{h}_n^\alpha\log^2|\mathcal{A}_n| \to 0$ and the second does from \cref{eq:asymp-variance}, $h_n' = 2^{1/(d+2)}h_n$ and $\bar{h}_n^\alpha\log^2|\mathcal{A}_n| \to 0$.
These complete the proof of $(\Delta^\sharp.I.II) = o(\log^{-2}|\mathcal{A}_n|)$.

Next, we show the convergence of $(\Delta^\sharp.I.I)$.
Define $\bar{\psi}_{jk,h_n} \coloneqq (n-1) \psi_{jk,h_n}$.
Then, the covariance between $J^\sharp_{a_1}$ and $J^\sharp_{a_2}$ conditional on observations is decomposed into six terms likewise the proof of Theorem 2 in \cite{cattaneo2014small} as follows:
\begin{align*}
    &\sigma_{a_1}\sigma_{a_2}\Cov^*(J^\sharp)_{a_1,a_2} \\
    &=  \frac{1}{n^2}\sum_{i=1}^n \left[ \sum_{l\neq i}^n \psi_{a_1}'(Z_i,Z_l) - \hat{\theta}_{a_1}' \right]\left[ \sum_{l\neq i}^n \psi_{a_2}'(Z_i,Z_l) - \hat{\theta}_{a_2}' \right] \\
    &= \frac{1}{n^2}\sum_{i=1}^n \left[ \sum_{l\neq i}^n \psi_{a_1}'(Z_i,Z_l) - P \bar\psi_{a_1}'(Z_i) + P \bar\psi_{a_1}'(Z_i) - P^2 \bar\psi_{a_1}' +  P^2\bar\psi_{a_1}' - \hat{\theta}_{a_1}' \right] \\
    & \quad \times \left[ \sum_{l\neq i}^n \psi_{a_2}'(Z_i,Z_l)  - P \bar\psi_{a_2}'(Z_i) + P \bar\psi_{a_2}'(Z_i) - P^2 \bar\psi_{a_2}' +  P^2\bar\psi_{a_2}' - \hat{\theta}_{a_2}' \right] \\
    & = I_{a_1,a_2} + II_{a_1,a_2} + III_{a_1,a_2} + IV_{a_1,a_2} + V_{a_1,a_2} + VI_{a_1,a_2},
\end{align*}
where 
\begin{align*}
    & I_{a_1,a_2} \coloneqq \frac{1}{n^2} \sum_{i=1}^n \left[ \sum_{l\neq i}^n \psi_{a_1}'(Z_i,Z_l) - P \bar\psi_{a_1}'(Z_i) \right] \left[ \sum_{l\neq i}^n \psi_{a_2}'(Z_i,Z_l) - P \bar\psi_{a_2}'(Z_i) \right], \\
    & II_{a_1,a_2} \coloneqq \frac{1}{n^2} \sum_{i=1}^n \left[ P \bar\psi_{a_1}'(Z_i) - P^2 \bar\psi_{a_1}'\right] \left[ P \bar\psi_{a_2}'(Z_i) - P^2 \bar\psi_{a_2}'\right], \\
    & III_{a_1,a_2} \coloneqq \frac{1}{n^2} \sum_{i=1}^n \left[ P^2 \bar\psi_{a_1}' - \hat{\theta}_{a_1}'  \right]\left[ P^2 \bar\psi_{a_2}' -\hat{\theta}_{a_2}'  \right], \\
    & IV_{a_1,a_2} \coloneqq \frac{1}{n^2} \sum_{i=1}^n \left[ \sum_{l\neq i}^n \psi_{a_1}'(Z_i,Z_l) - P\bar \psi_{a_1}'(Z_i) \right] \left[ P \bar\psi_{a_2}'(Z_i) - P^2 \bar\psi_{a_2}'\right] \\
    & \quad\quad + \frac{1}{n^2} \sum_{i=1}^n \left[ \sum_{l\neq i}^n \psi_{a_2}'(Z_i,Z_l) - P\bar \psi_{a_2}'(Z_i) \right] \left[ P \bar\psi_{a_1}'(Z_i) - P^2\bar\psi_{a_1}'\right], \\
    & V_{a_1,a_2} \coloneqq \frac{1}{n^2} \sum_{i=1}^n  \left[ \sum_{l\neq i}^n \psi_{a_1}'(Z_i,Z_l) - P \bar\psi_{a_1}'(Z_i) \right] \left[ P^2 \bar\psi_{a_2}' - \hat{\theta}_{a_2}'  \right] \\
    &\quad\quad + \frac{1}{n^2} \sum_{i=1}^n  \left[ \sum_{l\neq i}^n \psi_{a_2}'(Z_i,Z_l) - P \bar\psi_{a_2}'(Z_i) \right] \left[ P^2 \bar\psi_{a_1}' - \hat{\theta}_{a_1}'  \right], \\
    & VI_{a_1,a_2} \coloneqq \frac{1}{n^2} \sum_{i=1}^n \left[ P \bar\psi_{a_1}'(Z_i) - P^2 \bar\psi_{a_1}'\right]\left[ P^2 \bar\psi_{a_2}' - \hat{\theta}_{a_2}'  \right] + \frac{1}{n^2} \sum_{i=1}^n \left[ P \bar\psi_{a_2}'(Z_i) - P^2 \bar\psi_{a_2}'\right]\left[ P^2 \bar\psi_{a_1}' - \hat{\theta}_{a_1}'  \right] .
\end{align*}
To show the convergence of $(\Delta^\sharp.I.I)$, we first bound the $\sigma^{-1}_{a_1}\sigma_{a_2}^{-1}III_{a_1,a_2}$, $\sigma^{-1}_{a_1}\sigma_{a_2}^{-1}IV_{a_1,a_2}$, $\sigma^{-1}_{a_1}\sigma_{a_2}^{-1}V_{a_1,a_2}$ and $\sigma^{-1}_{a_1}\sigma_{a_2}^{-1}VI_{a_1,a_2}$ uniformly in $(a_1,a_2)\in\mathcal{A}_n^2$ and then show the uniform convergence of $\sigma^{-1}_{a_1}\sigma_{a_2}^{-1}I_{a_1,a_2}$ to $\sigma^{-1}_{a_1}\sigma_{a_2}^{-1}\Cov(\psi'_{a_1}, \psi_{a_2}')$ and of $\sigma^{-1}_{a_1}\sigma_{a_2}^{-1}II_{a_1,a_2}$ to $n\sigma^{-1}_{a_1}\sigma_{a_2}^{-1}\Cov(P\psi_{a_1}', P\psi_{a_2}')$.

\subparagraph{Convergence of $III$:}
Observe that
\begin{align*}
    & III_{a_1,a_2} \\
    &\coloneqq \frac{1}{n^2} \sum_{i=1}^n \left[ P^2 \bar\psi_{a_1}' - \hat{\theta}_{a_1}'  \right]\left[ P^2 \bar\psi_{a_2}' -\hat{\theta}_{a_2}'  \right] \\
    &= \frac{1}{n} \left[ P^2 \bar\psi_{a_1}' - \frac{1}{n(n-1)}\sum_{i=1}^n\sum_{l\neq i}^n \bar\psi_{a_1}'(Z_i,Z_l) \right] \left[ P^2 \bar\psi_{a_2}' - \frac{1}{n(n-1)}\sum_{i=1}^n\sum_{l\neq i}^n \bar\psi_{a_2}'(Z_i,Z_l) \right] \\
    &= \frac{1}{n} \left[ \frac{1}{n}\sum_{i=1}^n \pi_1\bar\psi_{a_1}'(Z_i) + \binom{n}{2}^{-1}\sum_{i=1}^n\sum_{l> i}^n \pi_2\bar\psi_{a_1}'(Z_i,Z_l) \right]  \left[ \frac{1}{n}\sum_{i=1}^n \pi_1\bar\psi_{a_2}'(Z_i) + \binom{n}{2}^{-1}\sum_{i=1}^n\sum_{l> i}^n \pi_2\bar\psi_{a_2}'(Z_i,Z_l) \right].
\end{align*}
Therefore the Cauchy-Schwarz inequality and the triangle inequality give
\begin{align*}
    & \left\|\max_{(a_1,a_2)\in\mathcal{A}_n^2}\frac{1}{\sigma_{a_1}\sigma_{a_2}} |III_{a_1,a_2}|\right\|_{L^1(\pr)}\\
    & \le \frac{1}{n}\left\| \max_{a_1\in\mathcal{A}_n} \frac{1}{\sigma_{a_1}}\left| \frac{1}{n}J( \pi_1\bar\psi_{a_1}') + \frac{1}{n(n-1)}J_2( \pi_2\bar\psi_{a_1}')  \right| \right\|_{L^2(\pr)} \left\| \max_{a_2\in\mathcal{A}_n} \left| \frac{1}{n}  \frac{1}{\sigma_{a_2}}J( \pi_1\bar\psi_{a_2}') + \frac{1}{n(n-1)}J_2( \pi_2\bar\psi_{a_2}')  \right| \right\|_{L^2(\pr)} \\
    & =  \frac{1}{n}\left\| \max_{a_1\in\mathcal{A}_n} \frac{1}{\sigma_{a_1}} \left| \frac{1}{n}J( \pi_1\bar\psi_{a_1}') + \frac{1}{n(n-1)}J_2( \pi_2\bar\psi_{a_1}')  \right| \right\|_{L^2(\pr)}^2 \\
    & \le  \frac{2}{n}\left( \left\| \max_{a_1\in\mathcal{A}_n}  \frac{1}{n}\frac{1}{\sigma_{a_1}} |J_1( \pi_1\bar\psi_{a_1}')|\right\|_{L^2(\pr)}^2 + \left\|\max_{a_1\in\mathcal{A}_n} \frac{1}{n^2}\frac{1}{\sigma_{a_1}}|J_2(\pi_2\bar\psi_{a_1}')|   \right\|_{L^2(\pr)}^2 \right).
\end{align*}
Under Assumption \ref{as:DWAD-DGP}, we can see $|\pi_1 \bar{\psi}_{jk,h_n'}|^2(Z_i) \lesssim |Y_{ij}|^2 + 1$ for each $(j,k,h_n)\in\mathcal{A}_n$ from Lemma 1(d) in \cite{NiRo00},
then Theorem \ref{thm:max-is}, \eqref{eq:max-out}, \eqref{eq:asymp-variance} and $\log^{1/2}|\mathcal{A}_n| = o(\sqrt{n})$ (Assumption \ref{as:DWAD-bandwdith}) give
\begin{align}
    & \left\| \max_{a_1\in\mathcal{A}_n} \frac{1}{n} \frac{1}{\sigma_{a_1}} |J_1( \pi_1\bar\psi_{a_1}')|\right\|_{L^2(\pr)} \nonumber\\
     &\lesssim \log^{1/2}|\mathcal{A}_n| \left(  \max_{a_1\in\mathcal{A}_n} P|\pi_1 \bar\psi_{a_1}'|^2  \right)^{1/2} +\frac{1}{\sqrt{n}} \log|\mathcal{A}_n|\left\|  \max_{a_1\in\mathcal{A}_n} M(|\pi_1 \bar\psi_{a_1}'|^2)  \right\|^{1/2}_{L^1(\pr)} \nonumber\\
     & \le  \log^{1/2}|\mathcal{A}_n| \left(  \max_{a_1\in\mathcal{A}_n} P|\pi_1 \bar\psi_{a_1}'|^2  \right)^{1/2} + \frac{1}{\sqrt{n}} \log|\mathcal{A}_n|\left\|  \max_{a_1\in\mathcal{A}_n} |\pi_1 \bar\psi_{a_1}'|^2  \right\|_{L^\infty(P)}^{1/2} \nonumber \\
     & \lesssim \log^{1/2}|\mathcal{A}_n|  + \frac{\log|\mathcal{A}_n|}{\sqrt{n}}  = O\left( \log^{1/2}|\mathcal{A}_n| \right). \label{eq:DWAD-max-hajek}
\end{align}  
Also, Theorem \ref{thm:max-is} and \eqref{eq:max-out} give
\begin{align*}
    \left\|\max_{a_1\in\mathcal{A}_n} \frac{1}{n^2}\frac{1}{\sigma_{a_1}} |J_2(\pi_2\bar\psi_{a_1}')|  \right\|_{L^2(\pr)}  
    & \lesssim \max_{0\le s\le 2} \frac{n^{1-s/2}}{n^2} \log^{1+s/2}|\mathcal{A}_n| \left\|  \max_{a_1\in\mathcal{A}_n} \frac{1}{\sigma_{a_1}^2}M(P^{2-s}|\pi_2 \bar\psi_{a_1}'|^2)  \right\|_{L^1(\pr)}^{1/2} \\
    & \le \max_{0\le s\le 2} \frac{n^{1-s/2}}{n^2} \log^{1+s/2}|\mathcal{A}_n| \left\|  \max_{a_1\in\mathcal{A}_n} \frac{1}{\sigma_{a_1}^2}P^{2-s}|\pi_2 \bar\psi_{a_1}'|^2  \right\|_{L^\infty(P^s)}^{1/2} .
\end{align*}
 Lemma 5 in \cite{Ro95} states that $P|\pi_2\bar\psi_{jk,h_n}|^2(Z_i) \lesssim h_n^{-(d+2)} (Y_{ij}^2 + 1) $ for each $(j,k)\in[p]\times [d]$, which also implies $P^2|\pi_2\bar\psi_{jk,h_n}|^2 \lesssim h_n^{-(d+2)}$ for each $j\in[p]$ and $k\in[d]$.
Also, the first line of the proof of Lemma 5 in \cite{Ro95} and Assumption \ref{as:DWAD-Kernel} give 
\[
|\pi_2\bar\psi_{jk,h_n}|^2(Z_i,Z_l) \lesssim (Y_{ij}^2 + Y_{lj}^2)\{(Y_{ij} - Y_{lj})\partial_kK_{il,h_n'}\}^2 h_n^{-2(d+1)}
\]
for all $(j,k)\in[p]\times [d]$. 
Therefore, $ \left\|  \max_{a_1\in\mathcal{A}_n} \sigma_{a_1}^{-2}|\pi_2 \bar\psi_{a_1}'|^2  \right\|_{L^\infty(P^2)}^{1/2} = O(n \ul{h}_n^{-d/2}) $.
In conjunction with $n^{-1/2}\log^{1/2}|\mathcal{A}_n|\to 0$ and $n^{-1}\underline{h}_n^{-d/2} \log|\mathcal{A}_n| \to 0$ (Assumption \ref{as:DWAD-bandwdith}), we have 
\begin{align}
     & \left\|\max_{a_1\in\mathcal{A}_n} \frac{1}{n^2}\frac{1}{\sigma_{a_1}}|J_2(\pi_2\bar\psi_{a_1}')|  \right\|_{L^2(\pr)} \nonumber \\
     & \lesssim  \log|\mathcal{A}_n| \vee  \frac{\log^{3/2}|\mathcal{A}_n|}{n^{1/2}} \vee \frac{\log^2|\mathcal{A}_n|}{n\underline{h}_n'^{d/2}} \nonumber \\
     &=  \log|\mathcal{A}_n| \vee  \log|\mathcal{A}_n| \cdot \frac{\log^{1/2}|\mathcal{A}_n|}{n^{1/2}} \vee \log|\mathcal{A}_n| \cdot \frac{\log|\mathcal{A}_n|}{n\underline{h}_n'^{d/2}}  = O\left( \log|\mathcal{A}_n|\right). \label{eq:DWAD-max-second}
\end{align}
Summing up, since $\underline{h}_n' \ge \underline{h}_n$ and $n^{-1}\log^4 |\mathcal{A}_n| \to 0$ (Assumption \ref{as:DWAD-bandwdith}), it holds that
\begin{align}
    & \left\|\max_{(a_1,a_2)\in\mathcal{A}_n^2}  \frac{1}{\sigma_{a_1}\sigma_{a_2}}|III_{a_1,a_2}|\right\|_{L^1(\pr)}\lesssim \frac{1}{n} \log^2|\mathcal{A}_n| =  \frac{\log^4 |\mathcal{A}_n|}{n}\cdot \frac{1}{ \log^2|\mathcal{A}_n|} = o\left( \frac{1}{ \log^2|\mathcal{A}_n|}\right) . \label{eq:convergence-III}
\end{align}

\subparagraph{Convergence of $IV$:}
Since $\psi_{a}'(Z_i,Z_l) - P\psi_{a}'(Z_i) = \pi_2\psi_{a}'(Z_i,Z_l) + \pi_1\psi_{a}'(Z_l)$ and $P\bar \psi_{a}'(Z_i) - P^2\bar\psi_{a}' = \pi_1 \bar\psi_{a}'(Z_i)$,
\begin{align*}
    IV_{a_1,a_2} &\coloneqq \frac{1}{n^2} \sum_{i=1}^n \left[ \sum_{l\neq i}^n \psi_{a_1}'(Z_i,Z_l) - P\bar \psi_{a_1}'(Z_i) \right] \left[ P \bar\psi_{a_2}'(Z_i) - P^2 \bar\psi_{a_2}'\right] \\
    & \quad\quad + \frac{1}{n^2} \sum_{i=1}^n \left[ \sum_{l\neq i}^n \psi_{a_2}'(Z_i,Z_l) - P\bar \psi_{a_2}'(Z_i) \right] \left[ P \bar\psi_{a_1}'(Z_i) - P^2\bar\psi_{a_1}'\right] \\
    & = \frac{1}{n^2} \sum_{i=1}^n \left[ \sum_{l\neq i}^n \{ \pi_2\psi_{a_1}'(Z_i,Z_l) + \pi_1\psi_{a_1}'(Z_l)\} \right] [\pi_1\bar\psi_{a_2}'(Z_i) ] \\
    & \quad\quad + \frac{1}{n^2} \sum_{i=1}^n \left[ \sum_{l\neq i}^n \{ \pi_2\psi_{a_2}'(Z_i,Z_l) + \pi_1\psi_{a_2}'(Z_l)\} \right] [\pi_1\bar\psi_{a_1}'(Z_i) ] \\
    & = \frac{1}{n^2} \sum_{(i,l)\in I_{n,2}} (\pi_2\psi_{a_1}' \star_1^0 \pi_1\bar\psi_{a_2}')(Z_i,Z_l) +  \frac{1}{n^2} \sum_{(i,l)\in I_{n,2}}(\pi_1\psi_{a_1}'\star_0^0\pi_1\bar\psi_{a_2}')(Z_i,Z_l) \\
    & \quad\quad + \frac{1}{n^2} \sum_{(i,l)\in I_{n,2}} (\pi_2\psi_{a_2}' \star_1^0\pi_1\bar\psi_{a_1}')(Z_i,Z_l) +  \frac{1}{n^2} \sum_{(i,l)\in I_{n,2}}(\pi_1\psi_{a_2}'\star_0^0\pi_1\bar\psi_{a_2}')(Z_i,Z_l) \\
    & \eqqcolon (IV.I)_{a_1,a_2} + (IV.II)_{a_1,a_2} + (IV.III)_{a_1,a_2} + (IV.IV)_{a_1,a_2}.
\end{align*}
For $(IV.I)_{a_1,a_2}$, letting $(\pi_2\bar\psi_{a_1}' \star_1^0 \tilde{\pi}_1 \bar{\psi}_{a_2}')(Z_i,Z_l) \coloneqq \frac{n-1}{2}\pi_2\psi_{a_1}'(Z_i,Z_l)\{\pi_1\bar\psi_{a_2}'(Z_i) + \pi_1\bar\psi_{a_2}'(Z_l)\}$, 
\begin{align*}
    (IV.I)_{a_1,a_2} = \frac{1}{n^2(n-1)} \sum_{(i,l)\in I_{n,2}}(\pi_2\bar\psi_{a_1}' \star_1^0\tilde{\pi}_1 \bar{\psi}_{a_2}')(Z_i,Z_l) \asymp \frac{1}{n^3} J_2(\pi_2\bar\psi_{a_1}' \star_1^0\tilde{\pi}_1 \bar{\psi}_{a_2}').
\end{align*}
Then, from Theorem \ref{thm:max-is} and \eqref{eq:max-out}, we have
\begin{align*}
    & \left\|\max_{(a_1,a_2)\in\mathcal{A}_n^2}\frac{1}{\sigma_{a_1}\sigma_{a_2}} |(IV.I)_{a_1,a_2}| \right\|_{L^1(\pr)} \\
    & \lesssim  \max_{0\le s\le 2} \frac{n^{1-s/2}}{n^3} \log^{1+s/2}|\mathcal{A}_n| \left\|  \max_{(a_1,a_2)\in\mathcal{A}_n^2} \frac{1}{\sigma_{a_1}^2\sigma_{a_2}^2} M(P^{2-s}|\pi_2\bar\psi_{a_1}' \star_1^0\tilde{\pi}_1 \bar{\psi}_{a_2}'|^2)  \right\|_{L^1(\pr)}^{1/2} \\
    & \le \max_{0\le s\le 2} \frac{n^{1-s/2}}{n^3} \log^{1+s/2}|\mathcal{A}_n| \left\|  \max_{(a_1,a_2)\in\mathcal{A}_n^2}\frac{1}{\sigma_{a_1}^2\sigma_{a_2}^2} P^{2-s}|\pi_2\bar\psi_{a_1}' \star_1^0\tilde{\pi}_1 \bar{\psi}_{a_2}'|^2  \right\|_{L^\infty(P^s)}^{1/2} .
\end{align*}
The first line of the proof of Lemma 5 in \cite{Ro95} and Lemma 1(d) in \cite{NiRo00} imply 
\begin{align*}
    |\pi_2\bar\psi_{a_1}' \star_1^0 \tilde{\pi}_1 \bar{\psi}_{a_2}'|^2(Z_i,Z_l) 
    & = |\pi_2\bar\psi_{a_1}'|^2(Z_i,Z_l)|\pi_1 \bar{\psi}_{a_2}'|^2 (Z_i) \\
    & \lesssim (Y_{ij_{a_1}}^2 + Y_{lj_{a_1}}^2)(Y_{ij_{a_2}}^2 + Y_{lj_{a_2}}^2 + 2)\{(Y_{ij_{a_1}} - Y_{lj_{a_1}})\partial_{k_{a_1}}K_{il,h_{n,a_1}'}\}^2 h_{n,a_1}'^{-2(d+1)}
\end{align*}
for all $(a_1,a_2)\in\mathcal{A}_n^2$.
Then, from \eqref{eq:asymp-variance} and some straightforward computations of expectations, we have $\left\| \max_{(a_1,a_2)\in\mathcal{A}_n^2} \sigma_{a_1}^{-2}\sigma_{a_2}^{-2} |\pi_2\bar\psi_{a_1}' \star_1^0\tilde{\pi}_1 \bar{\psi}_{a_2}'|^2  \right\|_{L^\infty(P^2)}^{1/2} \lesssim n^2\ul{h}_n^{-d/2}\ol{h}_n^{(d+2)/2}$ and $P^{2-s}|\pi_2\bar\psi_{a_1}' \star_1^0 \tilde{\pi}_1 \bar{\psi}_{a_2}'|^2 \lesssim {h}_{n,a_1}'^{-(d+2)}$ for  all $(a_1,a_2)\in\mathcal{A}^2_n$ and each $s=0,1$.
Then, in conjunction with $n^{-1/2}\log^{7/2}|\mathcal{A}_n| \to 0 $, $\bar{h}_n^{d/2}\log^5|\mathcal{A}_n|\to 0$ and $n^{-1}\underline{h}_n^{-d/2} \log|\mathcal{A}_n| \to 0$ (Assumption \ref{as:DWAD-bandwdith}),
\begin{align*}
    & \left\|{\max_{(a_1,a_2)\in\mathcal{A}_n^2}}\frac{1}{\sigma_{a_1}\sigma_{a_2}} |(IV.I)_{a_1,a_2}| \right\|_{L^1(\pr)} \\
    & \lesssim\overline{h}_n'^{(d+2)/2}\log |\mathcal{A}_n| \vee \frac{\overline{h}_n'^{(d+2)/2}\log^{3/2} |\mathcal{A}_n| }{n^{1/2} }  \vee \frac{\ol{h}_n^{(d+2)/2}\log^2|\mathcal{A}_n|}{n \ul{h}_n^{d/2}} \\
    & = \frac{1}{\log^2|\mathcal{A}_n|} \left( \overline{h}_n'^{(d+2)/2}\log^3|\mathcal{A}_n| \vee \frac{\overline{h}_n'^{(d+2)/2}\log^{7/2} |\mathcal{A}_n| }{n^{1/2} }  \vee \ol{h}_n'^{(d+2)/2}\log^3|\mathcal{A}_n|\cdot \frac{\log |A_n|}{n \ul{h}_n^{d/2}} \right) = o\left( \frac{1}{\log^2|\mathcal{A}_n|}  \right).
\end{align*}
Similarly, $\left\|\max_{(a_1,a_2)\in\mathcal{A}_n^2} \sigma_{a_1}^{-2}\sigma_{a_2}^{-2}|(IV.III)_{a_1,a_2} |\right\|_{L^1(\pr)}  = o(\log^{-2}|\mathcal{A}_n|)$.
In a similar way, for $ (IV.II)_{a_1,a_2} $ and $(IV.IV)_{a_1,a_2}$, using Theorem \ref{thm:max-is}, $|\pi_1 \bar{\psi}_{jk,h_n'}|^2(Z_i)\lesssim (Y_{ij}^2 + 1)$ for all $(j,k)\in[p] \times [d]$ and $n^{-1}\log^3|\mathcal{A}_n|\to 0$, we have 
\begin{align*}
     &\left\| \max_{(a_1,a_2)\in\mathcal{A}_n^2}\frac{1}{\sigma_{a_1}\sigma_{a_2}}|(IV.II)_{a_1,a_2} |\right\|_{L^1(\pr)} \lesssim \frac{\log|\mathcal{A}_n|}{n} = o\left( \frac{1}{\log^2|\mathcal{A}_n|}  \right),\\  &\left\|\max_{(a_1,a_2)\in\mathcal{A}_n^2}\frac{1}{\sigma_{a_1}\sigma_{a_2}}|(IV.IV)_{a_1,a_2} |\right\|_{L^1(\pr)} \lesssim \frac{\log |\mathcal{A}_n|}{n} = o\left( \frac{1}{\log^2|\mathcal{A}_n|}  \right).
\end{align*}
Summing up, we have
\begin{align}
     & \left\|\max_{(a_1,a_2)\in\mathcal{A}_n^2} \frac{1}{\sigma_{a_1}\sigma_{a_2}}|IV_{a_1,a_2}| \right\|_{L^1(\pr)} = o\left( \frac{1}{\log^2|\mathcal{A}_n|}\right). \label{eq:convergence-IV}
\end{align}

\subparagraph{Convergence of $V$:}
\begin{align*}
    V_{a_1,a_2} &\coloneqq \frac{1}{n^2} \sum_{i=1}^n  \left[ \sum_{l\neq i}^n \psi_{a_1}'(Z_i,Z_l) - P \bar\psi'_{a_1}(Z_i) \right] \left[ P^2 \bar\psi_{a_2}' - \hat{\theta}_{a_2}'  \right] \\
    &\quad\quad + \frac{1}{n^2} \sum_{i=1}^n  \left[ \sum_{l\neq i}^n \psi_{a_2}'(Z_i,Z_l) - P \bar\psi_{a_2}'(Z_i) \right] \left[ P^2 \bar\psi_{a_1}' - \hat{\theta}_{a_1}'  \right] \\
    & = -\frac{1}{n^2} \left[ \sum_{i=1}^n  \sum_{l\neq i}^n \{ \pi_2\psi_{a_1}'(Z_i,Z_l) + \pi_1\psi_{a_1}'(Z_l)\}\right]  \left[ \frac{1}{n}\sum_{i=1}^n \pi_1\bar\psi_{a_2}'(Z_i) + \binom{n}{2}^{-1}\sum_{i=1}^n\sum_{l> i}^n \pi_2\bar\psi_{a_2}'(Z_i,Z_l) \right] \\
    & \quad - \frac{1}{n^2} \left[ \sum_{i=1}^n  \sum_{l\neq i}^n \{ \pi_2\psi_{a_2}'(Z_i,Z_l) + \pi_1\psi_{a_2}'(Z_l)\}\right]  \left[ \frac{1}{n}\sum_{i=1}^n \pi_1\bar\psi_{a_1}'(Z_i) + \binom{n}{2}^{-1}\sum_{i=1}^n\sum_{l> i}^n \pi_2\bar\psi_{a_1
    }'(Z_i,Z_l) \right] 
\end{align*}
Therefore the Cauchy-Schwarz inequality and the triangle inequality give
\begin{align*}
   & \left\| \max_{(a_1,a_2)\in\mathcal{A}_n^2} \frac{1}{\sigma_{a_1}\sigma_{a_2}}| V_{a_1,a_2}|\right\|_{L^1(\pr)} 
   \lesssim  \frac{1}{n} \left( \left\| \max_{a_1\in\mathcal{A}_n}  \frac{1}{n}\frac{1}{\sigma_{a_1}} |J_1( \pi_1\bar\psi_{a_1}') |\right\|^2_{L^2(\pr)} + \left\|\max_{a_1\in\mathcal{A}_n} \frac{1}{n^2} \frac{1}{\sigma_{a_1}}|J_2(\pi_2\bar\psi_{a_1}') |  \right\|^2_{L^2(\pr)} \right)
\end{align*}
Since $n^{-1}\log^4 |\mathcal{A}_n| \to 0$ holds under Assumption \ref{as:DWAD-bandwdith}, in conjunction with \eqref{eq:DWAD-max-hajek} and \eqref{eq:DWAD-max-second}, we have 
\begin{align}
    &\left\| \max_{(a_1,a_2)\in\mathcal{A}_n^2} \frac{1}{\sigma_{a_1}\sigma_{a_2}} |V_{a_1,a_2}|\right\|_{L^1(\pr)}  \lesssim\frac{\log^4|\mathcal{A}_n|}{n}  = o\left( \frac{1}{\log^2|\mathcal{A}_n|} \right). \label{eq:convergence-V}
\end{align}

\subparagraph{Convergence of $VI$:}
\begin{align*}
    VI_{a_1,a_2} &\coloneqq \frac{1}{n^2} \sum_{i=1}^n \left[ P \bar\psi_{a_1}'(Z_i) - P^2 \bar\psi_{a_1}'\right]\left[ P^2 \bar\psi_{a_2}' - \hat{\theta}_{a_2}'  \right] + \frac{1}{n^2} \sum_{i=1}^n \left[ P \bar\psi_{a_2}'(Z_i) - P^2 \bar\psi_{a_2}'\right]\left[ P^2 \bar\psi_{a_1}' - \hat{\theta}_{a_1}'  \right]  \\
    &= \frac{-1}{n^2} \left[\sum_{i=1}^n \pi_1 \bar\psi_{a_1}'(Z_i)\right] \left[ \frac{1}{n}\sum_{i=1}^n \pi_1\bar\psi_{a_2}'(Z_i) + \binom{n}{2}^{-1}\sum_{i=1}^n\sum_{l> i}^n \pi_2\bar\psi_{a_2}'(Z_i,Z_l) \right] \\
    & \quad - \frac{1}{n^2} \left[\sum_{i=1}^n \pi_1 \bar\psi_{a_2}'(Z_i)\right] \left[ \frac{1}{n}\sum_{i=1}^n \pi_1\bar\psi_{a_1}'(Z_i) + \binom{n}{2}^{-1}\sum_{i=1}^n\sum_{l> i}^n \pi_2\bar\psi_{a_1}'(Z_i,Z_l) \right] 
\end{align*}
Therefore the Cauchy-Schwarz inequality and the triangle inequality give
\begin{align*}
    & \left\|\max_{(a_1,a_2)\in\mathcal{A}_n^2} \frac{1}{\sigma_{a_1}\sigma_{a_2}}|VI_{a_1,a_2}|\right\|_{L^1(\pr)} \\
    & \lesssim  \frac{1}{n}\left\| \max_{a_1\in\mathcal{A}_n}  \frac{1}{n} \frac{1}{\sigma_{a_1}}|J_1( \pi_1\bar\psi_{a_1}') |\right\|^2_{L^2(\pr)} +  \frac{1}{n}\left\| \max_{a_1\in\mathcal{A}_n}  \frac{1}{n}\frac{1}{\sigma_{a_1}}| J_1( \pi_1\bar\psi_{a_1}') |\right\|_{L^2(\pr)}\left\|\max_{a_2\in\mathcal{A}_n} \frac{1}{n^2}\frac{1}{\sigma_{a_2}}|J_2(\pi_2\bar\psi_{a_2}') | \right\|_{L^2(\pr)} . 
\end{align*}
Since $n^{-1}\log^{7/2} |\mathcal{A}_n| \to 0$ holds under Assumption \ref{as:DWAD-bandwdith}, in conjunction with \eqref{eq:DWAD-max-hajek} and \eqref{eq:DWAD-max-second}, we have 
\begin{align}
     & \left\|\max_{(a_1,a_2)\in\mathcal{A}_n^2} \frac{1}{\sigma_{a_1}\sigma_{a_2}}|VI_{a_1,a_2}|\right\|_{L^1(\pr)} \lesssim 
     \frac{\log^{3/2} |\mathcal{A}_n| }{n}= o\left( \frac{1}{\log^2 |\mathcal{A}_n|}\right). \label{eq:convergence-VI}
\end{align}

\subparagraph{Convergence of $I$:}
Since $\psi_{a}'(Z_i,Z_l) - P\psi_{a}'(Z_i) = \pi_2\psi_{a}'(Z_i,Z_l) + \pi_1\psi_{a}'(Z_l)$,
\begin{align*}
    I_{a_1,a_2}  &\coloneqq \frac{1}{n^2} \sum_{i=1}^n \left[ \sum_{l\neq i}^n \psi_{a_1}(Z_i,Z_l) - P \bar\psi_{a_1}'(Z_i) \right] \left[ \sum_{l\neq i}^n \psi_{a_2}'(Z_i,Z_l) - P \bar\psi_{a_2}'(Z_i) \right] \\
    &= \frac{1}{n^2}\sum_{i=1}^n \left[ \sum_{l\neq i}^n \{\pi_2\psi_{a_1}'(Z_i,Z_l) + \pi_1\psi_{a_1}'(Z_l)\} \right] \left[ \sum_{l\neq i}^n \{\pi_2\psi_{a_2}'(Z_i,Z_l) +  \pi_1\psi_{a_2}'(Z_l)\} \right] \\
    & = \frac{1}{n^2}\sum_{(i,l)\in I_{n,2}}\left[  \{\pi_2\psi_{a_1}'(Z_i,Z_l) +  \pi_1\psi_{a_1}'(Z_l)\}\{\pi_2\psi_{a_2}'(Z_i,Z_l) +  \pi_1\psi_{a_2}'(Z_l)\} \right] \\
    & \quad + \frac{1}{n^2} \sum_{(i,l,m) \in I_{n,3}} \left[  \{\pi_2\psi_{a_1}'(Z_i,Z_l) + \pi_1\psi_{a_1}'(Z_l)\}\{\pi_2\psi_{a_2}'(Z_i,Z_m) +  \pi_1\psi_{a_2}'(Z_m)\} \right] \\
    & \eqqcolon (I.I)_{a_1,a_2} + (I.II)_{a_1,a_2}.
\end{align*} 
Define $(I.I')_{a_1,a_2} \coloneqq \frac{1}{n^2}\sum_{i=1}^n \sum_{l\neq i}^n (\pi_2\psi_{a_1}' \star_2^0 \pi_2\psi_{a_2}')(Z_i, Z_l) $.
From \eqref{eq:eval-DWAD-hajek}, we can see that
\begin{align*}
     & \frac{1}{\sigma_{a_1}\sigma_{a_2}}\E[(\pi_2\psi_{a_1}' \star^0_2 \pi_2\psi_{a_2}')(Z_i, Z_l) ] 
     = \frac{1}{\sigma_{a_1}\sigma_{a_2}}\Cov(\psi_{a_1}',\psi_{a_2}') + O\left( \frac{1}{n}\right),
\end{align*}
for all $(a_1,a_2)\in\mathcal{A}_n^2$.
Therefore it is sufficient to show that $\sigma_{a_1}^{-1}\sigma_{a_2}^{-1}\{(I.I)_{a_1,a_2} - (I.I')_{a_1,a_2}\}$ and $\sigma_{a_1}^{-1}\sigma_{a_2}^{-1}(I.II)_{a_1,a_2}$ are negligible uniformly in $(a_1,a_2)\in\mathcal{A}_n^2$.
First, we show the difference $\sigma_{a_1}^{-1}\sigma_{a_2}^{-1}\{(I.I)_{a_1,a_2} - (I.I')_{a_1,a_2}\}$ is uniformly negligible.
Observe that
\begin{align*}
    (I.I)_{a_1,a_2} - (I.I')_{a_1,a_2}  &\asymp \frac{1}{n^2} \sum_{(i,l)\in I_{n,2}} (\pi_2\psi_{a_1}' \star^0_1 \pi_1\psi_{a_2}')(Z_i,Z_l)  \\
    & \quad+ \frac{1}{n^2} \sum_{(i,l)\in I_{n,2}} (\pi_2\psi_{a_2}' \star^0_1\pi_1\psi_{a_1}')(Z_i, Z_l)  + \frac{1}{n}\sum_{i=1}^n  (\pi_1\psi_{a_1}' \star_1^0 \pi_1\psi_{a_2}')(Z_i) \\
    & = \frac{1}{n^2} J_2(\pi_2\psi_{a_1}'\star_1^0 \tilde{\pi}_1\psi_{a_2}') + \frac{1}{n^2} J_2(\pi_2\psi_{a_2}' \star_1^0\tilde{\pi}_1\psi_{a_1}') + \frac{1}{n} J_1(\pi_1\psi_{a_1}'\star_1^0\pi_1\psi_{a_2}') \\
    & \asymp \frac{1}{n^4} J_2(\pi_2\bar\psi_{a_1}'\star_1^0 \tilde{\pi}_1\bar\psi_{a_2}') + \frac{1}{n^4} J_2(\pi_2\bar\psi_{a_2}' \star_1^0\tilde{\pi}_1\bar\psi_{a_1}') + \frac{1}{n^3} J_1(\pi_1\bar\psi_{a_1}'\star_1^0\pi_1\bar\psi_{a_2}') .
\end{align*}
Thus, from Theorem \ref{thm:max-is}, \eqref{eq:max-out}, we have
\begin{align*}
    & \left\| \max_{(a_1,a_2)\in\mathcal{A}_n^2 }\frac{1}{\sigma_{a_1}\sigma_{a_2}}|(I.I)_{a_1,a_2} - (I.I')_{a_1,a_2}| \right\|_{L^1(\pr)}  \\
    & \lesssim  \max_{0\le s\le 2} \frac{n^{1-s/2}}{n^4} \log^{1+s/2}|\mathcal{A}_n| \left\|  \max_{(a_1,a_2)\in\mathcal{A}_n^2}\frac{1}{\sigma_{a_1}^2\sigma_{a_2}^2} M(P^{2-s}|\pi_2\bar\psi_{a_1}' \star_1^0\tilde{\pi}_1 \bar{\psi}_{a_2}'|^2)  \right\|_{L^1(\pr)}^{1/2}\\
    & \quad + \max_{0\le s\le 1} \frac{n^{(1-s)/2}}{n^3} \log^{(1+s)/2}|\mathcal{A}_n| \left\|\max_{(a_1,a_2)\in\mathcal{A}_n^2}\frac{1}{\sigma_{a_1}^2\sigma_{a_2}^2} M(P^{1-s}|\pi_1\bar\psi_{a_1}'\star_1^0\pi_1\bar\psi_{a_2}'|^2)\right\|^{1/2}_{L^1(\pr)}\\
    & \lesssim \max_{0\le s\le 2} \frac{n^{1-s/2}}{n^4} \log^{1+s/2}|\mathcal{A}_n| \left\|  \max_{(a_1,a_2)\in\mathcal{A}_n^2} \frac{1}{\sigma_{a_1}^2\sigma_{a_2}^2} P^{2-s}|\pi_2\bar\psi_{a_1}' \star_1^0\tilde{\pi}_1 \bar{\psi}_{a_2}'|^2  \right\|_{L^\infty(P^s)}^{1/2} \\
    & \quad +  \max_{0\le s\le 1} \frac{n^{(1-s)/2}}{n^3} \log^{(1+s)/2}|\mathcal{A}_n| \left\|\max_{(a_1,a_2)\in\mathcal{A}_n^2} \frac{1}{\sigma_{a_1}^2\sigma_{a_2}^2} P^{1-s}|\pi_1\bar\psi_{a_1}'\star_1^0\pi_1\bar\psi_{a_2}'|^2\right\|^{1/2}_{L^\infty(P^s)}  .
\end{align*}
As stated above, $\left\| \max_{(a_1,a_2)\in\mathcal{A}_n^2} \sigma_{a_1}^{-2}\sigma_{a_2}^{-2} |\pi_2\bar\psi_{a_1}' \star_1^0\tilde{\pi}_1 \bar{\psi}_{a_2}'|^2  \right\|_{L^\infty(P^2)}^{1/2} \lesssim n^2\ul{h}_n^{-d/2}\ol{h}_n^{(d+2)/2}$ and $P^{2-s}|\pi_2\bar\psi_{a_1}' \star_1^0 \tilde{\pi}_1 \bar{\psi}_{a_2}'|^2 \lesssim h_{n,a_1}'^{-(d+2)}$ for  all $(a_1,a_2)\in\mathcal{A}^2_n$ and each $s=0,1$.
Also, from Lemma 1(d) in \cite{NiRo00}, we have  $|\pi_1\psi_{a_1}'\star_1^0\pi_1\psi_{a_2}'|^2(Z_i) = |\pi_1\psi_{a_1}'|^2(Z_i)|\pi_1\psi_{a_2}'|^2(Z_i) \lesssim (Y_{ij_1}^2 + 1)(Y_{ij_2}^2 + 1)$ for all $(a_1,a_2)\in\mathcal{A}_n^2$, so  $P^{1-s}|\pi_1\psi_{a_1}'\star_1^0\pi_1\psi_{a_2}'|^2 \lesssim 1$ for all $(a_1,a_2)\in\mathcal{A}_n^2$ and each $0 \le s \le 1$.
Therefore, under \cref{as:DWAD-bandwdith}, we can see that
\begin{align}
    & \left\| \max_{(a_1,a_2)\in\mathcal{A}_n^2 }\frac{1}{\sigma_{a_1}\sigma_{a_2}}|(I.I)_{a_1,a_2} - (I.I')_{a_1,a_2}| \right\|_{L^1(\pr)} \nonumber\\
    & \lesssim \frac{\ol{h}_n^{(d+2)/2}\log|\mathcal{A}_n|}{n} \vee \frac{\ol{h}_n^{(d+2)/2}\log^2|\mathcal{A}_n|}{n^{3/2}} \vee \frac{\ol{h}_n^{(d+2)/2}\log^{1/2}|\mathcal{A}_n|}{n^2\ul{h}_n^{d/2}} + \frac{\log^{1/2}|\mathcal{A}_n|}{n^{3/2}} \vee \frac{\log|\mathcal{A}_n|}{n^2}\nonumber\\
    & = o\left(\frac{1}{\log^2|\mathcal{A}_n|}\right). \label{eq:convergence-I.I}
\end{align}
Next, we show $(I.II)_{a_1,a_2}$ is also uniformly negligible.
Define $\widetilde{\pi_2\psi_{a_1}'\star_1^0 \pi_2\psi_{a_2}'}(Z_i,Z_l,Z_m) \coloneqq \frac{1}{3}\{\pi_2\psi_{a_1}'(Z_i,Z_l)\pi_2\psi_{a_2}'(Z_i,Z_m)+\pi_2\psi_{a_1}'(Z_l,Z_m)\pi_2\psi_{a_2}'(Z_l,Z_i)+\pi_2\psi_{a_1}'(Z_m,Z_i)\pi_2\psi_{a_2}'(Z_m,Z_l)\}$.
\begin{align*}
    (I.II)_{a_1,a_2} 
    &\asymp \frac{1}{n^2} J_3(\wt{\pi_2\psi_{a_1}'\star_1^0 \pi_2\psi_{a_2}'}) \\
    & \quad +  \frac{1}{n^2} J_2(\pi_2 \psi_{a_1}') J_1(\pi_1\psi_{a_2}') + \frac{1}{n^2} J_2(\pi_2 \psi_{a_2}') J_1(\pi_1\psi_{a_1}') + \frac{1}{n^2}J_1(\pi_1\psi_{a_1}')J_1(\pi_1\psi_{a_2}')\\
    & \asymp \frac{1}{n^4}J_{3}(\wt{\pi_{2}\bar{\psi}_{a_{1}}^{\prime}\star_{1}^{0}\pi_{2}\bar{\psi}_{a_{2}}^{\prime}})\\
    & \quad + \frac{1}{n^4} J_2(\pi_2 \bar\psi_{a_1}') J_1(\pi_1\bar\psi_{a_2}') + \frac{1}{n^4} J_2(\pi_2\bar\psi_{a_2}') J_1(\pi_1\bar\psi_{a_1}') + \frac{1}{n^4}J_1(\pi_1\bar\psi_{a_1}')J_1(\pi_1\bar\psi_{a_2}').
\end{align*}
From the Cauchy-Schwarz inequality, the triangle inequality, Theorem \ref{thm:max-is} and \eqref{eq:max-out}, we have
\begin{align*}
    & \left\|\max_{(a_1,a_2)\in\mathcal{A}_n^2} \frac{1}{\sigma_{a_1}\sigma_{a_2}}| (I.II)_{a_1,a_2}| \right\|_{L^1(\pr)}  \\
    & \lesssim \frac{1}{n^4} \left\|\max_{(a_1,a_2)\in\mathcal{A}_n^2} \frac{1}{\sigma_{a_1}\sigma_{a_2}}|J_3(\widetilde{\pi_2\bar{\psi}_{a_1}'\star_1^0 \pi_2\bar\psi_{a_2}'}) |   \right\|_{L^1(\pr)}  \\
    & \quad + \frac{1}{n^4} \left\|\max_{a_1\in\mathcal{A}_n} \frac{1}{\sigma_{a_1}} | J_2(\pi_2 \bar\psi_{a_1}')   |\right\|_{L^2(\pr)}\left\|\max_{a_2\in\mathcal{A}_n} \frac{1}{\sigma_{a_2}}|J_1(\pi_1\bar\psi_{a_2}')|  \right\|_{L^2(\pr)} + \frac{1}{n^4} \left\|\max_{a_1\in\mathcal{A}_n}   \frac{1}{\sigma_{a_1}}| J_1(\pi_1\bar\psi_{a_1}')| \right\|_{L^2(\pr)}^2 \\
    & \lesssim  \max_{0\le s\le 3} \frac{n^{(3-s)/2}}{n^4} \log^{(3+s)/2}|\mathcal{A}_n|   \left\| \max_{(a_1,a_2)\in\mathcal{A}_n^2} \frac{1}{\sigma_{a_1}^2\sigma_{a_2}^2}P^{3-s}|\wt{\pi_2\bar\psi_{a_1}'\star_1^0 \pi_2\bar\psi_{a_2}'}|^2\right\|_{L^\infty(P^s)}^{1/2} \\
    & \quad +  \frac{1}{n} \left\|\max_{a_1\in\mathcal{A}_n} \frac{1}{n^2}\frac{1}{\sigma_{a_1}}  |J_2(\pi_2 \bar\psi_{a_1}') |  \right\|_{L^2(\pr)}\left\|\max_{a_2\in\mathcal{A}_n} \frac{1}{n} \frac{1}{\sigma_{a_2}}|J_1(\pi_1\bar\psi_{a_2}')  |\right\|_{L^2(\pr)} + \frac{1}{n^2} \left\|\max_{a_1\in\mathcal{A}_n}  \frac{1}{n} \frac{1}{\sigma_{a_1}} |J_1(\pi_1\bar\psi_{a_1}')|\right\|_{L^2(\pr)}^2.
\end{align*}
From the first line of the proof of Lemma 5 in \cite{Ro95}, it holds that
{\small
\begin{align*}
    & |\wt{\pi_2\bar\psi_{a_1}'\star_1^0 \pi_2\bar\psi_{a_2}'}|^2(Z_i,Z_l,Z_m) \\
    & \asymp |\pi_2\bar\psi_{a_1}'|^2(Z_i,Z_l)|\pi_2\bar\psi_{a_2}'|^2(Z_i,Z_m) + |\pi_2\bar\psi_{a_1}'|^2(Z_l,Z_m)|\pi_2\bar\psi_{a_2}'|^2(Z_l,Z_i) + |\pi_2\bar\psi_{a_1}'|^2(Z_m,Z_i)|\pi_2\bar\psi_{a_2}'|^2(Z_m,Z_l) \\
    & \lesssim
    (Y_{ij_{a_1}}^2 + Y_{lj_{a_1}}^2)(Y_{ij_{a_2}}^2 + Y_{mj_{a_2}}^2)\{(Y_{ij_{a_1}} - Y_{lj_{a_1}})\partial_{k_{a_1}}K_{il,h_{n,a_1}'}\}^2\{(Y_{ij_{a_2}} - Y_{mj_{a_2}})\partial_{k_{a_2}}K_{im,h_{n,a_2}'}\}^2 h_{n,a_1}'^{-2(d+1)}h_{n,a_2}'^{-2(d+1)} \\
    & ~~ + (Y_{lj_{a_1}}^2 + Y_{mj_{a_1}}^2)(Y_{lj_{a_2}}^2 + Y_{ij_{a_2}}^2)\{(Y_{lj_{a_1}} - Y_{mj_{a_1}})\partial_{k_{a_1}}K_{lm,h_{n,a_1}'}\}^2\{(Y_{lj_{a_2}} - Y_{ij_{a_2}})\partial_{k_{a_2}}K_{li,h_{n,a_2}'}\}^2 h_{n,a_1}'^{-2(d+1)}h_{n,a_2}'^{-2(d+1)} \\
    & ~~ + (Y_{mj_{a_1}}^2 + Y_{ij_{a_1}}^2)(Y_{mj_{a_2}}^2 + Y_{lj_{a_2}}^2)\{(Y_{mj_{a_1}} - Y_{ij_{a_1}})\partial_{k_{a_1}}K_{mi,h_{n,a_1}'}\}^2\{(Y_{mj_{a_2}} - Y_{lj_{a_2}})\partial_{k_{a_2}}K_{ml,h_{n,a_2}'}\}^2 h_{n,a_1}'^{-2(d+1)}h_{n,a_2}'^{-2(d+1)} 
\end{align*}}
for each $(a_1,a_2)\in\mathcal{A}_n^2$.
Given this result, some straightforward calculations give
\begin{align*}
    & \left\| \max_{(a_1,a_2)\in\mathcal{A}_n^2} \frac{1}{\sigma^2_{a_1}\sigma^2_{a_2}}|\wt{\pi_2\bar\psi_{a_1}'\star_1^0 \pi_2\bar\psi_{a_2}'}|^2\right\|_{L^\infty(P^3)}^{1/2} \lesssim n^2\underline{h}_n'^{-d},\\
    & \left\| \max_{(a_1,a_2)\in\mathcal{A}_n^2} \frac{1}{\sigma^2_{a_1}\sigma^2_{a_2}}P|\wt{\pi_2\bar\psi_{a_1}'\star_1^0 \pi_2\bar\psi_{a_2}'}|^2\right\|_{L^\infty(P^2)}^{1/2} \lesssim  n^2\underline{h}_n'^{-d/2}, \\
    & \left\| \max_{(a_1,a_2)\in\mathcal{A}_n^2}\frac{1}{\sigma^2_{a_1}\sigma^2_{a_2}} P^2 |\wt{\pi_2\bar\psi_{a_1}'\star_1^0 \pi_2\bar\psi_{a_2}'}|^2\right\|_{L^\infty(P)}^{1/2} \lesssim  n^2\underline{h}_n'^{-d/2},\\
    & \left(\max_{(a_1,a_2)\in\mathcal{A}_n^2} \frac{1}{\sigma^2_{a_1}\sigma^2_{a_2}}P^3 |\wt{\pi_2\bar\psi_{a_1}'\star_1^0 \pi_2\bar\psi_{a_2}'}|^2\right)^{1/2} \lesssim  n^2
\end{align*}
In conjunction with \eqref{eq:DWAD-max-hajek} and \eqref{eq:DWAD-max-second}
, we can see that, under Assumption \ref{as:DWAD-bandwdith}, it holds that
\begin{align}
    & \left\|\max_{(a_1,a_2)\in\mathcal{A}_n^2} \frac{1}{\sigma_{a_1}\sigma_{a_2}}|(I.II)_{a_1,a_2}| \right\|_{L^1(\pr)}  \nonumber\\
    & \lesssim \frac{\log^{3/2}|\mathcal{A}_n|}{n^{1/2}} \vee \frac{\log^{2}|\mathcal{A}_n|}{n\underline{h}_n'^{d/2}} \vee \frac{\log^{5/2}|\mathcal{A}_n|}{n^{3/2}\underline{h}_n'^{d/2}} \vee \frac{\log^{3}|\mathcal{A}_n|}{n^{2}\underline{h}_n'^{d}} + \frac{\log|\mathcal{A}_n|}{n^2} = o\left( \frac{1}{ \log^2|\mathcal{A}_n|} \right). \label{eq:convergence-I.II}
\end{align}

\subparagraph{Convergence of $II$:}
Since $P\bar\psi_{jk,h_n'}(Z_i) - P^2\bar\psi_{jk,h_n'} = \pi_1\bar\psi_{jk,h_n'}(Z_i)$,
\begin{align*}
    II_{a_1,a_2} 
    &\coloneqq \frac{1}{n^2} \sum_{i=1}^n \left[ P \bar\psi_{a_1}'(Z_i) - P^2 \bar\psi_{a_1}'\right] \left[ P \bar\psi_{a_2}'(Z_i) - P^2 \bar\psi_{a_2}'\right] = \frac{1}{n^2} \sum_{i=1}^n (\pi_1\bar\psi_{a_1} \star_1^0 \pi_1\bar\psi_{a_2})(Z_i).
\end{align*}
Note that, from \eqref{eq:eval-DWAD-hajek} and $n^{-1}\log^7(n|\mathcal{A}_n|) \to 0$, 
\begin{align*}
    \frac{1}{\sigma_{a_1}\sigma_{a_2}}\E[II_{a_1,a_2}] = \frac{1}{n}\frac{1}{\sigma_{a_1}\sigma_{a_2}} \Cov(P\bar\psi_{a_1}', P\bar\psi_{a_2}') =  \frac{n}{\sigma_{a_1}\sigma_{a_2}} \Cov(P\psi_{a_1}', P\psi_{a_2}') + o\left( \frac{1}{ \log^2|\mathcal{A}_n|} \right),
\end{align*}
for all $(a_1,a_2)\in\mathcal{A}_n^2$.
Therefore, it is sufficient to show that $II_{a_1,a_2} - \E[II_{a_1,a_2}]$ are negligible uniformly in $(a_1,a_2)\in\mathcal{A}_n^2$.
Observe that
\begin{align*}
    & II_{a_1,a_2} - \E[II_{a_1,a_2}]= \frac{1}{n^2}\sum_{i=1}^n \{(\pi_1\bar\psi_{a_1}' \star_1^0\pi_1\bar\psi_{a_2}')(Z_i) - \E[(\pi_1\bar\psi_{a_1}' \star_1^0\pi_1\bar\psi_{a_2}')(Z_i)  ]\} 
\end{align*}
Then, \cref{u-nemirovski}, \eqref{eq:eval-DWAD-hajek}, \eqref{eq:asymp-variance} and $n^{-1} \log^3|\mathcal{A}_n| \to 0$  give
\begin{align}
    &\left\| \max_{(a_1,a_2)\in\mathcal{A}_n^2} \frac{1}{\sigma_{a_1}\sigma_{a_2}} |II_{a_1,a_2} - \E[II_{a_1,a_2}]|\right\|_{L^1(\pr)} \nonumber\\
    &\lesssim \frac{1}{n^2} \log^{1/2}|\mathcal{A}_n| \left\| \max_{(a_1,a_2)\in\mathcal{A}_n^2} \frac{1}{\sigma_{a_1}\sigma_{a_2}}\sqrt{\sum_{i=1}^n \{(\pi_1\bar\psi_{a_1}' \star_1^0\pi_1\bar\psi_{a_2}')(Z_i) - \E[(\pi_1\bar\psi_{a_1}' \star_1^0\pi_1\bar\psi_{a_2}')(Z_i)  ]\}^2} \right\|_{L^1(\pr)} \nonumber\\
    & \le \frac{1}{n^{3/2}}\log^{1/2}|\mathcal{A}_n| \left\| \max_{(a_1,a_2)\in\mathcal{A}_n^2} \frac{1}{\sigma_{a_1}\sigma_{a_2}}|(\pi_1\bar\psi_{a_1}' \star_1^0\pi_1\bar\psi_{a_2}')(Z_i) - \E[(\pi_1\bar\psi_{a_1}' \star_1^0\pi_1\bar\psi_{a_2}')(Z_i)  ]| \right\|_{L^\infty(\pr)} \nonumber\\
    & = O\left(  \frac{\log^{1/2}|\mathcal{A}_n|}{n^{1/2}} \right) = o\left( \frac{1}{\log^2|\mathcal{A}_n|}\right). \label{eq:convergence-II}
\end{align}

From \eqref{eq:convergence-III}, \eqref{eq:convergence-IV}, \eqref{eq:convergence-V}, \eqref{eq:convergence-VI}, \eqref{eq:convergence-I.I}, \eqref{eq:convergence-I.II} and \eqref{eq:convergence-II}, we can see that $(\Delta^\sharp I.I) =o_p\left(\log^{-2}|\mathcal{A}_n|\right)$. 
In conjunction with $(\Delta^\sharp I.II) =o_p\left(\log^{-2}|\mathcal{A}_n|\right)$, we have shown $(\Delta^\sharp.I) = o_p(\log^{-2}|\mathcal{A}_n|)$.

Next, we show $(\Delta^\sharp.II) = o_p(\log^{-2}|\mathcal{A}_n|)$. 
Observe that
\begin{align*}
    (\Delta^\sharp.II) & = \max_{(a_1,a_2)\in\mathcal{A}_n^2} \left| \left( \frac{1}{\hat{\sigma}_{a_1} \hat{\sigma}_{a_2}} - \frac{1}{\sigma_{a_1} \sigma_{a_2}} \right) \left( \frac{1}{n^2}\sum_{i=1}^n \left[ \sum_{l\neq i}^n \psi_{a_1}'(Z_i,Z_l) - \hat{\theta}_{a_1}' \right]\left[ \sum_{l\neq i}^n \psi_{a_2}'(Z_i,Z_l) - \hat{\theta}_{a_2}' \right] \right) \right| \\ 
    & \le  \max_{(a_1,a_2)\in\mathcal{A}_n^2} \left| \left( \frac{1}{\hat{\sigma}_{a_1} \hat{\sigma}_{a_2}} - \frac{1}{\sigma_{a_1} \sigma_{a_2}} \right) \sigma_{a_1} \sigma_{a_2} \Cov^*(J^\sharp)_{a_1,a_2}  \right| \\
    &\le  \max_{(a_1,a_2)\in\mathcal{A}_n^2} \left| \frac{\sigma_{a_1}(\sigma_{a_2}^2 - \hat{\sigma}_{a_2}^2)}{\hat{\sigma}_{a_1}\hat{\sigma}_{a_2}(\sigma_{a_2} + \hat{\sigma}_{a_2})} + \frac{(\sigma_{a_1}^2 - \hat{\sigma}_{a_1}^2)}{{\hat{\sigma}_{a_1}(\sigma_{a_1} + \hat{\sigma}_{a_1})}} \right| \max_{(a_1,a_2)\in\mathcal{A}_n^2} \left|  \Cov^*(J^\sharp)_{a_1,a_2}  \right| .
\end{align*}
Since $\hat{\sigma}_{a_1}^2$ has the same form as $\sigma_{a_1}^2\Cov^*(J^\sharp)_{a_1,a_1}$ and $\sigma_{a_1}^2$ does as $\sigma_{a_1}^2 \Cov(J)_{a_1,a_1}$, from the evaluation on $(\Delta^\sharp.I.I)$, we can immediately see that 
\[
\max_{a_1\in\mathcal{A}_n}|\hat{\sigma}_{a_1}^2 - \sigma_{a_1}^2| = \max_{a_1\in\mathcal{A}_n}|\sigma_{a_1}^2\Cov^*(J^\sharp)_{a_1,a_1} - \sigma_{a_1}^2 \Cov(J)_{a_1,a_1}| = o_p\left\{ \frac{1}{\min\{n, n^2\ul{h}_n^{d+2}\}\log^2|\mathcal{A}_n|} \right\}.
\]
This also implies $\hat{\sigma}_{a_1}^2 = \sigma_{a_1}^2 +  o_p\left( \min\{n, n^2\ul{h}_n^{d+2}\}^{-1}\log^{-2}|\mathcal{A}_n| \right)$ uniformly on $\mathcal{A}_n$. 
Also, from \eqref{eq:DWAD_COV_J}, \eqref{eq:DWAD_COV_J1}, \eqref{eq:DWAD_COV_J2} and  $\max_{(a_1,a_2)\in\mathcal{A}_n^2} \sigma_{a_1}^{-1}\sigma_{a_2}^{-1} = O(\min\{n, n^2\ul{h}_n^{d+2}\})$, we have $\max_{(a_1,a_2)\in\mathcal{A}_n^2}\left|  \Cov^*(J^\sharp)_{a_1,a_2}  \right| = O_p(1)$.
Therefore
\begin{align*}
    (\Delta^\sharp.II) &\le \max_{(a_1,a_2)\in\mathcal{A}_n^2} \left| \frac{(\sigma_{a_2}^2 - \hat{\sigma}_{a_2}^2)}{2\sigma_{a_2}^2} + \frac{(\sigma_{a_1}^2 - \hat{\sigma}_{a_1}^2)}{{2\sigma}_{a_1}^2} \right| +   o_p\left( \frac{1}{\min\{n, n^2\ul{h}_n^{d+2}\}\log^2|\mathcal{A}_n|} \right)  = o_p\left( \frac{1}{\log^2|\mathcal{A}_n|}\right).
\end{align*}

\paragraph{Convergence of ($\rho^\sharp.III$):} 
Observe that
\begin{align*}
    \max_{(j,k,h_n)\in\mathcal{A}_n} \left| \hat{J}_{jk,h_n}  -   J_{jk,h_n} \right| &\le \max_{(j,k,h_n)\in\mathcal{A}_n} |J_{jk,h_n}| \max_{(j,k,h_n)\in\mathcal{A}_n} \frac{|\hat\sigma_{jk,h_n}^2 - \sigma^2_{jk,h_n}|}{\hat\sigma_{jk,h_n}(\hat\sigma_{jk,h_n} + \sigma_{jk,h_n})} .
\end{align*}
From $(\rho^\sharp.I) \to 0$ and $\max_{(j,k,h_n)\in\mathcal{A}_n} |G_{jk,h_n}| = O_p(\sqrt{\log |\mathcal{A}_n|})$ (cf. Lemma 2.2.10 in \citealp{van1996weak}), we have $\max_{(j,k,h_n)\in\mathcal{A}_n} |J_{jk,h_n}| = O_p(\sqrt{\log |\mathcal{A}_n|})$.
In conjunction with the evaluation of $\max_{(j,k,h_n))\in\mathcal{A}_n}|\hat{\sigma}_{jk,h_n}^2 - \sigma_{jk,h_n}^2|/\sigma_{jk,h_n}^2 = o_p(\log^{-2}|\mathcal{A}_n|)$,
\begin{align*}
    & \max_{(j,k,h_n)\in\mathcal{A}_n} \left| \hat{J}_{jk,h_n}  -   J_{jk,h_n} \right|  \ll \log^{1/2}|\mathcal{A}_n| \cdot \left( \frac{1}{\log^2|\mathcal{A}_n|}  \right)  =  \frac{1}{\log^{3/2}|\mathcal{A}_n|} .
\end{align*}
Therefore, under Assumption \ref{as:DWAD-bandwdith}, $\max_{(j,k,h_n)\in\mathcal{A}_n} \left| \hat{J}_{jk,h_n}  -   J_{jk,h_n} \right|  = o_p(1/\sqrt{\log|\mathcal{A}_n|})$ so Lemma 1 in \cite{chernozhukov2023high} completes the proof. 

\section{Proofs of auxiliary results}\label{sec:proof-aux}

\subsection{Proof of Theorem \ref{thm:gexch}}\label{sec:gexch}

Without loss of generality, we may assume that $(Y,Y')$ and $Z$ are independent. 
First, we see that it suffices to prove \eqref{eq:gech} with $\mcl R_p$ replaced by $\mcl R_p^0:=\{\prod_{j=1}^p(-\infty, y_j]:y_1,\dots,y_p\in\mathbb R\}$. 
In fact, define functions $\bar{\mathsf W}:E\to\mathbb R^{2p}$ and $\bar{\mathsf G}:E^2\to\mathbb R^{2p}$ as $\bar{\mathsf W}(y)=(\mathsf W(y)^\top,-\mathsf W(y)^\top)^\top$ and $\bar{\mathsf G}(y,y')=(\mathsf G(y,y')^\top,-\mathsf G(y,y')^\top)^\top$ for $y,y'\in E$. 
For $\bar W:=\bar{\mathsf W}(Y)$ and $\bar G:=\bar{\mathsf G}(Y,Y')$, we evidently have $\E[\bar G\mid Y]=-(\bar W+\bar R)$ with $\bar R:=(R^\top,-R^\top)^\top$. 
We also have
\[
\sup_{A\in\mathcal R_p}\left|\pr(W\in A)-\pr(Z\in A)\right|
=\sup_{A\in\mathcal R_{2p}^0}\left|\pr(\bar W\in A)-\pr(\bar Z\in A)\right|,
\]
where $\bar Z:=(Z^\top,-Z^\top)^\top$. 
Moreover, for any $\eps>0$, we have with $\eps':=\eps\log(2p)/\log p$,
\ba{
\|\bar R+\E[G1_{\{\|\bar D\|_\infty>\eps'/\log(2p)\}}\mid Y]\|_\infty&=\|R^\eps\|_\infty,\\
\norm{\frac{1}{2}\E[GD^\top1_{\{\|\bar D\|_\infty\leq\eps'/\log(2p)\}}\mid Y]-\bar\Sigma}_\infty&=\|V^\eps\|_\infty,\\
\max_{j,k,l,m\in[2p]}\E[|\bar G_j\bar D_k\bar D_l\bar D_m|1_{\{\|\bar D\|_\infty\leq\eps'/\log(2p)\}}\mid Y]&=\Gamma^\eps,
}
where $\bar D:=\bar{\mathsf W}(Y')-\bar W$ and $\bar\Sigma:=\Cov[\bar Z]$. 
Therefore, noting that $\eps\leq\eps'\leq2\eps$, we can derive the claim asserted from the corresponding one with $\mcl R_p$ and $p$ replaced by $\mcl R_{2p}^0$ and $2p$, respectively. 

In the remaining proof, we proceed in five steps. \vspace{-5mm}

\paragraph{Step 1.} Fix a non-increasing $C^4$ function $g_0\colon\mathbb R\to\mathbb R$ such that (i) $g_0(t)\geq 0$ for all $t\in\mathbb R$, (ii) $g_0(t) = 0$ for all $t\geq 1$, and (iii) $g_0(t) = 1$ for all $t\leq 0$. For this function, there exists a constant $C_g>0$ such that
$$
\sup_{t\in\mathbb R}\Big(|g_0^{(1)}(t)|\vee|g_0^{(2)}(t)|\vee|g_0^{(3)}(t)|\vee|g_0^{(4)}(t)|\Big) \leq C_g.
$$
Since the function $g_0$ is fixed and can be chosen to be universal, we can also take the constant $C_g$ to be universal. 
Next, define a function $F_\beta:\mathbb R^p\to\mathbb R$ as
\[
F_\beta(w)=\beta^{-1}\log\bra{\sum_{j=1}^pe^{\beta w_j}},\quad w\in\mathbb R^p.
\]
By Eq.(8) in \cite{CCK13},
\begin{equation}\label{max-smooth}
\max_{j\in[p]}w_j  \leq F_\beta(w)  \leq  \max_{j\in[p]}w_j + \eps,\ \text{for all }w\in\mathbb R^p.
\end{equation}
Also, for all $y\in\mathbb R^p$, define a function $m^y\colon\mathbb R^p\to\mathbb R$ as
$
m^y(w) = g_0(\eps^{-1}F_\beta(w - y))$, $w\in\mathbb R^p.
$
Further, set
$
\mathcal I^{y} :=m^y(W) - m^y(Z).
$
By Step 2 of the proof of \cite[Lemma 5.1]{CCK17}, we have
\[
\sup_{A\in\mathcal R_p^0}\Big|\pr(W\in A) - \pr(Z\in A)\Big| \lesssim \frac{\eps}{\ul\sigma}\sqrt{\log p} + \sup_{y\in\mathbb R^p}|\E[\mathcal I^y]|.
\]
Therefore, we complete the proof once we show
\ben{\label{iy-bound}
\sup_{y\in\mathbb R^p}|\E[\mathcal I^y]|
\lesssim \frac{1}{\ul\sigma}\bra{\E\sbra{\|R^\eps\|_\infty}\sqrt{\log p}
+\eps^{-1}\E\sbra{\|V^\eps\|_\infty}(\log p)^{3/2}
+\eps^{-3}\E\sbra{\Gamma^\eps}(\log p)^{7/2}}.
}

\paragraph{Step 2.} 
Define a function $f:\mathbb{R}^p\to\mathbb{R}$ as
\[
f(w)=\int_0^1\frac{1}{2t}\E[m^y(\sqrt{t}w+\sqrt{1-t}Z)-m^y(Z)]dt,\quad w\in\mathbb{R}^p.
\]
$f$ is a solution to the following Stein equation (cf.~\citealp[Lemma 1]{meckes2009stein}):
\begin{equation*}
m^y(w)-\E[m^y(Z)]=w\cdot\nabla f(w)-\langle\Sigma,\nabla^2f(w)\rangle,\quad w\in\mathbb R^p.
\end{equation*}
Hence we have 
\ben{\label{eq:stein}
\E[\mcl I^y]=\E[W\cdot f(W)-\langle\Sigma,\nabla^2f(W)\rangle].
}
We expand the right-hand side of this identity by a standard argument in Stein's method. 
Since $(Y,Y')$ is an exchangeable pair and $\mathsf G(Y',Y)=-G$, we have
\ba{
\E[G\cdot\{\nabla f(W)+\nabla f(W')\}1_{\{\|D\|_\infty\leq\beta^{-1}\}}]
&=-\E[G\cdot\{\nabla f(W')+\nabla f(W)\}1_{\{\|D\|_\infty\leq\beta^{-1}\}}].
}
Hence
\ben{\label{exch-eq1-m}
\E[G\cdot\{\nabla f(W)+\nabla f(W')\}1_{\{\|D\|_\infty\leq\beta^{-1}\}}]=0.
}
Meanwhile, by the fundamental theorem of calculus, 
\ba{
&\E[G\cdot\{\nabla f(W')-\nabla f(W)\}1_{\{\|D\|_\infty\leq\beta^{-1}\}}]
=\sum_{j=1}^p\E[G_j\{\partial_jf(W')-\partial_jf(W)\}1_{\{\|D\|_\infty\leq\beta^{-1}\}}]\\
&=\sum_{j,k=1}^p\E[G_jD_k\partial_{jk}f(W)1_{\{\|D\|_\infty\leq\beta^{-1}\}}]
+\sum_{j,k,l=1}^p\E[(1-U)G_jD_kD_l\partial_{jkl}f(W+UD)1_{\{\|D\|_\infty\leq\beta^{-1}\}}]\\
&=\E[\langle \E[GD^\top1_{\{\|D\|_\infty\leq\beta^{-1}\}}\mid Y],\nabla^2 f(W)\rangle]
+\Delta,
}
where $U$ is a uniform random variable on $[0,1]$ independent of everything else and
\ba{
\Delta
&:=\sum_{j,k,l=1}^p\E[(1-U)G_jD_kD_l\partial_{jkl}f(W+UD)1_{\{\|D\|_\infty\leq\beta^{-1}\}}].
}
Hence, we can rewrite the left-hand side of \eqref{exch-eq1-m} as
\besn{\label{exch-eq2-m}
&\E[G\cdot\{\nabla f(W)+\nabla f(W')\}1_{\{\|D\|_\infty\leq\beta^{-1}\}}]\\
&=2\E[G\cdot \nabla f(W)1_{\{\|D\|\leq\beta^{-1}\}}]+\E[G\cdot\{\nabla f(W')-\nabla f(W)\}1_{\{\|D\|_\infty\leq\beta^{-1}\}}]\\
&=2\E[G\cdot \nabla f(W)]-2\E[G\cdot \nabla f(W)1_{\{\|D\|_\infty>\beta^{-1}\}}]\\
&\quad+\E[\langle \E[GD^\top1_{\{\|D\|_\infty\leq\beta^{-1}\}}\mid Y],\nabla^2 f(W)\rangle]
+\Delta.
}
Since $\E[G\cdot \nabla f(W)]=-\E[(W+R)\cdot \nabla f(W)]$ by \eqref{eq:nlr}, 
we deduce from \eqref{exch-eq1-m} and \eqref{exch-eq2-m}
\ba{
\E[W\cdot \nabla f(W)]
&=-\E[R^\eps\cdot \nabla f(W)]
+\frac{1}{2}\left(\E[\langle \E[GD^\top1_{\{\|D\|_\infty\leq\beta^{-1}\}}\mid Y],\nabla^2 f(W)\rangle]
+\Delta\right).
}
This and \eqref{eq:stein} give
\ba{
\E[\mcl I^y]
&=-\E[R^\eps\cdot \nabla f(W)]+\E[\langle V^\eps,\nabla^2 f(W)\rangle]
+\frac{1}{2}\Delta.
}
Therefore, \eqref{iy-bound} follows once we prove the following inequalities:
\ban{
|\E[R^\eps\cdot \nabla f(W)]|
&\lesssim\frac{\E\sbra{\|R^\eps\|_\infty}\sqrt{\log p}}{\ul\sigma},\label{eq:gexch-I}\\
|\E[\langle V^\eps,\nabla^2 f(W)\rangle]|
&\lesssim\frac{\eps^{-1}\E\sbra{\|V^\eps\|_\infty}(\log p)^{3/2}}{\ul\sigma},\label{eq:gexch-II}\\
|\Delta|&\lesssim\frac{\eps^{-3}\E\sbra{\Gamma^\eps}(\log p)^{7/2}}{\ul\sigma}.\label{eq:gexch-III}
}

\paragraph{Step 3.} 
This step proves \eqref{eq:gexch-I}. We rewrite $\E[R^\eps\cdot \nabla f(W)]$ as
\ba{
\E[R^\eps\cdot \nabla f(W)]
=\int_0^1\frac{1}{2\sqrt t}\E[R^\eps\cdot\nabla m^y(\sqrt{t}W+\sqrt{1-t}Z)]dt.
}
Since $g_0'(x)=0$ if $x\notin[0,1]$, we have $0\leq F_\beta(w-y)\leq \eps$ if $w\in\mathbb R^p$ satisfies $\nabla m^y(w)\neq0$. 
Since $-\eps\leq\max_{j\in[p]}(w_j-y_j)\leq\eps$ whenever $0\leq F_\beta(w-y)\leq \eps$ by \eqref{max-smooth}, we obtain
\ben{\label{ac-arg}
\E[R^\eps\cdot \nabla f(W)]
=\int_0^1\frac{1}{2\sqrt t}\E[R^\eps\cdot\nabla m^y(\sqrt{t}W+\sqrt{1-t}Z)1_{A(t)}]dt,
}
where $A(t):=\{-\eps\leq \max_{j\in[p]}(\sqrt{t}W_j+\sqrt{1-t}Z_j-y_j)\leq \eps\}$. 
Meanwhile, by Lemma A.2 in \cite{CCK13} and the chain rule,
$
\sum_{j=1}^p|\partial_j m^y(w)|\lesssim \eps^{-1}$ for all $w\in\mathbb R^p.$ 
Hence,
\ba{
|\E[R^\eps\cdot \nabla f(W)]|
\lesssim\eps^{-1}\int_0^1\frac{1}{\sqrt t}\E[\|R^\eps\|_\infty1_{A(t)}]dt.
}
Noting that $Y$ and $Z$ are independent and $W=\mathsf{W}(Y)$, we have for every $0<t<1$
\besn{\label{nazarov-arg}
\E[\|R^\eps\|_\infty1_{A(t)}]
&=\E\sbra{\|R^\eps\|_\infty\pr\bra{-\eps\leq \max_{j\in[p]}(\sqrt{t}W_j+\sqrt{1-t}Z_j-y_j)\leq \eps\mid Y}}\\
&\leq\E\sbra{\|R^\eps\|_\infty}\sup_{z\in\mathbb R^p}\pr\bra{-\eps\leq \max_{j\in[p]}(\sqrt{1-t}Z_j-z_j)\leq \eps}\\
&\lesssim\frac{\eps\E\sbra{\|R^\eps\|_\infty}\sqrt{\log p}}{\ul\sigma\sqrt{1-t}},
}
where the last inequality follows from Nazarov's inequality \cite[Lemma A.1]{CCK17}. 
Hence we conclude
\ba{
|\E[R^\eps\cdot \nabla f(W)]|
\lesssim\frac{\E\sbra{\|R^\eps\|_\infty}\sqrt{\log p}}{\ul\sigma}\int_0^1\frac{1}{\sqrt{t(1-t)}}dt
\lesssim\frac{\E\sbra{\|R^\eps\|_\infty}\sqrt{\log p}}{\ul\sigma}.
}

\paragraph{Step 4.} 
This step proves \eqref{eq:gexch-II}. Similarly to the derivation of \eqref{ac-arg}, we deduce
\[
\E[\langle V^\eps,\nabla^2 f(W)\rangle]=\int_0^1\frac{1}{2}\E[\langle V^\eps,\nabla^2 m^y(\sqrt{t}W+\sqrt{1-t}Z)1_{A(t)}]dt.
\]
Also, by Eqs.(C.4) and (C.7) in \cite{CCKK22}, we have 
$\sum_{j,k=1}^p|\partial_{jk} m^y(w)|\lesssim \eps^{-2}\log p$ for all $w\in\mathbb R^p.$ 
Hence
\[
|\E[\langle V^\eps,\nabla^2 f(W)\rangle]|\lesssim\eps^{-2}(\log p)\int_0^1\E[\| V^\eps\|_\infty1_{A(t)}]dt.
\]
By a similar argument to the proof of \eqref{nazarov-arg}, we obtain
\[
\E[\| V^\eps\|_\infty1_{A(t)}]\lesssim\frac{\eps\E\sbra{\|V^\eps\|_\infty}\sqrt{\log p}}{\ul\sigma\sqrt{1-t}}.
\]
Hence we conclude
\[
|\E[\langle V^\eps,\nabla^2 f(W)\rangle]|\lesssim\frac{\eps^{-1}\E\sbra{\|V^\eps\|_\infty}(\log p)^{3/2}}{\ul\sigma}\int_0^1\frac{1}{\sqrt{1-t}}dt
\lesssim\frac{\eps^{-1}\E\sbra{\|V^\eps\|_\infty}(\log p)^{3/2}}{\ul\sigma}.
\]

\paragraph{Step 5.} 
In this step, we prove \eqref{eq:gexch-III} and complete the proof. 
We begin by further expanding $\Delta$ using a symmetry trick introduced in \cite{fang2021high} (cf.~Eq.(2.16) ibidem); using the fact that $(Y,Y')$ is an exchangeable pair and $\mathsf G(Y',Y)=-G$ again, we rewrite $\Delta$ as
\besn{\label{exch-key}
\Delta&=\sum_{j,k,l=1}^p\E[(1-U)(-G_j)(-D_k)(-D_l)\partial_{jkl}f(W'-UD)1_{\{\|D\|_\infty\leq\beta^{-1}\}}]\\
&=-\sum_{j,k,l=1}^p\E[(1-U)G_jD_kD_l\partial_{jkl}f(W+(1-U)D)1_{\{\|D\|_\infty\leq\beta^{-1}\}}].
}
Therefore, 
\ba{
\Delta&=\frac{1}{2}\sum_{j,k,l=1}^p\E[(1-U)G_jD_kD_l\{\partial_{jkl}f(W+UD)-\partial_{jkl}f(W+(1-U)D)\}1_{\{\|D\|_\infty\leq\beta^{-1}\}}]\\
&=-\frac{1}{2}\sum_{j,k,l,r=1}^p\E[(1-U)G_jD_kD_lD_r\partial_{jklr}f(W+UD+U'\tilde UD)1_{\{\|D\|_\infty\leq\beta^{-1}\}}],
}
where $\tilde U:=1-2U$ and $U'$ is a uniform random variable on $[0,1]$ independent of everything else. 
Using the definition of $f$, we obtain
\ba{
\Delta
&=-\frac{1}{4}\sum_{j,k,l,r=1}^p\int_0^1t\E[(1-U)G_jD_kD_lD_r\partial_{jklr}m^y(W(t)+\sqrt t(U+U'\tilde U)D)1_{\{\|D\|_\infty\leq\beta^{-1}\}}]dt,
}
where $W(t):=\sqrt tW+\sqrt{1-t}Z$. 
Now, observe that $|U+U'\tilde U|\leq U\vee(U+\tilde U)=U\vee(1-U)\leq1$. 
Also, for any $j,k,l,r\in[p]$, a similar argument to the derivation of \eqref{ac-arg} shows
\[
\partial_{jklr}m^y(W(t)+\sqrt t(U+U'\tilde U)D)\neq 0\\
\Rightarrow-\eps\leq \max_{i\in[p]}(W(t)_i+\sqrt t(U+U'\tilde U)D_i-y_i)\leq \eps.
\]
Hence, on the event $\{\|D\|_\infty\leq\beta^{-1}\}$,
\[
\partial_{jklr}m^y(W(t)+\sqrt t(U+U'\tilde U)D)\neq 0\\
\Rightarrow-\eps-\beta^{-1}\leq \max_{i\in[p]}(W(t)_i-y_i)\leq \eps+\beta^{-1}.
\]
Therefore, with $A'(t):=\{-\eps-\beta^{-1}\leq \max_{j\in[p]}(W(t)_j-y_j)\leq \eps+\beta^{-1}\}$, we have 
\ba{
\Delta
&=-\frac{1}{4}\sum_{j,k,l,r=1}^p\int_0^1t\E[(1-U)G_jD_kD_lD_r\partial_{jklr}m^y(W(t)+\sqrt t(U+U'\tilde U)D)1_{\{\|D\|_\infty\leq\beta^{-1}\}\cap A'(t)}]dt.
}
By Eqs.(C.5), (C.6) and (C.8) in \cite{CCKK22}, there exist functions $U^y_{jklr}:\mathbb{R}^p\to\mathbb R$ ($j,k,l,r\in[p]$) such that for any $w,w'\in\mathbb R^p$ with $\|w'\|_\infty\leq\beta^{-1}$, 
\begin{equation}\label{m4:property1}
|\partial_{jklr}m^y(w)|\leq U^y_{jklr}(w),\qquad
U^y_{jklr}(w+w')\lesssim U^y_{jklr}(w)
\end{equation}
for all $j,k,l,r\in[p]$ and
\begin{equation}\label{m4:property2}
\sum_{j,k,l,r=1}^pU^y_{jklr}(w)\lesssim\eps^{-4}\log^3p.
\end{equation}
By \eqref{m4:property1}, 
\ba{
|\Delta|
&\lesssim\sum_{j,k,l,r=1}^p\int_0^1\E[|G_jD_kD_lD_r|U^y_{jklr}(W(t))1_{\{\|D\|_\infty\leq\beta^{-1}\}\cap A'(t)}]dt\\
&=\sum_{j,k,l,r=1}^p\int_0^1\E[\E[|G_jD_kD_lD_r|1_{\{\|D\|_\infty\leq\beta^{-1}\}}\mid Y]U^y_{jklr}(W(t))1_{A'(t)}]dt,
}
where the second line follows from the fact that both $W(t)$ and $A'(t)$ are $\sigma(Y,Z)$-measurable and $Z$ is independent of $(Y,Y')$. 
Using \eqref{m4:property2}, we obtain
\ben{\label{delta-aim}
|\Delta|
\lesssim\eps^{-4}(\log p)^3\int_0^1\E\sbra{\Gamma^\eps1_{A'(t)}}dt.
}
Similarly to the derivation of \eqref{nazarov-arg}, we deduce
\ba{
\E\sbra{\Gamma^\eps1_{A'(t)}}
\lesssim\frac{\eps\E\sbra{\Gamma^\eps}\sqrt{\log p}}{\ul\sigma\sqrt{1-t}}.
}
Combining this with \eqref{delta-aim} gives \eqref{eq:gexch-III}. 
\qed

\subsection{Proof of Corollary \ref{coro:exch}}

If $G=0$ or $D=0$, then $\sqrt{\E[\|V\|_\infty]}=\sqrt{\|\Sigma\|_\infty}\geq\ul\sigma$, so the claim trivially holds for any $C'\geq1$. 
Hence, we may assume $G\neq0$ and $D\neq0$ without loss of generality. In particular, we have $\E[\|G\|_\infty\|D\|_\infty^3]>0$ in this case. 

For every $\eps>0$, observe that
\ba{
\E[\|R^\eps\|_\infty]
&\leq\E[\|R\|_\infty]
+\beta^3\E[\|G\|_\infty\|D\|_\infty^3],\\
\E[\|V^\eps-V\|_\infty]
&\leq\frac{1}{2}\E[\|G\|_\infty\|D\|_\infty1_{\{\|D\|_\infty>\beta^{-1}\}}]
\leq\frac{\beta^{2}}{2}\E[\|G\|_\infty\|D\|_\infty^3],\\
\E[\Gamma^\eps]
&\leq\E[\|G\|_\infty\|D\|_\infty^3].
}
Inserting these bounds into \eqref{eq:gech} gives 
\bm{
\sup_{A\in\mathcal{R}_p}\left|\pr(W\in A)-\pr(Z\in A)\right|\\
\lesssim \frac{1}{\ul\sigma}\Biggl(
\sqrt{\log p}\E\sbra{\|R\|_\infty}
+\eps^{-1}(\log p)^{3/2}\E\sbra{\|V\|_\infty}
+\eps^{-3}(\log p)^{7/2}\E[\|G\|_\infty\|D\|_\infty^3]
+\eps\sqrt{\log p}
\Biggr).
}
Taking
$
\eps=\sqrt{\E\sbra{\|V\|_\infty}\log p}+\bra{\E[\|G\|_\infty\|D\|_\infty^3](\log p)^3}^{1/4}
$
gives the desired result.\qed

\subsection{Proof of Theorems \ref{thm:max-is} and \ref{lem:max-is}}\label{sec:max-is}

The proofs of Theorems \ref{thm:max-is} and \ref{lem:max-is} are more or less natural extensions of those of Lemmas 8 and 9 in \cite{CCK15}, respectively. 
In particular, the starting point is a symmetrization argument, which is summarized as the following lemma. 
\begin{lemma}\label{u-nemirovski}
Let $q\geq1$ and $\psi_j\in L^{q}(P^r)$ $(j=1,\dots,p)$ be degenerate, symmetric kernels of order $r\geq1$. 
Then there exists a constant $C_r$ depending only on $r$ such that
\ba{
\norm{\max_{j\in[p]}|J_r(\psi_j)|}_{L^q(\pr)}
&\leq C_r(q+\log p)^{r/2}\norm{\max_{j\in[p]}\sqrt{J_r(\psi_j^2)}}_{L^{q}(\pr)}.
}
\end{lemma}

\begin{proof}
First, by the the randomization theorem for $U$-processes \cite[Theorem 3.5.3]{de1999decoupling}, we have
\ba{
\norm{\max_{j\in[p]}|J_r(\psi_j)|}_{L^q(\pr)}
\leq C_r\norm{\max_{j\in[p]}|J^\eps_r(\psi_j)|}_{L^q(\pr)},
} 
where
\[
J^\eps_r(\psi_j):=\sum_{1\leq i_1<\cdots<i_r\leq n}\eps_{i_1}\cdots\eps_{i_r}\psi_j(X_{i_1},\dots,X_{i_r}),
\]
and $\eps_1,\dots,\eps_n$ are i.i.d.~Rademacher variables independent of $X$. 
For any $m\geq q\vee2$, we have
\ba{
\bra{\E\sbra{\max_{j\in[p]}|J^\eps_r(\psi_j)|^q\mid X}}^{1/q}
&\leq p^{1/m}\max_{j\in[p]}\bra{\E\sbra{|J^\eps_r(\psi_j)|^m\mid X}}^{1/m}\\
&\leq p^{1/m}m^{r/2}\max_{j\in[p]}\bra{\E\sbra{|J^\eps_r(\psi_j)|^2\mid X}}^{1/2}\\
&= p^{1/m}m^{r/2}\max_{j\in[p]}\sqrt{J_r(\psi_j^2)},
}
where the first inequality follows by \eqref{eq:max-out} and the second by the hypercontractivity of Rademacher chaoses \cite[Theorem 3.2.5]{de1999decoupling}. 
Taking $m=q+\log p$, we obtain
\ba{
\E\sbra{\max_{j\in[p]}|J^\eps_r(\psi_j)|^q\mid X}
\leq \bra{e(q+\log p)^{r/2}\max_{j\in[p]}\sqrt{J_r(\psi_j^2)}}^q.
}
The desired result follows by taking the expectation. 
\end{proof}

We first prove \cref{lem:max-is}. Then, \cref{thm:max-is} is obtained as its simple corollary after an application of \cref{u-nemirovski}. 
\begin{proof}[Proof of \cref{lem:max-is}]
First, since $\max_{j\in[p]}J_r(\psi_j)\leq n^r\max_{j\in[p]}M(\psi_j),$ 
the claim trivially holds if $q+\log p>n$; hence it suffices to consider the case $q+\log p\leq n$.

We prove the claim by induction on $r$. 
It is trivial when $r=0$. 
Next, suppose $r>0$ and that the claim holds for all non-negative integers less than $r$. We are going to show that there exists a constant $c_r\geq1$ depending only on $r$ such that \eqref{eq:max-is} holds. 
The following argument was inspired by the proof of \cite[Theorem 5.1]{chen2018gaussian}. 
By \eqref{eq:hoef-decomp}, we have
\ben{\label{max-is-step1}
I:=\norm{\max_{j\in[p]}J_r(\psi_j)}_{L^q(\pr)}
\leq\max_{j\in[p]}\E[J_r(\psi_j)]+\sum_{s=1}^r\binom{n-s}{r-s}\norm{\max_{j\in[p]}J_s(\pi_s\psi_j)}_{L^q(\pr)}.
}
For every $s\in[r]$, \cref{u-nemirovski} gives 
\ben{\label{applied-nemirovski}
\norm{\max_{j\in[p]}J_s(\pi_s\psi_j)}_{L^q(\pr)}
\leq C_r(q+\log p)^{s/2}\norm{\max_{j\in[p]}\sqrt{J_s\bra{(\pi_s\psi_j)^2}}}_{L^{q}(\pr)}.
}
By \eqref{eq:hoef-proj} and the so-called $c_r$-inequality,
\ba{
(\pi_s\psi_j)^2(x_1,\dots,x_s)
\leq C_r\sum_{k=0}^s\sum_{1\leq l(1)<\cdots<l(k)\leq s}(P^{r-k}\psi_j)^2(x_{l(1)},\dots,x_{l(k)}).
}
Thus, 
\ba{
J_s\bra{(\pi_s\psi_j)^2}
&\leq C_r\sum_{k=0}^s\sum_{1\leq l(1)<\cdots<l(k)\leq s}\sum_{1\leq i_1<\cdots<i_s\leq n}(P^{r-k}\psi_j)^2(X_{i_{l(1)}},\dots,X_{i_{l(k)}})\\
&=C_r\sum_{k=0}^s\binom{n-k}{s-k}\sum_{1\leq l(1)<\cdots<l(k)\leq s}\sum_{1\leq i_{l(1)}<\cdots<i_{l(k)}\leq n}(P^{r-k}\psi_j)^2(X_{i_{l(1)}},\dots,X_{i_{l(k)}})\\
&=C_r\sum_{k=0}^s\binom{n-k}{s-k}\binom{s}{k}\sum_{1\leq i_{1}<\cdots<i_{k}\leq n}(P^{r-k}\psi_j)^2(X_{i_{1}},\dots,X_{i_{k}})\\
&\leq C_r\sum_{k=0}^sn^{s-k}M(P^{r-k}\psi_j)J_k\bra{P^{r-k}\psi_j}.
}
Combining this bound with \eqref{applied-nemirovski}, the inequality $\sqrt{x+y}\leq\sqrt{x}+\sqrt{y}$ for any $x,y\geq0$ and Minkowski's inequality, we obtain
\ba{
\norm{\max_{j\in[p]}\sqrt{J_s\bra{(\pi_s\psi_j)^2}}}_{L^{q}(\pr)}
&\leq C_r\max_{0\leq k\leq s}n^{(s-k)/2}\norm{\max_{j\in[p]}\sqrt{M(P^{r-k}\psi_j)J_k\bra{P^{r-k}\psi_j}}}_{L^{q}(\pr)}\\
&\leq C_r\max_{0\leq k\leq s}n^{(s-k)/2}\norm{\max_{j\in[p]}M(P^{r-k}\psi_j)}_{L^q(\pr)}^{1/2}\norm{\max_{j\in[p]}J_k\bra{P^{r-k}\psi_j}}_{L^{q}(\pr)}^{1/2},
}
where the last inequality follows from the Schwarz inequality. 
Therefore, we have
\besn{\label{max-is-step2}
&\sum_{s=1}^r\binom{n-s}{r-s}\norm{\max_{j\in[p]}J_s(\pi_s\psi_j)}_{L^q(\pr)}\\
&\leq C_r\sum_{s=1}^{r}n^{r-s}(q+\log p)^{s/2}\max_{0\leq k\leq s}n^{(s-k)/2}\norm{\max_{j\in[p]}M(P^{r-k}\psi_j)}_{L^q(\pr)}^{1/2}\norm{\max_{j\in[p]}J_k\bra{P^{r-k}\psi_j}}_{L^{q}(\pr)}^{1/2}.
}
Now, by the assumption of the induction, for every $0\leq k< r$, there exists a constant $c_k\geq1$ depending only on $k$ such that
\ba{
\norm{\max_{j\in[p]}J_k\bra{P^{r-k}\psi_j}}_{L^{q}(\pr)}
\leq c_k\max_{0\leq l\leq k}n^{k-l}(q+\log p)^l\norm{\max_{j\in[p]}M(P^{r-l}\psi_j)}_{L^q(\pr)}.
}
Hence
\ba{
&n^{r-s}(q+\log p)^{s/2}\max_{0\leq k\leq s,k<r}n^{(s-k)/2}\norm{\max_{j\in[p]}M(P^{r-k}\psi_j)}_{L^q(\pr)}^{1/2}\norm{\max_{j\in[p]}J_k\bra{P^{r-k}\psi_j}}_{L^{q}(\pr)}^{1/2}\\
&\leq \max_{0\leq k\leq s,k<r}\sqrt{c_k}\max_{0\leq l\leq k}n^{r-(s+l)/2}(q+\log p)^{(s+l)/2}\norm{\max_{j\in[p]}M(P^{r-k}\psi_j)}_{L^q(\pr)}^{1/2}\norm{\max_{j\in[p]}M(P^{r-l}\psi_j)}_{L^q(\pr)}^{1/2}.
}
For any $0\leq k\leq s$, we have $n^{-s/2}(q+\log p)^{s/2}\leq n^{-k/2}(q+\log p)^{k/2}$ because $q+\log p\leq n$. 
Thus we obtain
\ba{
&n^{r-s}(q+\log p)^{s/2}\max_{0\leq k\leq s,k<r}n^{(s-k)/2}\norm{\max_{j\in[p]}M(P^{r-k}\psi_j)}_{L^q(\pr)}^{1/2}\norm{\max_{j\in[p]}J_k\bra{P^{r-k}\psi_j}}_{L^{q}(\pr)}^{1/2}\\
&\leq \max_{0\leq k\leq s,k<r}\sqrt{c_k}\max_{0\leq l\leq k}n^{r-(k+l)/2}(q+\log p)^{(k+l)/2}\norm{\max_{j\in[p]}M(P^{r-k}\psi_j)}_{L^q(\pr)}^{1/2}\norm{\max_{j\in[p]}M(P^{r-l}\psi_j)}_{L^q(\pr)}^{1/2}\\
&\leq\max_{0\leq k\leq s,k<r}c_kn^{r-k}(q+\log p)^{k}\norm{\max_{j\in[p]}M(P^{r-k}\psi_j)}_{L^q(\pr)}.
}
Inserting this bound into \eqref{max-is-step2} and then using \eqref{max-is-step1}, we obtain
\ba{
I
&\leq \max_{j\in[p]}\E[J_r(\psi_j)]+K_r\max_{0\leq k<r}c_kn^{r-k}(q+\log p)^{k}\norm{\max_{j\in[p]}M(P^{r-k}\psi_j)}_{L^q(\pr)}\\
&\quad+K_r(q+\log p)^{r/2}\norm{\max_{j\in[p]}M(\psi_j)}_{L^q(\pr)}^{1/2}\sqrt{I},
}
where $K_r\geq1$ is a constant depending only on $r$. 
By the AM-GM inequality,
\ba{
K_r(q+\log p)^{r/2}\norm{\max_{j\in[p]}M(\psi_j)}_{L^q(\pr)}^{1/2}\sqrt{I}
\leq \frac{K_r^2}{2}(q+\log p)^{r}\norm{\max_{j\in[p]}M(\psi_j)}_{L^q(\pr)}+\frac{I}{2}.
}
Hence we conclude
\ba{
I
&\leq 2\max_{j\in[p]}\E[J_r(\psi_j)]+K'_r\max_{0\leq k\leq r}n^{r-k}(q+\log p)^{k}\norm{\max_{j\in[p]}M(P^{r-k}\psi_j)}_{L^q(\pr)},
}
where $K_r':=K_r^2\vee\max_{0\leq k<r}2K_rc_k$. 
Since $\E[J_r(\psi_j)]=\binom{n}{r}P^r\psi_j\leq n^rP^r\psi_j$, \eqref{eq:max-is} holds with $c_r=2+K_r'$. 
\end{proof}

\begin{proof}[Proof of \cref{thm:max-is}]
By \cref{u-nemirovski} and Lyapunov's inequality,
\ba{
\norm{\max_{j\in[p]}|J_r(\psi_j)|}_{L^q(\pr)}
&\leq C_r(q+\log p)^{r/2}\norm{\max_{j\in[p]}J_r\bra{\psi_j^2}}_{L^{1\vee\frac{q}{2}}(\pr)}^{1/2}.
}
Applying \cref{lem:max-is} to the last expression gives the desired result. 
\end{proof}

\subsection{Proof of Lemmas \ref{max-rosenthal} and \ref{lem:nonneg-ada}}

Unlike Theorems \ref{thm:max-is} and \ref{lem:max-is}, the proof strategy is essentially different from that of Lemmas 8 and 9 in \cite{CCK15}. 
This is because symmetrization of a $p$-dimensional martingale in the maximum norm is no longer free lunch, producing an additional $\log p$ factor; see Propositions 5.9 and 5.38 in \cite{Pi16}. 
To avoid this issue, we rely on a classical extrapolation argument. Specifically, we use it in the following form. 
\begin{lemma}[Extrapolation principle]\label{extrapolation}
    Let $(v_i)_{i=0}^N$ and $(w_i)_{i=0}^N$ be sequences of non-negative random variables adapted to a filtration $\mathbf G=(\mcl G_i)_{i=0}^N$. 
    Fix $\alpha>0$ and assume that for any $\mathbf G$-stopping time $T$
    \[
    \|v_T1_{\{T>0\}}\|_{L^\alpha(\pr)}\leq\|w_T1_{\{T>0\}}\|_{L^\alpha(\pr)}.
    \]
    Moreover, assume that there exists a $\mathbf G$-adapted non-negative sequence $(\lambda_i)_{i=0}^{N-1}$ such that
    \[
    w_{i+1}-w_i\leq\lambda_i\quad\text{for all }i=0,1,\dots,N-1.
    \]
    Then for any $0<m<\alpha$
    \[
    \E[v_N^m]\leq\frac{\alpha}{\alpha-m}\E\sbra{w^{*m}}
    +\E\sbra{(w^{*}+\lambda^*)^m},
    \]
    where $w^*:=\max_{i=0,1,\dots,N}w_i$ and $\lambda^*:=\max_{i=0,1,\dots,N-1}\lambda_i$.
\end{lemma}

\begin{proof}
    This is a straightforward consequence of \cite[Lemma 5.23]{Pi16} once we extend $(v_i)_{i=0}^N$, $(w_i)_{i=0}^N$, $(\mathcal{G}_i)_{i=0}^N$ and $(\lambda_i)_{i=0}^{N-1}$ to infinite sequences by setting $v_i=v_N,w_i=w_N$ and $\mcl G_i=\mcl G_N$ for $i>N$ and $\lambda_i=0$ for $i\geq N$. 
    %
\end{proof}

\cref{extrapolation} allows us to reduce the proof of \cref{max-rosenthal} to moment estimates of \emph{one-dimensional} martingales. 
At this point, we need a Rosenthal type bound with sharp constants. 
\begin{lemma}[Rosenthal's inequality with sharp constants]\label{rosenthal}
Let $(\xi_i)_{i=1}^N$ be a martingale difference sequence with respect to a filtration $(\mcl G_i)_{i=0}^N$. 
There exists a universal constant $C$ such that for any $q\geq1$,
\[
\norm{\max_{n\in[N]}\abs{\sum_{i=1}^n\xi_{ij}}}_{L^q(\pr)}
\leq C \bra{\sqrt q\norm{\sqrt{\sum_{i=1}^N\E[\xi_{ij}^2\mid\mcl G_{i-1}]}}_{L^q(\pr)}
+q\norm{\max_{i\in[N]}|\xi_i|}_{L^q(\pr)}}.
\]
\end{lemma}

\begin{proof}
For the case $1\leq q\leq 2$, see \cite[Theorem 2.6]{Pi94} or \cite[Corollary 3.6]{van2021maximal}. 
For the case $q\geq2$, see \cite[Theorem 4.1]{Pi94} or \cite[Theorem 3.1]{van2021maximal}. 
Note that $\mathbb R$ is a $(2,1)$-smooth Banach space as it is a Hilbert space. 
\end{proof}

\begin{proof}[Proof of \cref{max-rosenthal}]
We consider the Davis decomposition of $(\xi_i)_{i=1}^N$. 
Let $\xi_n^*:=\max_{i\in[n]}\|\xi_i\|_\infty$ for every $n=0,1,\dots,N$. 
Define
\[
\xi_i':=\xi_i1_{\{\xi_i^*\leq2\xi_{i-1}^*\}}-\E[\xi_i1_{\{\xi_i^*\leq2\xi_{i-1}^*\}}\mid\mcl G_{i-1}]
\quad\text{and}\quad
\xi_i'':=\xi_i-\xi_i'\quad\text{for every }i\in[N]. 
\]
Since $\E[\xi_i\mid\mcl G_{i-1}]=0$, we have $\xi_i''=\xi_i1_{\{\xi_i^*>2\xi_{i-1}^*\}}-\E[\xi_i1_{\{\xi_i^*>2\xi_{i-1}^*\}}\mid\mcl G_{i-1}]$. Hence
\ba{
\norm{\max_{j\in[p]}\max_{n\in[N]}\abs{\sum_{i=1}^n\xi''_{ij}}}_{L^m(\pr)}
&\leq\norm{\sum_{i=1}^N\|\xi_{i}\|_\infty1_{\{\xi_i^*>2\xi_{i-1}^*\}}}_{L^m(\pr)}
+\norm{\sum_{i=1}^N\E[\|\xi_{i}\|_\infty1_{\{\xi_i^*>2\xi_{i-1}^*\}}\mid\mcl G_{i-1}]}_{L^m(\pr)}\\
&\leq(1+m)\norm{\sum_{i=1}^N\|\xi_{i}\|_\infty1_{\{\xi_i^*>2\xi_{i-1}^*\}}}_{L^m(\pr)},
}
where the second inequality follows from the dual to Doob's inequality \cite[Theorem 1.26]{Pi16}. 
When $\xi_i^*>2\xi_{i-1}^*$, we have $\|\xi_i\|_\infty\leq\|\xi_i\|_\infty+(\xi_i^*-2\xi_{i-1}^*)\leq2(\xi_i^*-\xi_{i-1}^*)$. Hence $\|\xi_i\|_\infty1_{\{\xi_i^*>2\xi_{i-1}^*\}}\leq2(\xi_i^*-\xi_{i-1}^*)$. 
Consequently,
\ba{
\norm{\sum_{i=1}^N\|\xi_{i}\|_\infty1_{\{\xi_i^*>2\xi_{i-1}^*\}}}_{L^m(\pr)}
\leq2\norm{\sum_{i=1}^N(\xi_i^*-\xi_{i-1}^*)}_{L^m(\pr)}
=2\norm{\xi_N^*}_{L^m(\pr)}.
}
Therefore, we complete the proof once we show
\ben{\label{max-rosenthal-aim}
\norm{\max_{j\in[p]}\max_{n\in[N]}\abs{\sum_{i=1}^n\xi'_{ij}}}_{L^m(\pr)}
\lesssim \sqrt{\alpha}\norm{\max_{j\in[p]}\sqrt{\sum_{i=1}^N\E[\xi_{ij}^2\mid\mcl G_{i-1}]}}_{L^m(\pr)}+\alpha\norm{\xi^*_{N}}_{L^m(\pr)},
}
where $\alpha:=m+\log p$.
By construction, $(\xi_i')_{i=1}^N$ is a martingale difference sequence in $\mathbb R^p$. 
Set $S_n':=\sum_{i=1}^n\xi_i'$ for $n\in[N]$ and $S_0':=0\in\mathbb R^p$. 
Then $(S'_n)_{n=0}^N$ is a martingale in $\mathbb R^p$. 
For any $\mathbf G$-stopping time $T$, we have by \eqref{eq:max-out}
\ba{
\norm{\sup_{n\in[T]}\|S'_n\|_\infty1_{\{T>0\}}}_{L^\alpha(\pr)}
\leq e\max_{j\in[p]}\norm{\sup_{n\in[T]}|S'_{n,j}|1_{\{T>0\}}}_{L^\alpha(\pr)}.
}
For every $j\in[p]$, $(S'_{n\wedge T,j}1_{\{T>0\}})_{n=0}^N$ is a martingale, so \cref{rosenthal} yields
\ba{
\norm{\sup_{n\in[T]}|S'_{n,j}|1_{\{T>0\}}}_{L^\alpha(\pr)}
\lesssim\sqrt{\alpha}\norm{\sqrt{\sum_{i=1}^T\E[|\xi_{ij}'|^2\mid\mcl G_{i-1}]1_{\{T>0\}}}}_{L^\alpha(\pr)}
+\alpha\norm{\sup_{i\in[T]}|\xi_{ij}'|1_{\{T>0\}}}_{L^\alpha(\pr)}.
}
Since $\E[|\xi_{ij}'|^2\mid\mcl G_{i-1}]\leq\E[|\xi_{ij}|^21_{\{\xi_i^*\leq2\xi_{i-1}^*\}}\mid\mcl G_{i-1}]\leq\E[|\xi_{ij}|^2\mid\mcl G_{i-1}]$, we conclude 
\[
\norm{\sup_{n\in[T]}\|S'_{n}\|_\infty1_{\{T>0\}}}_{L^\alpha(\pr)}\leq c\norm{B_T1_{\{T>0\}}}_{L^\alpha(\pr)},
\] 
where $c>0$ is a universal constant and
\[
B_n:=\max_{j\in[p]}\bra{\sqrt{\alpha\sum_{i=1}^n\E[\xi_{ij}^2\mid\mcl G_{i-1}]}+\alpha\sup_{i\in[n]}|\xi'_{ij}|}\quad\text{for }n\in[N]\text{ and }B_0:=0.
\]
Observe that $|\xi'_{ij}|\leq4\xi_{i-1}^*$ for every $i\in[N]$ by construction. 
Hence
\[
B_{n+1}-B_n\leq\max_{j\in[p]}\sqrt{\alpha\E[|\xi_{n+1,j}|^2\mid\mcl G_{n}]}+4\alpha\xi_{n}^*=:D_n\quad\text{for all }n=0,1,\dots,N-1.
\]
Since $(D_n)_{n=0}^{N-1}$ is $\mathbf G$-adapted, we can apply \cref{extrapolation} with $v_n=\sup_{k\in[n]}\|S'_{k}\|_\infty$, $w_n=cB_n$ and $\lambda_n=cD_n$. 
Since $\{\alpha/(\alpha-m)\}^{1/m}\leq (m+1)^{1/m}\lesssim1$, this gives
\[
\norm{\sup_{n\in[N]}\|S'_{n}\|_\infty}_{L^m(\pr)}
\lesssim \norm{\sup_{n\in[N]}(B_n\vee D_{n-1})}_{L^m(\pr)}.
\]
Since 
\[
\sup_{n\in[N]}(B_n\vee D_{n-1})\leq\max_{j\in[p]}\sqrt{\alpha\sum_{i=1}^N\E[\xi_{ij}^2\mid\mcl G_{i-1}]}+4\alpha\xi^*_N,
\]
we obtain \eqref{max-rosenthal-aim} via Minkowski's inequality.
\end{proof}

\begin{proof}[Proof of \cref{lem:nonneg-ada}]
We follow the proof of \cite[Theorem 5.1]{Hi90}. 
Let $\eta_n^*:=\max_{i\in[n]}\|\eta_i\|_\infty$ for every $n\geq0$. 
Define $\eta_i':=\eta_i1_{\{\eta_i^*\leq2\eta_{i-1}^*\}}$ and $\eta_i'':=\eta_i-\eta_i'$. 
By the proof of \cref{max-rosenthal}, we have $\|\eta_i''\|_\infty=\|\eta_i\|_\infty1_{\{\eta_i^*>2\eta_{i-1}^*\}}\leq2(\eta_i^*-\eta_{i-1}^*)$. Hence
\[
\E\sbra{\max_{j\in[p]}\sum_{i=1}^N\eta''_{ij}}\leq2\E[\eta_N^*].
\]
Therefore, we complete the proof once we show
\ben{\label{nonneg-aim}
\E\sbra{\max_{j\in[p]}\sum_{i=1}^N\eta'_{ij}}
\lesssim \E\sbra{\max_{j\in[p]}\sum_{i=1}^N\E[\eta_{ij}\mid\mcl G_{i-1}]}
+\E[\eta_N^*]\log p.
}
With $\xi'_{i}:=\eta'_{i}-\E[\eta'_{i}\mid\mcl G_{i-1}]$ for every $i\in[N]$, we can bound the right hand side of \eqref{nonneg-aim} as
\ben{\label{nonneg-eq1}
\E\sbra{\max_{j\in[p]}\sum_{i=1}^N\eta'_{ij}}
\leq\E\sbra{\max_{j\in[p]}\sum_{i=1}^N\E[\eta'_{ij}\mid\mcl G_{i-1}]}
+\E\sbra{\max_{j\in[p]}\abs{\sum_{i=1}^N\xi_{ij}}}
=:I+II.
}
By definition,
\ben{\label{nonneg-eq2}
I\leq\E\sbra{\max_{j\in[p]}\sum_{i=1}^N\E[\eta_{ij}\mid\mcl G_{i-1}]}.
}
Meanwhile, since $(\xi_i)_{i=1}^N$ is a martingale difference sequence in $\mathbb R^p$ with respect to $(\mcl G_i)_{i=0}^N$ by construction, we have by \cref{max-rosenthal}
\ba{
II\lesssim\E\sbra{\max_{j\in[p]}\sqrt{\sum_{i=1}^N\E[\xi_{ij}^2\mid\mcl G_{i-1}]}}\sqrt{\log p}
+\E\sbra{\max_{i\in[N]}\|\xi_i\|_\infty}\log p.
}
Since $\|\xi_{i}\|_\infty\leq4\eta_{i-1}^*\leq4\eta_N^*$ by construction,
\ba{
II\lesssim\E\sbra{\max_{j\in[p]}\sqrt{\eta_N^*\sum_{i=1}^N\E[|\xi_{ij}|\mid\mcl G_{i-1}]}}\sqrt{\log p}
+\E[\eta_N^*]\log p.
}
By the AM-GM inequality,
\ba{
\E\sbra{\max_{j\in[p]}\sqrt{\eta_N^*\sum_{i=1}^N\E[|\xi_{ij}|\mid\mcl G_{i-1}]}}\sqrt{\log p}
\leq\frac{1}{2}\bra{\E[\eta_N^*]\log p+\E\sbra{\max_{j\in[p]}\sum_{i=1}^N\E[|\xi_{ij}|\mid\mcl G_{i-1}]}}.
}
Since $\E[|\xi_{ij}|\mid\mcl G_{i-1}]\leq2\E[\eta_{ij}'\mid\mcl G_{i-1}]\leq2\E[\eta_{ij}\mid\mcl G_{i-1}]$, we conclude
\ben{\label{nonneg-eq3}
II\lesssim\E\sbra{\max_{j\in[p]}\sum_{i=1}^N\E[\eta_{ij}\mid\mcl G_{i-1}]}+\E[\eta_N^*]\log p.
}
Combining \eqref{nonneg-eq1}--\eqref{nonneg-eq3} gives \eqref{nonneg-aim}. 
\end{proof}

\subsection{Proof of Lemma \ref{influence-mom}}

We need the following auxiliary estimate for the proof of \eqref{influence-mom-1}. 
\begin{lemma}\label{nemirovski-Lp}
Let $m\geq1$ and $\psi_j\in L^m(P^2)$ $(j\in[p])$. 
There exists a universal constant $C$ such that
\besn{\label{eq:nemirovski-Lp}
&\E\sbra{\max_{i\in[n]}\max_{j\in[p]}\int_S\abs{\sum_{i'\in[n]:i'< i}\cbra{\psi_j(X_{i'},x)-\E[\psi_j(X_{i'},x)]}}^mP(dx)}\\
&\leq (C\sqrt{m+\log p})^m\E\sbra{\max_{j\in[p]}\int_S\bra{\sum_{i=1}^{n-1}\psi_j(X_i,x)^2}^{m/2}P(dx)}.
}
\end{lemma}

\begin{proof}
Consider the vector space $\mathbb B:=L^m(P)^p=\{(f_1,\dots,f_p):f_1,\dots,f_p\in L^m(\pr)\}$ equipped with a norm $(f_1,\dots,f_p)\mapsto\max_{j\in[p]}\|f_j\|_{L^m(\pr)}$. 
It is straightforward to check that $\mathbb B$ is a Banach space. 
Then, for every $i\in[n]$, we define a map $\Psi_i:\Omega\to\mathbb B$ as follows: First, for $j\in[p]$ and $x\in S$, define a function $\psi_j^x:S\to\mathbb R$ as $\psi_j^x(y)=\psi_j(x,y)$ for $y\in S$. Fubini's theorem implies that $\psi_j^x\in L^m(P)$ $P$-a.s.~$x$. Since the law of $X_i$ is $P$, this means that $\psi_j^{X_i(\omega)}\in L^m(P)$ $\pr$-a.s.~$\omega$. 
Hence we can define the map $\Psi_i$ as $\Psi_i(\omega)=(\Psi_{i1}(\omega),\dots,\Psi_{ip}(\omega)):=(\psi_1^{X_i(\omega)},\dots,\psi_p^{X_i(\omega)})$ for $\omega\in\Omega$. 
Using the fact that $\{1_{E_1\times E_2}:E_1,E_2\in\mcl S\}$ is total in $L^m(P^2)$, one can easily verify that $\Psi_i$ is strongly $\pr$-measurable (see \citealp[Definition 1.1.14]{HvNVW16}). In particular, there exists a closed separable subspace $\mathbb B_0\subset\mathbb B$ such that $\Psi_i(\omega)\in \mathbb B_0$ $\pr$-a.s.~$\omega$ by the Pettis measurability theorem \cite[Theorem 1.1.20]{HvNVW16}. 
Further, by construction
\[
\max_{j\in[p]}\int_S\abs{\sum_{i'\in[n]:i'< i}\cbra{\psi_j(X_{i'},x)-\E[\psi_j(X_{i'},x)]}}^mP(dx)
=\norm{\sum_{i'\in[n]:i'< i}\bra{\Psi_{i'}-\E[\Psi_{i'}]}}_{\mathbb B}^m.
\]
Therefore, the left-hand side of \eqref{eq:nemirovski-Lp} is equal to
\ba{
I
&:=\E\sbra{\max_{i\in[n]}\norm{\sum_{i'\in[n]:i'< i}\bra{\Psi_{i'}-\E[\Psi_{i'}]}}_{\mathbb B}^m}.
}
By a standard symmetrization argument (cf.~the proof of \citealp[Lemma 1.2.6]{de1999decoupling}),
\[
I\leq 2^{m+1}\E\sbra{\norm{\sum_{i=1}^{n-1}\eps_i\Psi_{i}}_{\mathbb B}^m},
\]
where $\eps_1,\dots,\eps_n$ are i.i.d.~Rademacher variables independent of $X$. 
With $\alpha:=m+\log p$, we have by \eqref{eq:max-out}
\ba{
\bra{\E\sbra{\norm{\sum_{i=1}^{n-1}\eps_i\Psi_{i}}_{\mathbb B}^m\mid X}}^{1/m}
&\leq e\max_{j\in[p]}\bra{\E\sbra{\norm{\sum_{i=1}^{n-1}\eps_i\Psi_{ij}}_{L^m(P)}^\alpha\mid X}}^{1/\alpha}.
}
Khintchine's inequality in $L^m(P)$ \cite[Proposition 6.3.3]{HvNVW17} gives
\ba{
\bra{\E\sbra{\norm{\sum_{i=1}^{n-1}\eps_i\Psi_{ij}}_{L^m(P)}^\alpha\mid X}}^{1/\alpha}
\leq\sqrt{\alpha-1}\norm{\bra{\sum_{i=1}^{n-1}\Psi_{ij}^2}^{1/2}}_{L^m(P)}.
}
As a result,
\ba{
I
\leq 2(2e\sqrt{m+\log p})^m\E\sbra{\max_{j\in[p]}\norm{\bra{\sum_{i=1}^{n-1}\Psi_{ij}^2}^{1/2}}_{L^{m}(P)}^{m}}.
}
Since
\[
\norm{\bra{\sum_{i=1}^{n-1}\Psi_{ij}^2}^{1/2}}_{L^{m}(P)}^{m}=\int_S\bra{\sum_{i=1}^{n-1}\psi_j(X_i,x)^2}^{m/2}P(dx),
\]
we complete the proof.
\end{proof}

\begin{proof}[Proof of \cref{influence-mom}]
First, we prove \eqref{influence-mom-1}. 
By \cref{nemirovski-Lp},
\ba{
I&:=\E\sbra{\max_{i\in[n]}\max_{j\in[p]}\int_S\abs{\sum_{i'\in[n]:i'< i}\psi_j(X_{i'},x)}^4P(dx)}\\
&\lesssim (\log p)^2\E\sbra{\max_{j\in[p]}\int_S\bra{\sum_{i=1}^{n-1}\psi_j(X_i,x)^2}^{2}P(dx)}\\
&\lesssim (\log p)^2\Biggl(\max_{j\in[p]}\int_S\bra{\sum_{i=1}^{n-1}P(\psi_j^2)(x)}^{2}P(dx)\\
&\hphantom{\lesssim(\log p)^2}+\E\sbra{\max_{j\in[p]}\int_S\abs{\sum_{i=1}^{n-1}\{\psi_j(X_{i},x)^2-P(\psi_j^2)(x)\}}^{2}P(dx)}\Biggr)\\
&=:I_1+I_2.
}
By definition,
\[
I_1\leq n^{2}\max_{j\in[p]}\|P(\psi_j^2)\|_{L^{2}(P)}^{2}\log^2p.
\]
Meanwhile, applying \cref{nemirovski-Lp} to functions $(y,x)\mapsto\psi_j(y,x)^2$ $(j\in[p])$, we obtain
\ba{
I_2\lesssim (\log p)^3\E\sbra{\max_{j\in[p]}\int_S\sum_{i=1}^{n-1}\psi_j(X_{i},x)^4P(dx)}
=(\log p)^3\E\sbra{\max_{j\in[p]}\sum_{i=1}^{n-1}P(\psi_j^4)(X_{i})}.
}
Lemma 9 in \cite{CCK15} gives
\ba{
\E\sbra{\max_{j\in[p]}\sum_{i=1}^{n-1}P(\psi_{j}^4)(X_i)}
&\lesssim \max_{j\in[p]}\sum_{i=1}^{n-1}\E\sbra{P(\psi_{j}^4)(X_i)}
+\E\sbra{\max_{i\in[n]}\max_{j\in[p]}P(\psi_{j}^4)(X_i)}\log p\\
&\leq n\max_{j\in[p]}\|\psi_j\|_{L^4(P^2)}^4+\E\sbra{\max_{j\in[p]}M(P(\psi_{j}^4))}\log p.
}
Consequently, we obtain \eqref{influence-mom-1}. 

Next, we prove \eqref{influence-mom-2}. 
Since
\[
\sum_{i'\in[n]:i'\neq i}\psi_j(X_{i'},X_i)
=\sum_{i'\in[n]:i'< i}\psi_j(X_{i'},X_i)
+\sum_{i'\in[n]:i'> i}\psi_j(X_{i'},X_i),
\]
we have
\ba{
&\E\sbra{\max_{j\in[p]}\sum_{i=1}^n\abs{\sum_{i'\in[n]:i'\neq i}\psi_j(X_{i'},X_i)}^4}\\
&\leq 8\bra{\E\sbra{\max_{j\in[p]}\sum_{i=1}^n\abs{\sum_{i'\in[n]:i'< i}\psi_j(X_{i'},X_i)}^4}
+\E\sbra{\max_{j\in[p]}\sum_{i=1}^n\abs{\sum_{i'\in[n]:i'> i}\psi_j(X_{i'},X_i)}^4}}.
}
Since $(X_i)_{i=1}^n$ is i.i.d., it has the same law as $(X_{n-i+1})_{i=1}^n$. Hence 
\ba{
\E\sbra{\max_{j\in[p]}\sum_{i=1}^n\abs{\sum_{i'\in[n]:i'> i}\psi_j(X_{i'},X_i)}^4}
=\E\sbra{\max_{j\in[p]}\sum_{i=1}^n\abs{\sum_{i'\in[n]:i'< i}\psi_j(X_{i'},X_i)}^4},
}
and thus
\ban{
\E\sbra{\max_{j\in[p]}\sum_{i=1}^n\abs{\sum_{i'\in[n]:i'\neq i}\psi_j(X_{i'},X_i)}^4}
&\leq 16\E\sbra{\max_{j\in[p]}\sum_{i=1}^n\abs{\sum_{i'\in[n]:i'< i}\psi_j(X_{i'},X_i)}^4}\notag\\
&=:16II.\label{non-u:1st-reduction}
}
To bound $II$, we are going to apply \cref{lem:nonneg-ada}. Define a filtration $(\mcl G_i)_{i=1}^n$ as $\mcl G_i:=\sigma(X_1,\dots,X_i)$ for $i\in[n]$. 
Also, for every $i\in[n]$, define a random vector $\eta_i=(\eta_{i1},\dots,\eta_{ip})^\top$ as 
\[
\eta_{ij}:=\abs{\sum_{i'\in[n]:i'< i}\psi_j(X_{i'},X_i)}^4,\quad j=1,\dots,p.
\]
Then $(\eta_i)_{i=1}^n$ is adapted to the filtration $(\mcl G_i)_{i=1}^n$. 
Hence \cref{lem:nonneg-ada} gives
\ben{
II\lesssim \E\sbra{\max_{j\in[p]}\sum_{i=1}^n\E[\eta_{ij}\mid\mcl G_{i-1}]}
+\E\sbra{\max_{i\in[n]}\max_{j\in[p]}\eta_{ij}}\log p
=:III+IV\log p,\label{non-u:I-bound}
}
where we set $\mcl G_0:=\{\emptyset,\Omega\}$. 
Since $X_i$ is independent of $\mcl G_{i-1}$ for every $i$, we have
\ba{
III&=\E\sbra{\max_{j\in[p]}\sum_{i=1}^n\int_S\abs{\sum_{i'\in[n]:i'< i}\psi_j(X_{i'},x)}^4P(dx)}\leq nI.
}
Hence, the first part of the proof gives
\ben{
III\lesssim 
n^3\max_{j\in[p]}\|P(\psi_j^2)\|_{L^{2}(P)}^{2}\log^{2}p
+n^2\max_{j\in[p]}\|\psi_j\|_{L^4(P^2)}^4\log^{3}p\\
+n\E\sbra{\max_{j\in[p]}M(P(\psi_{j}^4))}\log^{4} p.
}
To bound $IV$, recall that $(X_i)_{i=1}^n$ has the same law as $(X_{n-i+1})_{i=1}^n$. 
Thus, $IV$ can be rewritten as
\ba{
IV=\E\sbra{\max_{i\in[n]}\max_{j\in[p]}\abs{\sum_{i'\in[n]:i'> i}\psi_j(X_{i'},X_i)}^4}
=\E\sbra{\max_{(i,j)\in[n]\times[p]}\abs{\sum_{i'=1}^nY_{i',(i,j)}}^4},
}
where $Y_{i',(i,j)}=\psi_j(X_{i'},X_i)$ if $i'>i$ and $Y_{i',(i,j)}=0$ otherwise. 
Observe that $(Y_{i',(i,j)})_{i'=1}^n$ is a martingale difference sequence with respect to $(\mcl G_{i'})_{i'=0}^n$ for all $i\in[n]$ and $j\in[p]$. 
Hence \cref{max-rosenthal} gives
\ban{
IV&\lesssim \E\sbra{\max_{i\in[n],j\in[p]}\bra{\sum_{i'=1}^n\E[Y_{i',(i,j)}^2\mid\mcl G_{i'-1}]}^{2}}\log^{2} (np)
+\E\sbra{\max_{i,i'\in[n],j\in[p]}|Y_{i',(i,j)}|^4}\log^4(np)\notag\\
&\leq n^{2}\E\sbra{\max_{i\in[n],j\in[p]}P(\psi_j^2)(X_i)^{2}}\log^{2} (np)
+\E\sbra{\max_{j\in[p]}M(\psi_j)^4}\log^4(np).
\label{non-u:III-bound}
}
Combining \eqref{non-u:1st-reduction} with \eqref{non-u:I-bound}--\eqref{non-u:III-bound} gives the desired result. 
\end{proof}

\subsection{Proof of Lemma \ref{lem:delta}}\label{proof:delta}

By the same reasoning as in the proof of \cref{thm:clt-ustat}, we may assume $\sigma_j=1$ for all $j\in[p]$ without loss of generality. 

The proof of \cref{lem:delta} is based on elementary but lengthy computations using properties of contraction kernels. 
In addition to the basic properties given in \cite[Lemma 2.4]{DoPe19}, we need the following ones. 
\begin{lemma}\label{max-contraction}
Given two symmetric kernels $\psi\in L^2(P^{r}),\varphi\in L^2(P^{r'})$ and two integers $0\leq l\leq s\leq r\wedge r'$, we have the following properties.  
\begin{enumerate}[label=(\alph*)]

    \item\label{contr-c} If $r=r'$, then $\psi\star_r^l\varphi=P^{l}(\psi\varphi)$.

    \item\label{contr-b} $M(\psi\star_s^l\varphi)^2\leq M(P^l(\psi^2))M(P^l(\varphi^2))$.

    \item \label{contr-d} For $P^{r'-s}$-a.s.~$v\in S^{r'-s}$,
    \[
\int_{S^{r-s}}\psi\star_s^s\varphi(u,v)^2P^{r-s}(du)\leq\|\psi\star_s^s\psi\|_{L^2(P^{2r-2s})}P^s(\varphi^2)(v).
    \]

\end{enumerate}
\end{lemma}

\begin{proof}
Property \ref{contr-c} immediately follows by definition. 
Property \ref{contr-b} follows from the Schwarz inequality. 
Let us prove property \ref{contr-d}. 
Using Fubini's theorem repeatedly, we obtain for $P^{r'-s}$-a.s.~$v$
\ba{
\int_{S^{r-s}}\psi\star_s^s\varphi(u,v)^2P^{r-s}(du)
&=\int_{S^{r+s}}\psi(y,u)\varphi(y,v)\psi(y',u)\varphi(y',v)P^s(dy)P^s(dy')P^{r-s}(du)\\
&=\int_{S^{2s}}\psi\star^{r-s}_{r-s}\psi(y,y')\varphi(y,v)\varphi(y',v)P^s(dy)P^s(dy').
}
Hence, the Schwarz inequality gives
\ba{
\int_{S^{r-s}}\psi\star_s^s\varphi(u,v)^2P(du)
&\leq\|\psi\star^{r-s}_{r-s}\psi\|_{L^2(P^{2s})}\sqrt{\int_{S^{2s}}\varphi(y,v)^2\varphi(y',v)^2P^s(dy)P^s(dy')}\\
&=\|\psi\star^s_s\psi\|_{L^2(P^{2r-2s})}P^s(\varphi^2)(v),
}
where the last equality follows by Eq.(7.7) in \cite{DoPe19} and Fubini's theorem.
\end{proof}

\begin{corollary}\label{aux-delta1}
For any $a,b\in[r]$, $s\in[a\wedge b]$ and $0\leq l\leq s\wedge(a+b-s-1)$,
\bm{
\Delta_1(a,b;s,l,a+b-l-s)
\leq 
n^{2r+l-2a}(\log p)^{2a-l-s}\sqrt{\E\sbra{\max_{j\in[p]}\frac{M(P^l(|\pi_a\psi_j|^2))^2}{\sigma_j^2}}}\\
+n^{2r+l-2b}(\log p)^{2b-l-s}\sqrt{\E\sbra{\max_{j\in[p]}\frac{M(P^l(|\pi_b\psi_{j}|^2))^2}{\sigma_j^2}}}.
}
\end{corollary}

\begin{proof}
Recall that we may assume $\sigma_j=1$ for all $j\in[p]$. 
By \eqref{sym-M},
\ba{
\Delta_1(a,b;s,l,a+b-l-s)
\leq n^{2r+l-a-b}(\log p)^{a+b-l-s}
\sqrt{\E\sbra{\max_{j,k\in[p]}M(\pi_a\psi_j\star^{l}_s\pi_b\psi_{k})^2}}.
}
By \cref{max-contraction}\ref{contr-b} and the AM-GM inequality,
\ba{
&2n^{-2a-2b}(\log p)^{2(a+b)}\max_{j,k\in[p]}M(\pi_a\psi_j\star^{l}_s\pi_b\psi_{k})^2\\
&\leq n^{-4a}(\log p)^{4a}\max_{j\in[p]}M(P^l(|\pi_a\psi_j|^2))^2
+n^{-4b}(\log p)^{4b}\max_{k\in[p]}M(P^l(|\pi_b\psi_{k}|^2))^2.
}
Combining these bounds gives the desired result. 
\end{proof}

\begin{proof}[Proof of \eqref{lem:delta-eq11}]
By \cref{aux-delta1},
\ba{
\Delta_1(1,1;1,0,1)\log^2p\leq2n^2(\log p)^3\sqrt{\E\sbra{\max_{j\in[p]}M(\pi_1\psi_j)^4}}
\leq2\sqrt{\Delta_{2,*}^{(2)}(1)\log^5p}.
}
Also, by Lemma 2.4(iv) in \cite{DoPe19},
\ba{
\Delta_1(1,1;1,0,0)\log^2p\leq n^{\frac{5}{2}}\max_{j\in[p]}\|\pi_1\psi_j\|_{L^4(P)}^2\log^{5/2}p
=\sqrt{\Delta_{2,*}^{(1)}(1)\log^5p}.
}
Hence we obtain \eqref{lem:delta-eq11}.
\end{proof}

\begin{proof}[Proof of \eqref{lem:delta-eq22}]
Observe that
\ba{
\Delta_1(2,2)
&=\max_{0\leq u\leq 3}\Delta_1(2,2;1,0,u)
+\max_{0\leq u\leq 2}\Delta_1(2,2;1,1,u)\\
&\quad+\max_{0\leq u\leq 2}\Delta_1(2,2;2,0,u)
+\max_{0\leq u\leq 1}\Delta_1(2,2;2,1,u)\\
&=:\max_{0\leq u\leq 3}I_u
+\max_{0\leq u\leq 2}II_u
+\max_{0\leq u\leq 2}III_u
+\max_{0\leq u\leq 1}IV_u.
}
By \cref{aux-delta1},
\ben{\label{I3-III2-bound}
\bra{I_3\vee III_2}\log^2p\leq \sqrt{\E\sbra{\max_{j\in[p]}M\bra{\pi_2\psi_j}^4}\log^{10}p}
\leq\sqrt{\Delta_{2,*}^{(4)}(2)\log^5p}.
}
and
\ben{\label{II2-IV1-bound}
(II_2\vee IV_1)\log^2p\leq n\sqrt{\E\sbra{\max_{j\in[p]}M\bra{P(|\pi_2\psi_j|^2)}^2}}\log^4p
\leq\sqrt{\Delta_{2,*}^{(5)}(2)\log^5p}.
}
Next, by \eqref{sym-lp},
\ba{
I_0&\leq n^{\frac{3}{2}}(\log p)^{\frac{3}{2}}\max_{j,k\in[p]}\|\pi_2\psi_j\star^{0}_1\pi_2\psi_{k}\|_{L^2(P^3)},\\
II_0&\leq n^2(\log p)\max_{j,k\in[p]}\|\pi_2\psi_j\star^{1}_1\pi_2\psi_{k}\|_{L^2(P^2)},\\
III_0&\leq n(\log p)\max_{j\in[p]}\|\pi_2\psi_j\star^{0}_2\pi_2\psi_{k}\|_{L^2(P^2)},\\
IV_0&\leq n^{\frac{3}{2}}(\log p)^{\frac{1}{2}}\max_{j,k\in[p]}\|\pi_2\psi_j\star^{1}_2\pi_2\psi_{k}\|_{L^2(P)}.
}
By \cite[Lemma 2.4(iii)]{DoPe19} and \cref{max-contraction}\ref{contr-c},
\ba{
\max_{j,k\in[p]}\|\pi_2\psi_j\star^{0}_1\pi_2\psi_{k}\|_{L^2(P^3)}
\leq\max_{j\in[p]}\|\pi_2\psi_j\star^{1}_2\pi_2\psi_{j}\|_{L^2(P)}
=\max_{j\in[p]}\|P(|\pi_2\psi_j|^2)\|_{L^2(P)}.
}
Hence
\ben{\label{I0-bound}
(I_0\vee IV_0)\log^2p\leq n^{\frac{3}{2}}\max_{j\in[p]}\|P(|\pi_2\psi_j|^2)\|_{L^2(P)}(\log p)^{\frac{7}{2}}
\leq\sqrt{\Delta_{2,*}^{(2)}(2)\log^5p}.
}
Also, Lemma 2.4(vi) in \cite{DoPe19} gives
\ben{\label{II0-bound}
II_0\log^2p\leq \Delta_1^{(0)}\log^3p.
}
Moreover, Lemma 2.4(iv) in \cite{DoPe19} gives
\ben{\label{III0-bound}
III_0\log^2p\leq n\max_{j\in[p]}\|\pi_2\psi_j\|_{L^4(P^2)}^2\log^3p
\leq\sqrt{\Delta_{2,*}^{(1)}(2)\log^5p}.
}
It remains to bound $I_1,I_2,II_1$ and $III_1$. 

\paragraph{Step 1.} Let us bound
\ba{
I_1=n(\log p)^2\sqrt{\E\sbra{\max_{j,k\in[p]}M\bra{P^{2}(|\wt{\pi_2\psi_j\star^{0}_1\pi_2\psi_{k}}|^2)}}}.
}
For any $j,k\in[p]$ and $i\in[n]$, Jensen's inequality gives
\ba{
&P^{2}(|\wt{\pi_2\psi_j\star^{0}_1\pi_2\psi_{k}}|^2)(X_i)\\
&\leq\frac{1}{6}\int\pi_2\psi_j(y,u)^2\pi_2\psi_k(y,X_i)^2P(dy)P(du)
+\frac{1}{6}\int\pi_2\psi_j(y,X_i)^2\pi_2\psi_k(y,u)^2P(dy)P(du)\\
&\quad+\frac{1}{6}\int\pi_2\psi_j(u,y)^2\pi_2\psi_k(u,X_i)^2P(dy)P(du)
+\frac{1}{6}\int\pi_2\psi_j(u,X_i)^2\pi_2\psi_k(u,y)^2P(dy)P(du)\\
&\quad+\frac{1}{6}\int\pi_2\psi_j(X_i,y)^2\pi_2\psi_k(X_i,u)^2P(dy)P(du)
+\frac{1}{6}\int\pi_2\psi_j(X_i,u)^2\pi_2\psi_k(X_i,y)^2P(dy)P(du)\\
&\leq\frac{2}{3}\max_{j,k\in[p]}M\bra{P\bra{|\pi_2\psi_j|^2}\star_1^1|\pi_2\psi_k|^2}
+\frac{1}{3}\max_{j\in[p]}M\bra{P\bra{|\pi_2\psi_j|^2}}^2.
}
By \cref{max-contraction}\ref{contr-b},
\ba{
M\bra{P\bra{|\pi_2\psi_j|^2}\star_1^1|\pi_2\psi_k|^2}
\leq\|P\bra{|\pi_2\psi_j|^2}\|_{L^2(P)}\sqrt{M\bra{P(|\pi_2\psi_k|^4)}}.
}
Hence, by the AM-GM inequality,
\ba{
n^2M\bra{P\bra{|\pi_2\psi_j|^2}\star_1^1|\pi_2\psi_k|^2}\log^8p
\leq \frac{n^3}{2}\|P\bra{|\pi_2\psi_j|^2}\|_{L^2(P)}^2\log^7p+\frac{n}{2}M\bra{P(|\pi_2\psi_k|^4)}\log^9p.
}
All together, we obtain
\ban{
(I_1\log^2p)^2
&\leq 
n^3\max_{j\in[p]}\|P\bra{|\pi_2\psi_j|^2}\|_{L^2(P)}^2\log^7p
+n\E\sbra{\max_{j\in[p]}M\bra{P\bra{|\pi_2\psi_j|^4}}}\log^9p\notag\\
&\quad+n^2\E\sbra{\max_{j\in[p]}M\bra{P\bra{|\pi_2\psi_j|^2}}^2}\log^8p\notag\\
&\leq\bra{\Delta_{2,*}^{(2)}(2)+\Delta_{2,*}^{(3)}(2)+\Delta_{2,*}^{(5)}(2)}\log^5p,\label{I1-bound}
}
where we used \eqref{eq:max-out} in the last line. 

\paragraph{Step 2.} Let us bound
\ba{
I_2&=n^{\frac{1}{2}}(\log p)^{\frac{5}{2}}\sqrt{\E\sbra{\max_{j,k\in[p]}M\bra{P(|\wt{\pi_2\psi_j\star^{0}_1\pi_2\psi_{k}}|^2)}}}.
}
For any $j,k\in[p]$ and $i_1,i_2\in[n]$, Jensen's inequality gives
\ba{
&P(|\wt{\pi_2\psi_j\star^{0}_1\pi_2\psi_{k}}|^2)(X_{i_1},X_{i_2})\\
&\leq\frac{1}{6}\int\pi_2\psi_j(y,X_{i_1})^2\pi_2\psi_k(y,X_{i_2})^2P(dy)
+\frac{1}{6}\int\pi_2\psi_j(y,X_{i_2})^2\pi_2\psi_k(y,X_{i_1})^2P(dy)\\
&\quad+\frac{1}{6}\int\pi_2\psi_j(X_{i_1},y)^2\pi_2\psi_k(X_{i_1},X_{i_2})^2P(dy)
+\frac{1}{6}\int\pi_2\psi_j(X_{i_1},X_{i_2})^2\pi_2\psi_k(X_{i_1},y)^2P(dy)\\
&\quad+\frac{1}{6}\int\pi_2\psi_j(X_{i_2},y)^2\pi_2\psi_k(X_{i_2},X_{i_1})^2P(dy)
+\frac{1}{6}\int\pi_2\psi_j(X_{i_2},X_{i_1})^2\pi_2\psi_k(X_{i_2},y)^2P(dy)\\
&\leq\frac{1}{3}\max_{j,k\in[p]}M\bra{(\pi_2\psi_j)^2\star_1^1(\pi_2\psi_k)^2}
+\frac{2}{3}\max_{j,k\in[p]}M\bra{P(|\pi_2\psi_j|^2)}M(\pi_2\psi_k)^2.
}
By \cref{max-contraction}\ref{contr-b},
\ba{
\max_{j,k\in[p]}M\bra{(\pi_2\psi_j)^2\star_1^1(\pi_2\psi_k)^2}
\leq\max_{j\in[p]}M\bra{P(|\pi_2\psi_j|^4)}.
}
Also, by the AM-GM inequality,
\ba{
n\max_{j,k\in[p]}M\bra{P(|\pi_2\psi_j|^2)}M(\pi_2\psi_k)^2\log^9p
\leq\frac{n^2}{2}\max_{j\in[p]}M\bra{P(|\pi_2\psi_j|^2)}^2\log^8p
+\frac{1}{2}\max_{j\in[p]}M(\pi_2\psi_k)^4\log^{10}p.
}
Consequently,
\ban{
(I_2\log^2p)^2
&\leq n\E\sbra{\max_{j\in[p]}M\bra{P(|\pi_2\psi_j|^4)}}\log^9p
+n^2\E\sbra{\max_{j\in[p]}M\bra{P(|\pi_2\psi_j|^2)}^2}\log^8p\notag\\
&\quad+\E\sbra{\max_{j\in[p]}M(\pi_2\psi_k)^4\log^{10}p}\notag\\
&\leq\bra{\Delta_{2,*}^{(3)}(2)+\Delta_{2,*}^{(5)}(2)+\Delta_{2,*}^{(4)}(2)}\log^5p.\label{I2-bound}
}

\paragraph{Step 3.} 
Let us bound
\ba{
II_1=n^{\frac{3}{2}}(\log p)^{\frac{3}{2}}\sqrt{\E\sbra{\max_{j,k\in[p]}M\bra{P(|\wt{\pi_2\psi_j\star^{1}_1\pi_2\psi_{k}}|^2)}}}.
}
For any $j,k\in[p]$ and $i\in[n]$, Jensen's inequality gives
\ba{
P(|\wt{\pi_2\psi_j\star^{1}_1\pi_2\psi_{k}}|^2)(X_i)
\leq\frac{1}{2}\int\pi_2\psi_j\star^{1}_1\pi_2\psi_{k}(u,X_i)^2P(du)
+\frac{1}{2}\int\pi_2\psi_k\star^{1}_1\pi_2\psi_{j}(u,X_i)^2P(du).
}
Hence, by \cref{max-contraction}\ref{contr-d},
\ba{
P(|\wt{\pi_2\psi_j\star^{1}_1\pi_2\psi_{k}}|^2)(X_i)
\leq\max_{j,k\in[p]}\|\pi_2\psi_j\star_1^1\pi_2\psi_j\|_{L^2(P^2)}M\bra{P(|\pi_2\psi_k|^2)}.
}
Therefore, the AM-GM inequality gives
\ban{
II_1\log^2p
&\leq\frac{n^2}{2}\max_{j\in[p]}\|\pi_2\psi_j\star_1^1\pi_2\psi_j\|_{L^2(P^2)}\log^3 p
+\frac{n}{2}\sqrt{\E\sbra{\max_{j\in[p]}M\bra{P(|\pi_2\psi_j|^2)}^2}}\log^4 p\notag\\
&\leq \Delta_1^{(0)}\log^3 p
+\sqrt{\Delta_{2,*}^{(5)}(2)\log^5p}.\label{II1-bound}
}

\paragraph{Step 4.} 
Let us bound
\ba{
III_1=n^{\frac{1}{2}}(\log p)^{\frac{3}{2}}\sqrt{\E\sbra{\max_{j,k\in[p]}M\bra{P(|\wt{\pi_2\psi_j\star^{0}_2\pi_2\psi_{k}}|^2)}}}.
}
Observe that $\pi_2\psi_j\star^{0}_2\pi_2\psi_{k}=\pi_2\psi_j\pi_2\psi_k$ for any $j,k\in[p]$. Hence the Schwarz inequality gives
$
P(|\wt{\pi_2\psi_j\star^{0}_2\pi_2\psi_{k}}|^2)
=P(|\pi_2\psi_j\pi_2\psi_k|^2)
\leq \sqrt{P(|\pi_2\psi_j|^4)P(|\pi_2\psi_k|^4)}.
$ Consequently, 
\ben{\label{III1-bound}
III_1\log^2p\leq n^{\frac{1}{2}}\sqrt{\E\sbra{\max_{j,k\in[p]}M\bra{P(|\pi_2\psi_j|^4)}}}\log^{\frac{7}{2}}p\leq\sqrt{\Delta_{2,*}^{(3)}(2)\log^5p}.
}

Combining \eqref{I3-III2-bound}--\eqref{III1-bound} gives \eqref{lem:delta-eq22}.
\end{proof}

\begin{proof}[Proof of \eqref{lem:delta-eq12}]
Observe that
\ba{
\Delta_1(1,2)
&=\max_{0\leq u\leq 2}\Delta_1(1,2;1,0,u)
+\max_{0\leq u\leq1}\Delta_1(1,2;1,1,u)\\
&=:\max_{0\leq u\leq2}\mathbf I_u+\max_{0\leq u\leq1}\mathbf{II}_u.
}
By \cref{aux-delta1},
\ban{
\mathbf{I}_2\log^2p
&\leq n^2(\log p)^3\sqrt{\E\sbra{\max_{j\in[p]}M(\pi_1\psi_j)^4}}
+(\log p)^5\sqrt{\E\sbra{\max_{j\in[p]}M(\pi_2\psi_j)^4}}\notag\\
&\leq \sqrt{\bra{\Delta_{2,*}^{(2)}(1)+\Delta_{2,*}^{(4)}(2)}\log^5p}.\label{I2-bound12}
}
Meanwhile, \cref{max-contraction}\ref{contr-b} gives $M\bra{\pi_1\psi_j\star^{1}_1\pi_2\psi_{k}}^2\leq\|\pi_1\psi_j\|_{L^2(P)}^2M(P(|\pi_2\psi_k|^2))$. 
Combining this with \eqref{sym-M} yields
\ban{
\mathbf{II}_1\log^2p
&\leq n^2\sqrt{\E\sbra{\max_{j,k\in[p]}\|\pi_1\psi_j\|_{L^2(P)}^2M(P(|\pi_2\psi_k|^2))}}\log^3p\notag\\
&\leq n^{3/2}\max_{j\in[p]}\|\pi_1\psi_j\|_{L^2(P)}\bra{\Delta_{2,*}^{(5)}(2)\log^{9}p}^{1/4},\label{II1-bound12}
}
where we used Lyapunov's inequality in the last line. 
Next, by \eqref{sym-lp},
\ban{
\mathbf I_0&\leq n^{2}(\log p)\max_{j,k\in[p]}\|\pi_1\psi_j\star^{0}_1\pi_2\psi_{k}\|_{L^2(P^2)},\notag\\
\mathbf{II}_0&\leq n^{\frac{5}{2}}(\log p)^{\frac{1}{2}}\max_{j,k\in[p]}\|\pi_1\psi_j\star^{1}_1\pi_2\psi_{k}\|_{L^2(P)}\leq\Delta_1^{(1)}\log^{1/2} p.\label{II0-bound12}
}
For any $j,k\in[p]$, Lemma 2.4(iii) in \cite{DoPe19} and \cref{max-contraction}\ref{contr-c} give
\ba{
\|\pi_1\psi_j\star^{0}_1\pi_2\psi_{k}\|_{L^2(P^2)}^2
&\leq\|\pi_1\psi_j\star^{0}_1\pi_1\psi_{j}\|_{L^2(P)}\|\pi_2\psi_k\star^{1}_2\pi_2\psi_{k}\|_{L^2(P)}\\
&=\|\pi_1\psi_j\|_{L^4(P)}^2\|P(|\pi_2\psi_k|^2)\|_{L^2(P)}.
}
Hence, using the AM-GM inequality, we obtain
\ban{
(\mathbf I_0\log^2p)^2
&\leq \frac{n^{5}}{2}\max_{j\in[p]}\|\pi_1\psi_j\|_{L^4(P)}^4\log^{5}p
+\frac{n^{3}}{2}\max_{j\in[p]}\|P(|\pi_2\psi_j)|^2\|_{L^2(P^2)}^2\log^{7}p\notag\\
&\leq\bra{\Delta_{2,*}^{(1)}(1)+\Delta_{2,*}^{(2)}(2)}\log^5p.\label{I0-bound12}
}
Third, by definition,
\ba{
\mathbf{I}_1=n^{\frac{3}{2}}(\log p)^{\frac{3}{2}}\sqrt{\E\sbra{\max_{j,k\in[p]}M\bra{P(|\wt{\pi_1\psi_j\star^{0}_1\pi_2\psi_{k}}|^2)}}}.
}
For any $j,k\in[p]$ and $i\in[n]$, the Jensen and Schwarz inequalities give
\ba{
&P(|\wt{\pi_1\psi_j\star^{0}_1\pi_2\psi_{k}}|^2)(X_{i})\\
&\leq\frac{1}{2}\int\pi_1\psi_j(X_i)^2\pi_2\psi_k(X_i,v)^2P(dv)
+\frac{1}{2}\int\pi_1\psi_j(y)^2\pi_2\psi_k(y,X_i)^2P(dy)\\
&\leq \frac{1}{2}M(\pi_1\psi_j)^2M\bra{P(|\pi_2\psi_k|^2)}
+\frac{1}{2}\|\pi_1\psi_j\|_{L^4(P)}^2\sqrt{M\bra{P(|\pi_2\psi_k|^4)}}.
}
Hence, by the AM-GM inequality,
\bm{
n^3M\bra{P(|\wt{\pi_1\psi_j\star^{0}_1\pi_2\psi_{k}}|^2)}\log^7p
\leq\frac{n^4}{4}M(\pi_1\psi_j)^4\log^6p+
\frac{n^2}{4}M\bra{P(|\pi_2\psi_k|^2)}^2\log^{8}p\\
+\frac{n^5}{4}\|\pi_1\psi_j\|_{L^4(P)}^4\log^5p
+\frac{n}{4}M\bra{P(|\pi_2\psi_k|^4)}\log^9p.
}
Consequently,
\ben{\label{I1-bound12}
(\mathbf{I}_1\log^2p)^2\leq\bra{\Delta_{2,*}^{(2)}(1)+\Delta_{2,*}^{(5)}(2)+\Delta_{2,*}^{(1)}(1)+\Delta_{2,*}^{(3)}(2)}\log^5p.
}
Combining \eqref{II0-bound12}--\eqref{I1-bound12} gives \eqref{lem:delta-eq12}.
\end{proof}

\addcontentsline{toc}{section}{References}
\bibliography{ref}

\begin{thebibliography}{}

\bibitem[\protect\citeauthoryear{Adamczak}{Adamczak}{2006}]{adamczak2006moment}
Adamczak, R. (2006).
\newblock Moment inequalities for {$U$}-statistics.
\newblock {\em Annals of Probability\/}~{\em 34\/}(6), 2288--2314.

\bibitem[\protect\citeauthoryear{Adams and Fournier}{Adams and Fournier}{2003}]{adams2003sobolev}
Adams, R.~A. and J.~J.~F. Fournier (2003).
\newblock {\em Sobolev spaces\/} (second ed.).
\newblock Elsevier.

\bibitem[\protect\citeauthoryear{Altonji, Ichimura, and Otsu}{Altonji et~al.}{2012}]{altonji2012estimating}
Altonji, J.~G., H.~Ichimura, and T.~Otsu (2012).
\newblock Estimating derivatives in nonseparable models with limited dependent variables.
\newblock {\em Econometrica\/}~{\em 80\/}(4), 1701--1719.

\bibitem[\protect\citeauthoryear{Amemiya}{Amemiya}{1973}]{amemiya1973regression}
Amemiya, T. (1973).
\newblock Regression analysis when the dependent variable is truncated normal.
\newblock {\em Econometrica\/}~{\em 41\/}(6), 997--1016.

\bibitem[\protect\citeauthoryear{Arias-Castro, Pelletier, and Saligrama}{Arias-Castro et~al.}{2018}]{arias2018remember}
Arias-Castro, E., B.~Pelletier, and V.~Saligrama (2018).
\newblock Remember the curse of dimensionality: The case of goodness-of-fit testing in arbitrary dimension.
\newblock {\em Journal of Nonparametric Statistics\/}~{\em 30\/}(2), 448--471.

\bibitem[\protect\citeauthoryear{Armstrong and Koles{\'a}r}{Armstrong and Koles{\'a}r}{2018}]{armstrong2018simple}
Armstrong, T.~B. and M.~Koles{\'a}r (2018).
\newblock A simple adjustment for bandwidth snooping.
\newblock {\em The Review of Economic Studies\/}~{\em 85\/}(2), 732--765.

\bibitem[\protect\citeauthoryear{Belloni, Chernozhukov, Chetverikov, Hansen, and Kato}{Belloni et~al.}{2018}]{belloni2018high}
Belloni, A., V.~Chernozhukov, D.~Chetverikov, C.~Hansen, and K.~Kato (2018).
\newblock High-dimensional econometrics and regularized {GMM}.
\newblock {\em arXiv preprint arXiv:1806.01888\/}.

\bibitem[\protect\citeauthoryear{Campbell}{Campbell}{2011}]{campbell2011competition}
Campbell, J.~R. (2011).
\newblock Competition in large markets.
\newblock {\em Journal of Applied Econometrics\/}~{\em 26\/}(7), 1113--1136.

\bibitem[\protect\citeauthoryear{Cattaneo, Crump, and Jansson}{Cattaneo et~al.}{2010}]{cattaneo2010robust}
Cattaneo, M.~D., R.~K. Crump, and M.~Jansson (2010).
\newblock Robust data-driven inference for density-weighted average derivatives.
\newblock {\em Journal of the American Statistical Association\/}~{\em 105\/}(491), 1070--1083.

\bibitem[\protect\citeauthoryear{Cattaneo, Crump, and Jansson}{Cattaneo et~al.}{2013}]{cattaneo2013generalized}
Cattaneo, M.~D., R.~K. Crump, and M.~Jansson (2013).
\newblock Generalized jackknife estimators of weighted average derivatives.
\newblock {\em Journal of the American Statistical Association\/}~{\em 108\/}(504), 1243--1256.

\bibitem[\protect\citeauthoryear{Cattaneo, Crump, and Jansson}{Cattaneo et~al.}{2014a}]{cattaneo2014bootstrapping}
Cattaneo, M.~D., R.~K. Crump, and M.~Jansson (2014a).
\newblock Bootstrapping density-weighted average derivatives.
\newblock {\em Econometric Theory\/}~{\em 30\/}(6), 1135--1164.

\bibitem[\protect\citeauthoryear{Cattaneo, Crump, and Jansson}{Cattaneo et~al.}{2014b}]{cattaneo2014small}
Cattaneo, M.~D., R.~K. Crump, and M.~Jansson (2014b).
\newblock Small bandwidth asymptotics for density-weighted average derivatives.
\newblock {\em Econometric Theory\/}~{\em 30\/}(1), 176--200.

\bibitem[\protect\citeauthoryear{Cattaneo, Farrell, Jansson, and Masini}{Cattaneo et~al.}{2024}]{cattaneo2024higher}
Cattaneo, M.~D., M.~H. Farrell, M.~Jansson, and R.~P. Masini (2024).
\newblock Higher-order refinements of small bandwidth asymptotics for density-weighted average derivative estimators.
\newblock {\em Journal of Econometrics\/}~{\em In press}, 105855.

\bibitem[\protect\citeauthoryear{Cattaneo and Jansson}{Cattaneo and Jansson}{2018}]{cattaneo2018kernel}
Cattaneo, M.~D. and M.~Jansson (2018).
\newblock Kernel-based semiparametric estimators: Small bandwidth asymptotics and bootstrap consistency.
\newblock {\em Econometrica\/}~{\em 86\/}(3), 955--995.

\bibitem[\protect\citeauthoryear{Cattaneo and Jansson}{Cattaneo and Jansson}{2022}]{cattaneo2022average}
Cattaneo, M.~D. and M.~Jansson (2022).
\newblock Average density estimators: Efficiency and bootstrap consistency.
\newblock {\em Econometric Theory\/}~{\em 38\/}(6), 1140--1174.

\bibitem[\protect\citeauthoryear{Cattaneo, Jansson, and Newey}{Cattaneo et~al.}{2018a}]{cattaneo2018alternative}
Cattaneo, M.~D., M.~Jansson, and W.~K. Newey (2018a).
\newblock Alternative asymptotics and the partially linear model with many regressors.
\newblock {\em Econometric Theory\/}~{\em 34\/}(2), 277--301.

\bibitem[\protect\citeauthoryear{Cattaneo, Jansson, and Newey}{Cattaneo et~al.}{2018b}]{cattaneo2018inference}
Cattaneo, M.~D., M.~Jansson, and W.~K. Newey (2018b).
\newblock Inference in linear regression models with many covariates and heteroscedasticity.
\newblock {\em Journal of the American Statistical Association\/}~{\em 113\/}(523), 1350--1361.

\bibitem[\protect\citeauthoryear{Chakrabortty and Kuchibhotla}{Chakrabortty and Kuchibhotla}{2025}]{ChKu25}
Chakrabortty, A. and A.~K. Kuchibhotla (2025).
\newblock Tail bounds for canonical {$U$}-statistics and {$U$}-processes with unbounded kernels.
\newblock {\em arXiv preprint arXiv:2504.01318\/}.

\bibitem[\protect\citeauthoryear{Chao, Swanson, Hausman, Newey, and Woutersen}{Chao et~al.}{2012}]{chao2012asymptotic}
Chao, J.~C., N.~R. Swanson, J.~A. Hausman, W.~K. Newey, and T.~Woutersen (2012).
\newblock Asymptotic distribution of {JIVE} in a heteroskedastic {IV} regression with many instruments.
\newblock {\em Econometric Theory\/}~{\em 28\/}(1), 42--86.

\bibitem[\protect\citeauthoryear{Chen}{Chen}{2018}]{chen2018gaussian}
Chen, X. (2018).
\newblock Gaussian and bootstrap approximations for high-dimensional {U}-statistics and their applications.
\newblock {\em Annals of Statistics\/}~{\em 46\/}(2), 642--678.

\bibitem[\protect\citeauthoryear{Chen and Kato}{Chen and Kato}{2019}]{ChKa19}
Chen, X. and K.~Kato (2019).
\newblock Randomized incomplete {$U$}-statistics in high dimensions.
\newblock {\em Annals of Statistics\/}~{\em 47\/}(6), 3127--3156.

\bibitem[\protect\citeauthoryear{Chen and Kato}{Chen and Kato}{2020}]{ChKa20}
Chen, X. and K.~Kato (2020).
\newblock Jackknife multiplier bootstrap: finite sample approximations to the {$U$}-process supremum with applications.
\newblock {\em Probability Theory and Related Fields\/}~{\em 176}, 1097--1163.

\bibitem[\protect\citeauthoryear{Cheng, Liu, and Peng}{Cheng et~al.}{2022}]{cheng2022gaussian}
Cheng, G., Z.~Liu, and L.~Peng (2022).
\newblock Gaussian approximations for high-dimensional non-degenerate {$U$}-statistics via exchangeable pairs.
\newblock {\em Statistics \& Probability Letters\/}~{\em 182}, 109295.

\bibitem[\protect\citeauthoryear{Chernozhukov, Chetverikov, and Kato}{Chernozhukov et~al.}{2013}]{CCK13}
Chernozhukov, V., D.~Chetverikov, and K.~Kato (2013).
\newblock Gaussian approximations and multiplier bootstrap for maxima of sums of high-dimensional random vectors.
\newblock {\em Annals of Statistics\/}~{\em 41\/}(6), 2786--2819.

\bibitem[\protect\citeauthoryear{Chernozhukov, Chetverikov, and Kato}{Chernozhukov et~al.}{2015}]{CCK15}
Chernozhukov, V., D.~Chetverikov, and K.~Kato (2015).
\newblock Comparison and anti-concentration bounds for maxima of {G}aussian random vectors.
\newblock {\em Probability Theory and Related Fields\/}~{\em 162}, 47--70.

\bibitem[\protect\citeauthoryear{Chernozhukov, Chetverikov, and Kato}{Chernozhukov et~al.}{2017}]{CCK17}
Chernozhukov, V., D.~Chetverikov, and K.~Kato (2017).
\newblock Central limit theorems and bootstrap in high dimensions.
\newblock {\em Annals of Probability\/}~{\em 45\/}(4), 2309--2353.

\bibitem[\protect\citeauthoryear{Chernozhukov, Chetverikov, and Kato}{Chernozhukov et~al.}{2019}]{chernozhukov2019inference}
Chernozhukov, V., D.~Chetverikov, and K.~Kato (2019).
\newblock Inference on causal and structural parameters using many moment inequalities.
\newblock {\em The Review of Economic Studies\/}~{\em 86\/}(5), 1867--1900.

\bibitem[\protect\citeauthoryear{Chernozhukov, Chetverikov, Kato, and Koike}{Chernozhukov et~al.}{2022}]{CCKK22}
Chernozhukov, V., D.~Chetverikov, K.~Kato, and Y.~Koike (2022).
\newblock Improved central limit theorem and bootstrap approximation in high dimensions.
\newblock {\em Annals of Statistics\/}~{\em 50\/}(5), 2562--2586.

\bibitem[\protect\citeauthoryear{Chernozhukov, Chetverikov, Kato, and Koike}{Chernozhukov et~al.}{2023}]{chernozhukov2023high}
Chernozhukov, V., D.~Chetverikov, K.~Kato, and Y.~Koike (2023).
\newblock High-dimensional data bootstrap.
\newblock {\em Annual Review of Statistics and Its Application\/}~{\em 10\/}(1), 427--449.

\bibitem[\protect\citeauthoryear{Chetverikov, Wilhelm, and Kim}{Chetverikov et~al.}{2021}]{chetverikov2021adaptive}
Chetverikov, D., D.~Wilhelm, and D.~Kim (2021).
\newblock An adaptive test of stochastic monotonicity.
\newblock {\em Econometric Theory\/}~{\em 37\/}(3), 495--536.

\bibitem[\protect\citeauthoryear{Chiang, Kato, and Sasaki}{Chiang et~al.}{2023}]{chiang2023inference}
Chiang, H.~D., K.~Kato, and Y.~Sasaki (2023).
\newblock Inference for high-dimensional exchangeable arrays.
\newblock {\em Journal of the American Statistical Association\/}~{\em 118\/}(543), 1595--1605.

\bibitem[\protect\citeauthoryear{Coppejans and Sieg}{Coppejans and Sieg}{2005}]{coppejans2005kernel}
Coppejans, M. and H.~Sieg (2005).
\newblock Kernel estimation of average derivatives and differences.
\newblock {\em Journal of Business \& Economic Statistics\/}~{\em 23\/}(2), 211--225.

\bibitem[\protect\citeauthoryear{Dalalyan and Yoshida}{Dalalyan and Yoshida}{2011}]{dalalyan2011second}
Dalalyan, A. and N.~Yoshida (2011).
\newblock Second-order asymptotic expansion for a non-synchronous covariation estimator.
\newblock {\em Annales de l'Institut Henri Poincar{\'e}-Probabilit{\'e}s et Statistiques\/}~{\em 47\/}(3), 748--789.

\bibitem[\protect\citeauthoryear{de~la Pe{\~n}a and Gin{\'e}}{de~la Pe{\~n}a and Gin{\'e}}{1999}]{de1999decoupling}
de~la Pe{\~n}a, V. and E.~Gin{\'e} (1999).
\newblock {\em Decoupling: From Dependence to Independence}.
\newblock Springer Science \& Business Media.

\bibitem[\protect\citeauthoryear{Deaton and Ng}{Deaton and Ng}{1998}]{deaton1998parametric}
Deaton, A. and S.~Ng (1998).
\newblock Parametric and nonparametric approaches to price and tax reform.
\newblock {\em Journal of the American Statistical Association\/}~{\em 93\/}(443), 900--909.

\bibitem[\protect\citeauthoryear{Di~Nezza, Palatucci, and Valdinoci}{Di~Nezza et~al.}{2012}]{di2012hitchhiker}
Di~Nezza, E., G.~Palatucci, and E.~Valdinoci (2012).
\newblock Hitchhiker's guide to the fractional {S}obolev spaces.
\newblock {\em Bulletin des Sciences Math{\'e}matiques\/}~{\em 136\/}(5), 521--573.

\bibitem[\protect\citeauthoryear{D{\"o}bler}{D{\"o}bler}{2023}]{dobler2023normal}
D{\"o}bler, C. (2023).
\newblock Normal approximation via non-linear exchangeable pairs.
\newblock {\em ALEA\/}~{\em 20}, 167--224.

\bibitem[\protect\citeauthoryear{D{\"o}bler, Kasprzak, and Peccati}{D{\"o}bler et~al.}{2022}]{dobler2022functional}
D{\"o}bler, C., M.~J. Kasprzak, and G.~Peccati (2022).
\newblock Functional convergence of sequential {$U$}-processes with size-dependent kernels.
\newblock {\em Annals of Applied Probability\/}~{\em 32\/}(1), 551--601.

\bibitem[\protect\citeauthoryear{D{\"o}bler and Peccati}{D{\"o}bler and Peccati}{2017}]{DoPe17}
D{\"o}bler, C. and G.~Peccati (2017).
\newblock Quantitative de {J}ong theorems in any dimension.
\newblock {\em Electronic Journal of Probability\/}~{\em 22\/}(2), 1--35.

\bibitem[\protect\citeauthoryear{D{\"o}bler and Peccati}{D{\"o}bler and Peccati}{2019}]{DoPe19}
D{\"o}bler, C. and G.~Peccati (2019).
\newblock Quantitative {CLT}s for symmetric {$U$}-statistics using contractions.
\newblock {\em Electronic Journal of Probability\/}~{\em 24\/}(5), 1--43.

\bibitem[\protect\citeauthoryear{Dong and Sasaki}{Dong and Sasaki}{2022}]{dong2022estimation}
Dong, H. and Y.~Sasaki (2022).
\newblock Estimation of average derivatives of latent regressors: with an application to inference on buffer-stock saving.
\newblock {\em arXiv preprint arXiv:2209.05914\/}.

\bibitem[\protect\citeauthoryear{Fan and Li}{Fan and Li}{1996}]{fan1996consistent}
Fan, Y. and Q.~Li (1996).
\newblock Consistent model specification tests: Omitted variables and semiparametric functional forms.
\newblock {\em Econometrica\/}~{\em 64\/}(4), 865--90.

\bibitem[\protect\citeauthoryear{Fang and Koike}{Fang and Koike}{2021}]{fang2021high}
Fang, X. and Y.~Koike (2021).
\newblock High-dimensional central limit theorems by {S}tein's method.
\newblock {\em Annals of Applied Probability\/}~{\em 31\/}(4), 1660--1686.

\bibitem[\protect\citeauthoryear{Fang and Koike}{Fang and Koike}{2023}]{fang2023p}
Fang, X. and Y.~Koike (2023).
\newblock From $p$-{W}asserstein bounds to moderate deviations.
\newblock {\em Electronic Journal of Probability\/}~{\em 28}, 1--52.

\bibitem[\protect\citeauthoryear{Gin{\'e}, Lata{\l}a, and Zinn}{Gin{\'e} et~al.}{2000}]{gine2000exponential}
Gin{\'e}, E., R.~Lata{\l}a, and J.~Zinn (2000).
\newblock Exponential and moment inequalities for {$U$}-statistics.
\newblock In {\em High Dimensional Probability II}, pp.\  13--38. Springer.

\bibitem[\protect\citeauthoryear{Gin{\'e} and Nickl}{Gin{\'e} and Nickl}{2008}]{gine2008simple}
Gin{\'e}, E. and R.~Nickl (2008).
\newblock A simple adaptive estimator of the integrated square of a density.
\newblock {\em Bernoulli\/}~{\em 14\/}(1), 47--61.

\bibitem[\protect\citeauthoryear{Gin{\'e} and Nickl}{Gin{\'e} and Nickl}{2016}]{gine2016mathematical}
Gin{\'e}, E. and R.~Nickl (2016).
\newblock {\em Mathematical Foundations of Infinite-Dimensional Statistical Models}.
\newblock Cambridge University Press.

\bibitem[\protect\citeauthoryear{H{\"a}rdle, Hildenbrand, and Jerison}{H{\"a}rdle et~al.}{1991}]{hardle1991empirical}
H{\"a}rdle, W., W.~Hildenbrand, and M.~Jerison (1991).
\newblock Empirical evidence on the law of demand.
\newblock {\em Econometrica: Journal of the Econometric Society\/}~{\em 59\/}(6), 1525--1549.

\bibitem[\protect\citeauthoryear{H{\"a}rdle and Mammen}{H{\"a}rdle and Mammen}{1993}]{hardle1993comparing}
H{\"a}rdle, W. and E.~Mammen (1993).
\newblock Comparing nonparametric versus parametric regression fits.
\newblock {\em Annals of Statistics\/}~{\em 21\/}(4), 1926--1947.

\bibitem[\protect\citeauthoryear{Hitczenko}{Hitczenko}{1990}]{Hi90}
Hitczenko, P. (1990).
\newblock Best constants in martingale version of {R}osenthal's inequality.
\newblock {\em Annals of Probability\/}~{\em 18\/}(4), 1656--1668.

\bibitem[\protect\citeauthoryear{Horowitz and Spokoiny}{Horowitz and Spokoiny}{2001}]{horowitz2001adaptive}
Horowitz, J.~L. and V.~G. Spokoiny (2001).
\newblock An adaptive, rate-optimal test of a parametric mean-regression model against a nonparametric alternative.
\newblock {\em Econometrica\/}~{\em 69\/}(3), 599--631.

\bibitem[\protect\citeauthoryear{Hyt{\"o}nen, {v}an Neerven, Veraar, and Weis}{Hyt{\"o}nen et~al.}{2016}]{HvNVW16}
Hyt{\"o}nen, T., J.~{v}an Neerven, M.~Veraar, and L.~Weis (2016).
\newblock {\em Analysis in {B}anach Spaces}, Volume I: Martingales and {L}ittlewood-{P}aley Theory.
\newblock Springer.

\bibitem[\protect\citeauthoryear{Hyt{\"o}nen, {v}an Neerven, Veraar, and Weis}{Hyt{\"o}nen et~al.}{2017}]{HvNVW17}
Hyt{\"o}nen, T., J.~{v}an Neerven, M.~Veraar, and L.~Weis (2017).
\newblock {\em Analysis in {B}anach Spaces}, Volume II: Probabilistic Methods and Operator Theory.
\newblock Springer.

\bibitem[\protect\citeauthoryear{Ibragimov and Sharakhmetov}{Ibragimov and Sharakhmetov}{2002}]{IbSh02}
Ibragimov, R. and S.~Sharakhmetov (2002).
\newblock Bounds on moments of symmetric statistics.
\newblock {\em Studia Scientiarum Mathematicarum Hungarica\/}~{\em 39\/}(3-4), 251--275.

\bibitem[\protect\citeauthoryear{Ingster}{Ingster}{2000}]{ingster2000adaptive}
Ingster, Y.~I. (2000).
\newblock Adaptive chi-square tests.
\newblock {\em Journal of Mathematical Sciences\/}~{\em 99}, 1110--1119.

\bibitem[\protect\citeauthoryear{Jiang, Li, and Zhang}{Jiang et~al.}{2025}]{jiang2025online}
Jiang, J., X.~Li, and J.~Zhang (2025).
\newblock Online stochastic optimization with wasserstein-based nonstationarity.
\newblock {\em Management Science\/}~{\em 71\/}(4).

\bibitem[\protect\citeauthoryear{Koike}{Koike}{2019a}]{koike2019gaussian}
Koike, Y. (2019a).
\newblock Gaussian approximation of maxima of {W}iener functionals and its application to high-frequency data.
\newblock {\em Annals of Statistics\/}~{\em 47\/}(3), 1663--1687.

\bibitem[\protect\citeauthoryear{Koike}{Koike}{2019b}]{koike2019mixed}
Koike, Y. (2019b).
\newblock Mixed-normal limit theorems for multiple {S}korohod integrals in high-dimensions, with application to realized covariance.
\newblock {\em Electronic Journal of Statistics\/}~{\em 13}, 1443--1522.

\bibitem[\protect\citeauthoryear{Koike}{Koike}{2023}]{koike2023high}
Koike, Y. (2023).
\newblock High-dimensional central limit theorems for homogeneous sums.
\newblock {\em Journal of Theoretical Probability\/}~{\em 36\/}(1), 1--45.

\bibitem[\protect\citeauthoryear{Kontorovich}{Kontorovich}{2023}]{kontorovich2023decoupling}
Kontorovich, A. (2023).
\newblock Decoupling maximal inequalities.
\newblock {\em arXiv preprint arXiv:2302.14150\/}.

\bibitem[\protect\citeauthoryear{Li and Yuan}{Li and Yuan}{2024}]{li2024optimality}
Li, T. and M.~Yuan (2024).
\newblock On the optimality of {G}aussian kernel based nonparametric tests against smooth alternatives.
\newblock {\em Journal of Machine Learning Research\/}~{\em 25\/}(334), 1--62.

\bibitem[\protect\citeauthoryear{Meckes}{Meckes}{2009}]{meckes2009stein}
Meckes, E. (2009).
\newblock On {S}tein's method for multivariate normal approximation.
\newblock In {\em High dimensional probability V: the Luminy volume}, Volume~5, pp.\  153--179. Institute of Mathematical Statistics.

\bibitem[\protect\citeauthoryear{Meinshausen, Maathuis, and B{\"u}hlmann}{Meinshausen et~al.}{2011}]{meinshausen2011asymptotic}
Meinshausen, N., M.~H. Maathuis, and P.~B{\"u}hlmann (2011).
\newblock Asymptotic optimality of the westfall-young permutation procedure for multiple testing under dependence.
\newblock {\em The Annals of Statistics\/}, 3369--3391.

\bibitem[\protect\citeauthoryear{Nishiyama and Robinson}{Nishiyama and Robinson}{2000}]{NiRo00}
Nishiyama, Y. and P.~M. Robinson (2000).
\newblock Edgeworth expansions for semiparametric averaged derivatives.
\newblock {\em Econometrica\/}~{\em 68\/}(4), 931--979.

\bibitem[\protect\citeauthoryear{Nourdin, Peccati, and Swan}{Nourdin et~al.}{2014}]{NPS14}
Nourdin, I., G.~Peccati, and Y.~Swan (2014).
\newblock Entropy and the fourth moment phenomenon.
\newblock {\em Journal of Functional Analysis\/}~{\em 266}, 3170--3207.

\bibitem[\protect\citeauthoryear{Pinelis}{Pinelis}{1994}]{Pi94}
Pinelis, I. (1994).
\newblock Optimum bounds for the distributions of martingales in {B}anach spaces.
\newblock {\em Annals of Probability\/}~{\em 22\/}(4), 1679--1706.

\bibitem[\protect\citeauthoryear{Pisier}{Pisier}{2016}]{Pi16}
Pisier, G. (2016).
\newblock {\em Martingales in {B}anach Spaces}.
\newblock Cambridge University Press.

\bibitem[\protect\citeauthoryear{Powell, Stock, and Stoker}{Powell et~al.}{1989}]{powell1989semiparametric}
Powell, J.~L., J.~H. Stock, and T.~M. Stoker (1989).
\newblock Semiparametric estimation of index coefficients.
\newblock {\em Econometrica\/}~{\em 57\/}(6), 1403--1430.

\bibitem[\protect\citeauthoryear{Robinson}{Robinson}{1995}]{Ro95}
Robinson, P.~M. (1995).
\newblock The normal approximation for semiparametric averaged derivatives.
\newblock {\em Econometrica\/}~{\em 63\/}(3), 667--680.

\bibitem[\protect\citeauthoryear{Romano and Wolf}{Romano and Wolf}{2005}]{romano2005exact}
Romano, J.~P. and M.~Wolf (2005).
\newblock Exact and approximate stepdown methods for multiple hypothesis testing.
\newblock {\em Journal of the American Statistical Association\/}~{\em 100\/}(469), 94--108.

\bibitem[\protect\citeauthoryear{Rudin}{Rudin}{1991}]{Rudin1991}
Rudin, W. (1991).
\newblock {\em Functional Analysis\/} (second ed.).
\newblock McGraw-Hill.

\bibitem[\protect\citeauthoryear{Schrab, Kim, Albert, Laurent, Guedj, and Gretton}{Schrab et~al.}{2023}]{schrab2023mmd}
Schrab, A., I.~Kim, M.~Albert, B.~Laurent, B.~Guedj, and A.~Gretton (2023).
\newblock {MMD} aggregated two-sample test.
\newblock {\em Journal of Machine Learning Research\/}~{\em 24\/}(194), 1--81.

\bibitem[\protect\citeauthoryear{Song, Chen, and Kato}{Song et~al.}{2019}]{song2019approximating}
Song, Y., X.~Chen, and K.~Kato (2019).
\newblock Approximating high-dimensional infinite-order {$ U $}-statistics: Statistical and computational guarantees.
\newblock {\em Electronic Journal of Statistics\/}~{\em 13\/}(2), 4794.

\bibitem[\protect\citeauthoryear{Song, Chen, and Kato}{Song et~al.}{2023}]{song2023stratified}
Song, Y., X.~Chen, and K.~Kato (2023).
\newblock Stratified incomplete local simplex tests for curvature of nonparametric multiple regression.
\newblock {\em Bernoulli\/}~{\em 29\/}(1), 323--349.

\bibitem[\protect\citeauthoryear{Sriperumbudur, Gretton, Fukumizu, Sch{\"o}lkopf, and Lanckriet}{Sriperumbudur et~al.}{2010}]{sriperumbudur2010hilbert}
Sriperumbudur, B.~K., A.~Gretton, K.~Fukumizu, B.~Sch{\"o}lkopf, and G.~R. Lanckriet (2010).
\newblock Hilbert space embeddings and metrics on probability measures.
\newblock {\em Journal of Machine Learning Research\/}~{\em 11}, 1517--1561.

\bibitem[\protect\citeauthoryear{Stoker}{Stoker}{1989}]{stoker1989tests}
Stoker, T.~M. (1989).
\newblock Tests of additive derivative constraints.
\newblock {\em The Review of Economic Studies\/}~{\em 56\/}(4), 535--552.

\bibitem[\protect\citeauthoryear{Tobin}{Tobin}{1958}]{tobin1958corner}
Tobin, J. (1958).
\newblock Estimation of relationships for limited dependent variables.
\newblock {\em Econometrica\/}~{\em 26\/}(1), 24--36.

\bibitem[\protect\citeauthoryear{van~der Vaart and Wellner}{van~der Vaart and Wellner}{1996}]{van1996weak}
van~der Vaart, A. and J.~A. Wellner (1996).
\newblock {\em Weak Convergence and Empirical Processes: With Applications to Statistics}.
\newblock Springer.

\bibitem[\protect\citeauthoryear{{v}an Neerven and Veraar}{{v}an Neerven and Veraar}{2022}]{van2021maximal}
{v}an Neerven, J. and M.~Veraar (2022).
\newblock Maximal inequalities for stochastic convolutions and pathwise uniform convergence of time discretisation schemes.
\newblock {\em Stochastics and Partial Differential Equations: Analysis and Computations\/}~{\em 10\/}(2), 516--581.

\bibitem[\protect\citeauthoryear{Westfall and Young}{Westfall and Young}{1993}]{westfall1993resampling}
Westfall, P.~H. and S.~S. Young (1993).
\newblock {\em Resampling-based multiple testing: Examples and methods for p-value adjustment}.
\newblock John Wiley \& Sons.

\bibitem[\protect\citeauthoryear{Zhang}{Zhang}{2022}]{zhang2022berry}
Zhang, Z.-S. (2022).
\newblock {B}erry--{E}sseen bounds for generalized {$U$}-statistics.
\newblock {\em Electronic Journal of Probability\/}~{\em 27}, 1--36.

\bibitem[\protect\citeauthoryear{Zheng}{Zheng}{1996}]{zheng1996consistent}
Zheng, J.~X. (1996).
\newblock A consistent test of functional form via nonparametric estimation techniques.
\newblock {\em Journal of Econometrics\/}~{\em 75\/}(2), 263--289.

\end{thebibliography}
\bibliographystyle{chicago}
\end{document}